\documentclass[aos,preprint]{imsart}

\RequirePackage[OT1]{fontenc}
\RequirePackage{amsthm,amsmath,graphicx}
\RequirePackage[sort,numbers]{natbib}
\RequirePackage[colorlinks,citecolor=red,urlcolor=blue]{hyperref}

\usepackage{amsfonts,amssymb}
\usepackage[ruled,vlined]{algorithm2e}
\usepackage[usenames,dvipsnames]{xcolor}

\usepackage{mathtools}

\startlocaldefs

\theoremstyle{plain}
\newtheorem{lemma}{Lemma}[section]
\newtheorem{proposition}{Proposition}[section]
\newtheorem{theorem}{Theorem}[section]

\newcommand{\p} {\mathbb{P}}
\newcommand{\E} {\mathbb{E}}

\newcommand{\norm}[1]{\left\|{#1} \right\|}

\newcommand{\fnorm}[1]{\left\|{#1} \right\|_\text{F}}

\newcommand{\opnorm}[1]{\left\|{#1} \right\|}

\DeclareMathOperator*{\argmin}{arg\,min}

\let\oldenumerate\enumerate
\renewcommand{\enumerate}{
  \oldenumerate
  \setlength{\itemsep}{1pt}
  \setlength{\parskip}{0pt}
  \setlength{\parsep}{0pt}
}

\newcommand{\iid}{\stackrel{iid}{\sim}}
\newcommand{\mathr}{\mathbb{R}}
\newcommand{\mathn}{\mathcal{N}}
\newcommand{\br}[1]{\left( #1 \right)}

\newcommand{\cbr}[1]{\left\{ #1 \right\}}
\newcommand{\abs}[1]{\left| #1 \right|}
\newcommand{\pbr}[1]{\p\left( #1 \right)}
\newcommand{\ebr}[1]{\exp\left( #1 \right)}
\newcommand{\indic}[1]{{\mathbb{I}\left\{{#1}\right\}}}
\newcommand{\iprod}[2]{\left \langle #1, #2 \right\rangle}

\newcommand{\tspan}[1]{\text{span}\br{#1}}
\newcommand{\dist}{\stackrel{d}{=}}

\newcommand{\cd}[1]{{\cdot , #1}}
\newcommand{\cdi}[1]{{ #1, \cdot }}
\newcommand{\tpo}{{\br{1}}}
\newcommand{\tpt}{{\br{2}}}

\endlocaldefs

\begin{document}
\begin{frontmatter}
\title{Optimality of Spectral Clustering 
in the Gaussian Mixture Model}
\runtitle{Optimality of Spectral Clustering in the GMM}

\begin{aug}
\author{\fnms{Matthias} \snm{L\"offler}\thanksref{a}
	\ead[label=e0]{matthias.loeffler@stat.math.ethz.ch}\ead[label=u1,url]{https://people.math.ethz.ch/\textasciitilde mloeffler/}}, 
\author{\fnms{Anderson Y.} \snm{Zhang}\ead[label=e1]{ayz@wharton.upenn.edu}},
\and
\author{\fnms{Harrison H.} \snm{Zhou}\ead[label=e2]{huibin.zhou@yale.edu}\ead[label=u2,url]{http://www.stat.yale.edu/\textasciitilde hz68/}}
\runauthor{L\"offler, Zhang, Zhou}

\affiliation{ETH Z\"urich, 
	 University of Pennsylvania, and Yale University}

\thankstext{a}{M. L\"offler gratefully acknowledges financial support of ERC grant UQMSI/647812 and EPSRC grant EP/L016516/1, which funded a research visit to Yale University, where parts of this work were completed. These grants also funded M. L\"offler during his PhD studies at the University of Cambridge.}

\address[a]{Seminar for Statistics \\
	ETH Z\"urich \\
	R\"amistrasse 101 \\
	8092 Z\"urich Switzerland \\
			\printead{e0}\\
			\printead{u1}}
			
\address{Department of Statistics\\
The Wharton School\\
University of Pennsylvania\\
Philadelphia, PA 19104\\
\printead{e1}}

\address{Department of Statistics\\
Yale University\\
New Haven, CT 06511 \\
\printead{e2}\\
\printead{u2}}
\end{aug}

\begin{abstract}
Spectral clustering is one of the most popular algorithms to group high dimensional data. It is easy to implement and computationally efficient. Despite its popularity and successful applications, its theoretical properties have not been fully understood. 
In this paper, we show that spectral clustering is minimax optimal in the  Gaussian Mixture Model with isotropic covariance matrix,  when the number of clusters is fixed and the signal-to-noise ratio is large enough. 
Spectral gap conditions are widely assumed in the literature to analyze spectral clustering. On the contrary, these conditions are not needed to establish optimality of spectral clustering in this paper. 
\end{abstract}



\end{frontmatter}


\section{Introduction}
Clustering is a central and fundamental problem in statistics and machine learning. One popular approach to clustering of high-dimensional data is to use a spectral method \citep{spielman1996spectral,von2007tutorial}.  It tracks back to \citep{hall1970r, fiedler1973algebraic} and has enjoyed tremendous success. 
In computer science and machine learning, spectral clustering and its variants have been widely used to solve many different problems, including parallel computation \citep{van1995improved,simon1991partitioning,hendrickson1995improved}, graph partitioning  \citep{donath2003lower,guattery1998quality,ding2001min,chaudhuri2012spectral,coja2010graph,mcsherry2001spectral,qin2013regularized,vu2018simple},
and explanatory data mining and statistical data analysis \citep{alpert1995spectral, kannan2004clusterings,ng2002spectral,belkin2003laplacian}. It also has many real data applications, including image segmentation \citep{shi2000normalized, stella2003multiclass, meila2001learning}, text mining \citep{dhillon2001co, pan2010cross, ding2005equivalence}, speech separation \cite{bach2006learning,furui1989unsupervised},
and many others.  In recent years, spectral clustering has also been one of the most favored and studied methods for community detection \citep{lei2015consistency, rohe2011spectral, sarkar2015role, Jin15,balakrishnan2011noise, fishkind2013consistent,anandkumar2014tensor}. 

Spectral clustering is  easy to implement and has remarkably good performance. The idea behind spectral clustering is dimensionality reduction. First it performs a spectral decomposition on the dataset, or some related distance matrix, and only keeps the leading few spectral components. This way the dimensionality of the data is greatly reduced. 
Then a standard clustering method (for example, the $k$-means algorithm) is performed on the low dimensional denoised data to obtain an estimate of the cluster assignments. Due to the dimensionality reduction, spectral clustering is computationally less demanding than many other classical clustering algorithms.

In spite of its popularity, 
the theoretical properties of spectral clustering are not fully understood. 
One line of  theoretical investigation of spectral clustering is to consider the performance under general conditions when applied to eigenvectors of the graph Laplacian. For instance, \citep{hein2006uniform, hein2005graphs, gine2006empirical, von2008consistency, belkin2003laplacian} provide various forms of asymptotic convergence guarantees for the  graph  Laplacian, related spectral properties and spectral clustering. 
Another approach is to consider the performance of spectral clustering in a specific statistical model. Particularly,  spectral clustering for community detection in the stochastic blockmodel has been investigated frequently. \citep{lei2015consistency, Jin15, qin2013regularized, rohe2011spectral, zhou2019analysis} show that  spectral clustering applied to the adjacency matrix of the network can consistently recover hidden community structure. However, their upper bounds on the number of nodes incorrectly clustered are polynomial in the signal-to-noise ratio,  whereas the optimal rate of community detection  is exponential in the signal-to-noise ratio \cite{ZhangZhou16}. Therefore, in the literature spectral clustering is often used as a way to initialize (i.e., `warm start' ) iterative algorithms which eventually achieve the optimal misclustering error rate.

In this paper, we investigate the theoretical performance of spectral clustering in the isotropic Gaussian Mixture Model.  
In this model data points are generated from a mixture of Gaussian  distributions with identity covariance each, whose centers are separated from each other, resulting in a cluster structure. The goal is to recover the underlying  true cluster assignment.

Maximum likelihood estimation for the cluster assignment labels in the isotropic Gaussian Mixture Model is equivalent to the $k$-means algorithm. Finding an exact solution to the $k$-means objective has an exponential dependence on the dimension of the data points \cite{InabaKatohImai94,MahajanNimbhorkarVaradarajan09} and hence is not feasible, even in moderate dimensions. 
As a result, various approximations have been used and studied. One direction is to relax the $k$-means objective by semi-definite programming (SDP) \citep{Royer17, GiraudVerzelen19, PengWei07, fei2018hidden}. These relaxations are more robust to outliers than spectral methods \cite{SrivastavaSarkarHanasusanto20}, but have a slower running time. Another possibility is to apply Lloyd’s Algorithm \citep{Lloyd82, lu2016statistical}, which is a greedy iterative method to approximately find a solution to the $k$-means objective. 
Given a sufficiently good initializer, typically provided by spectral clustering \citep{kumar2010clustering}, Lloyd's Algorithm achieves the optimal misclustering rate \citep{lu2016statistical, Ndaoud19}. However, in fact, we show that spectral clustering itself is already optimal when the error variance is isotropic. 



A closely related result about spectral clustering for the Gaussian Mixture Model is \citep{VempalaWang04}. Under a strong separation condition 
spectral clustering is proved to achieve exact recovery of the underlying cluster structure with high probability.
In this paper, we consider also situations where only partial recovery is possible. We measure the performance of the spectral clustering output $\hat z$ by the normalized Hamming loss function $\ell(\cdot,\cdot)$.  
We summarize our main result informally in Theorem \ref{thm:informal}.

\begin{theorem}[Informal Statement of the Main Result]\label{thm:informal}
For $n$ data points generated from a Gaussian Mixture Model with isotropic covariance matrix, we assume that
\begin{itemize}
\item the number of clusters is finite
\item the size of the clusters are of the same order
\item the minimum distance among the centers, $\Delta$, goes to infinity
\item the dimension $p$ of each data point is at most of the same order as $n$.
\end{itemize}
Then,  with high probability, spectral clustering achieves the optimal misclustering rate, which is
\begin{align*}
\ell(\hat z, z^*)\leq  \ebr{-\br{1-o(1)} \frac{\Delta^2}{8}}.
\end{align*}
\end{theorem}
This provides the first theoretical guarantee on the optimality of spectral clustering in a general setting. The separation parameter $\Delta$ covers a wide scale of values, ranging  from consistent cluster estimation to exact recovery. We refer readers to Theorem \ref{thm:main} 
for a rigorous statement and a slightly stronger result, where we allow the number of clusters to grow with $n$, the cluster sizes to be not necessarily of the same order, and the dimension $p$ to possibly grow faster than $n$. 

In particular, in Theorem \ref{thm:informal}, no spectral gap (i.e., singular value gap) condition is needed. This is contrary to the existing literature \cite{abbe2017entrywise, lei2015consistency, Jin15, rohe2011spectral}, where various forms of eigenvalue gap or singular value gap conditions are required to apply matrix perturbation theory. This does not match the intuition that the difficulty of clustering should be determined by the distances between the cluster centers,  regardless of the spectral structure. In this paper, we completely drop any condition on the spectral gap. We achieve this by showing that the contribution of singular vectors from smaller singular values is negligible. 

A recent related paper by Abbe et al. \cite{abbe2017entrywise} studies community detection in an  idealized scenario, where the network has two equal-size communities and the connectivity probabilities are  equal to $an^{-1}\log n$ or $bn^{-1}\log n$, where $a$ and $b$ are fixed constants. They show that the performance of  clustering on the second leading eigenvector matches with the minimax rate, by using a leave-one-out technique.  The technical tools we use in this paper are different. We extend spectral operator perturbation theory of \citep{KoltchinskiiLouniciAAHP, KoltchinskiiXia15} and introduce new techniques to establish optimality of spectral clustering and to remove the spectral gap condition.


\paragraph{Organization} The paper is organized as follows. In Section \ref{sec:main}, we first introduce the Gaussian Mixture Model,  followed by the spectral clustering algorithm, and then state the main results. We discuss extensions and potential caveats of our analysis in Section \ref{sec:disc}. The proof of the main theorem is given in Section \ref{sec:proof}, which is started with a proof sketch. We include the proofs of all the lemmas in the supplement.

\paragraph{Notation} For any matrix $M$, we denote by $\opnorm{M}$ and $\fnorm{M}$ its operator norm and Frobenius norm, respectively.  $M_{i,\cdot}$ denotes the $i$-th row of $M$ and $M_\cd{i}$  its $i$-th column.  For matrices $M,N$ of the same dimension, their inner product is defined as $\iprod{M}{N} = \sum_{i,j}M_{ij}N_{ij}$. For any $d$, we denote by  $\{e_a\}_{a=1}^d$ the standard Euclidean basis with $e_1=(1,0,0,\ldots), e_2=(0,1,0,\ldots,0),\ldots, e_d=(0,0,0,\ldots, 1)$. We let $1_d$ be a vector of length $d$ whose entries are all $1$. We use $[d]$ to denote the set $\{1,2,\ldots,d\}$ and $\indic{\cdot}$ to denote the indicator function. For $y_1,y_2,\ldots, y_d\in\mathr$,  $\text{diag}(y_1,y_2,\ldots, y_d)$ denotes the $d\times d$ diagonal matrix with diagonal entries  $y_1,y_2,\ldots, y_d$. 

\section{Main Results} \label{sec:main}

\subsection{Gaussian Mixture Model}
We consider an isotropic Gaussian Mixture Model with $k$ centers 
$\theta_1^*,\ldots,\theta_k^*\in\mathr^p$ and a cluster assignment vector $z^* \in [k]^n$. In this model, 
independent observations $\{X_i\}_{i\in[n]}$ are generated as follows:
\begin{align} \label{eq:GMM}
X_i = \theta^*_{z^*_i} + \epsilon_i, ~~~~{\epsilon_i} \thicksim  \mathn(0,I_p).
\end{align}

The goal of clustering is to recover the cluster assignment $z^*$. We measure the quality of a clustering algorithm by the average number of misclustered labels. Since the cluster structure is invariant to permutation of the label symbols, we define the misclustering error as 
\begin{align*}
\ell(z,z^*) := \min_{\phi\in \Phi} \frac{1}{n}\sum_{i\in[n]}  \indic{\phi\br{z_i} \neq z^*_i},
\end{align*}
where $\Phi=\cbr{\phi: ~ \phi~ \text{is a bijection from }[k]\text{ to }[k]}$. 

The difficulty of clustering is mainly determined by the distances between the centers $\cbr{\theta^*_1,\ldots,\theta^*_k}$. If two centers are exactly equal to each other, it is impossible to distinguish the corresponding two clusters. We define $\Delta$ to be the minimum distance among centers:
\begin{align} 
\Delta = \min_{j,l\in[k]:j\neq l}\norm{\theta^*_j - \theta^*_l}.\label{eqn:delta}
\end{align}

Another quantity that determines the possibility of consistent clustering is 
the size of the clusters. 
When the size of a cluster is small, recovery might be more difficult. We quantify the size of the smallest cluster by $\beta$, defined as  
\begin{align}\label{eqn:beta_def}
 \beta = \frac{\min_{j\in[k]} \abs{\cbr{i\in[n]:z^*_i = j}}}{n/k}.
\end{align} 
Note that $\beta$ cannot be greater than 1. 
We allow the case $\beta =o(1)$, such that clusters sizes may differ in magnitude. 



\subsection{Spectral Clustering}
Various forms of spectral clustering have been proposed and studied in the  literature. Spectral clustering is an umbrella term for clustering  after a  dimension reduction through a spectral decomposition. The variants differ mostly for the matrix on which the spectral decomposition is applied, and wich spectral components are used for the subsequent clustering. The clustering method used most commonly is the $k$-means algorithm.

 In the context of community detection, spectral clustering \citep{lei2015consistency, Jin15, qin2013regularized, rohe2011spectral, zhou2019analysis} is usually performed on the eigenvectors of the adjacency matrix. For general clustering settings, \citep{hein2006uniform, hein2005graphs, gine2006empirical, von2008consistency, belkin2003laplacian, von2007tutorial} first obtain a similarity matrix from the original data points by applying a kernel function. Then the graph Laplacian is constructed, whose eigenvectors are used for clustering. In \citep{kannan2009spectral, kumar2010clustering}, spectral clustering is performed directly on the original data matrix.

The spectral clustering algorithm considered in this paper is presented in Algorithm \ref{alg:main}. It is simple, involves only one singular value decomposition (SVD) and one $k$-means clustering step. Despite the simplicity of this approach, it is powerful, as it achieves the optimal misclustering rate. 
The key step in the algorithm that leads to the optimal rate is to weight the empirical singular vectors by the corresponding empirical singular values. 

%

\begin{algorithm}[h]
\SetAlgoLined
\KwIn{Data matrix $X\in\mathr^{p\times n}$, number of clusters $k$}
\KwOut{Clustering assignment vector $\hat z\in [k]^n$}
 \nl Perform SVD on $X$  to decompose $$X = \sum_{i=1}^{p \wedge n} \hat \sigma_i \hat u_i \hat v_i^T,$$ where $\hat \sigma_1 \geq \hat \sigma_2 \geq \ldots \geq \hat \sigma_{p\wedge n} \geq 0$ and $\cbr{\hat u_i}_{i=1}^{p \wedge n}\in\mathr^p, \cbr{\hat v_i}_{i=1}^{p \wedge n}\in\mathr^n$.
 

\nl Consider the first $k$ singular values and corresponding singular vectors. Define $\hat \Sigma:=\text{diag}(\hat \sigma_1, \dots, \hat \sigma_k), ~\hat V:= \br{\hat v_1,\ldots, \hat v_k}, ~\hat U:= \br{\hat u_1,\ldots, \hat u_k}$ and

  $$\hat Y := \hat U^T X=\hat \Sigma \hat V^T\in\mathr^{k\times n}.$$ 
 \nl  Perform $k$-means on the columns of $\hat Y $ 
 and return an estimator $\hat z$ for the clustering assignment vector, i.e., 
 \begin{align}
 \br{\hat z,\cbr{\hat c_j}_{j=1}^k} = \argmin_{ z\in [k]^n , \cbr{ c_j}_{j=1}^{k} \in \mathr^{k}} \sum_{i\in[n]} \norm{\hat Y_\cd{i} - c_{z_i}}^2.\label{eqn:kmeans_Y}
 \end{align}
\caption{Spectral Clustering}\label{alg:main}
\end{algorithm}

As common in the clustering literature, we assume that $k$, the number of clusters, is known. The purpose of the SVD is to reduce the dimensionality of the data while preserving underlying structure. After SVD, the dimensionality of the data vectors is reduced from $p$ to $k$ \footnote{Here we assume $p\geq k$. If $p<k$ then the dimensionality reduction is not needed and Algorithm \ref{alg:main} reduces to the $k$-means algorithm. To accommodate both $p\geq k$ and $p<k$, Step 2 of Algorithm \ref{alg:main} can be slightly changed by using the leading $\min\cbr{k,p}$ singular vectors instead.}. This makes the follow-up $k$-means algorithm computationally feasible compared to applying it directly onto the columns of $X$. Finding an exact solution for the  $k$-means objective of the projected data (i.e.,  (\ref{eqn:kmeans_Y})) has computational complexity $O(n^{k^2+1})$ \cite{InabaKatohImai94}, which is polynomial in $n$ if $k$ is constant. In Section \ref{subsec:one_plus_epsilon}, we show how to modify Algorithm \ref{alg:main}, using a $(1+\varepsilon)$-solution for the $k$-means algorithm to achieve linear (in $n$) complexity. 

The idea of weighting singular vectors by the corresponding singular values is natural. The importance of singular vectors is different:  singular vectors with smaller singular values should carry relatively less useful information, and consequently deserve less attention. Clustering on $\hat Y$ instead of $\hat V$ is also the main reason why we are able  to remove the spectral gap condition. In particular, we will show in Lemma \ref{lem:equivalence} that
Algorithm \ref{alg:main} is equivalent to Algorithm \ref{alg:rankk}, which 
performs clustering on the columns of the rank-$k$ matrix approximation of $X$. Similar ideas of using low rank matrix approximations for clustering have also been proposed in  \citep{kumar2010clustering, gao2018community}.
\subsection{Consistency}
We first present a preliminary result that proves consistency of the estimator $\hat z$ obtained in  Algorithm \ref{alg:main}. 
\begin{proposition}
\label{prop:cons}
Assume that $\Delta/(\beta^{-0.5} k\br{1+p/n}^{0.5}) \geq C$ for some large enough constant $C>0$. Then the output of Algorithm \ref{alg:main}, $\hat z$, satisfies for another constant $C'>0$
\begin{align} \label{prop:cons error}
&\ell(\hat z,z^*)\leq  \frac{C'k\br{1+\frac{p}{n}}}{\Delta^2}
\end{align} 
with probability at least $1-\exp(-0.08n)$,
\end{proposition}

Proposition \ref{prop:cons} is an immediate consequence of Lemma \ref{lem:equivalence} and Lemma \ref{lem:theta_dist}, which are stated in Section \ref{sec:proof}. It is worth mentioning that there is no spectral gap condition assumed.
In addition, Proposition \ref{prop:cons}  can be extended to mixture models  where the errors $\cbr{\epsilon_i}$ are not  necessarily $\mathcal{N}(0,I_p)$ distributed. We include this extension in Appendix \ref{apx:prop_cons_extension} as Proposition \ref{prop:cons_extension}. 


\subsection{Optimality}
In the next theorem we establish that Algorithm \ref{alg:main} achieves in fact an exponential convergence rate in the Gaussian Mixture Model when the covariance matrix of the Gaussian noise variables is isotropic. 

\begin{theorem}\label{thm:main}
Suppose that
\begin{align}\label{eqn:Delta_assumption}
\frac{\Delta}{k^{10.5}\beta^{-0.5}(1+\frac{p}{n})\left ({\frac{n-k}{n}} \right )^{-0.5}} \rightarrow \infty.
\end{align}
Then the output of Algorithm \ref{alg:main}, $\hat z$, satisfies
\begin{align}\label{eqn:main}
 \ell(\hat z,z^*)  \leq \ebr{ - \br{1- \br{\frac{\Delta}{k^{10.5}\beta^{-0.5}(1+\frac{p}{n})^{}\left ({\frac{n-k}{n}} \right )^{-0.5}}}^{-0.1}} \frac{\Delta^2}{8}}
\end{align}
with probability at least $1-\ebr{-\Delta}-3nk\ebr{-0.08(n-k)}$.
\end{theorem}
In Theorem \ref{thm:main}, we allow  the number of clusters $k$ to grow with $n$, the cluster sizes not to be of the same order (quantified by $\beta$), 
and the dimension $p$ to be of larger order than $n$. This is slightly stronger than the informal statement we make in Theorem \ref{thm:informal}. 
%
%

The following minimax lower bound for recovering $z^*$ in the Gaussian Mixture Model is established in \citep{lu2016statistical}: 
\begin{align}\label{eqn:minimax}
\inf_{\hat z}\sup_{ (\theta_1^*,\ldots,\theta^*_k), z^* } \E \ell(\hat z, z^*) \geq  \ebr{-\br{1-o(1)} \frac{\Delta^2}{8}}, ~\text{if}~\frac{\Delta}{\log \br{k\beta^{-1}}}\rightarrow\infty.
\end{align}
Here the infimum is taken over all feasible estimators $\hat z$, and the supremum is taken over all possible parameters, where the true centers $\br{\theta_1^*,\ldots,\theta^*_k} \in\mathr^{p\times k}$ are separated by at least $\Delta$, and  the true cluster assignment $z^*$ has minimum cluster size $\beta n/k$. 

When $\Delta \rightarrow \infty$, $p=o(n\Delta)$ and $k$ and $\beta$ are constants, the convergence rate in \eqref{eqn:main} matches the minimax lower bound (\ref{eqn:minimax}) up to a $(1+o(1))$ factor in the exponent. Moreover, when additionally $\liminf_{n\rightarrow\infty} \Delta^2/(8\log n)>1$, $\hat z$ equals
$z^*$ with high probability and we achieve exact recovery. This sharply matches the exact recovery threshold  \cite{lu2016statistical, ChenYang20}.


Whereas $\Delta \rightarrow \infty$ is a necessary condition for consistent recovery \cite{lu2016statistical, Ndaoud19}, the condition in \eqref{eqn:Delta_assumption}  is not optimal. The assumption that $p=o(n\Delta)$ is an artifact of our proof technique. It can be improved to $p=o(n\Delta^2)$ under additional assumptions on the singular values of the population matrix $\mathbb{E}X$. 
When $n\Delta^2=o(p)$,  Algorithm \ref{alg:main} may only achieve  suboptimal convergence rates and we discuss the intuition behind this in Section \ref{sec:discfail}. The dependence on $k$ is suboptimal as well,  mainly due to higher order perturbation terms in our proof. In contrast, \cite{GiraudVerzelen19,fei2018hidden} only need to assume $kp=o(n\Delta^2)$ for their SDP relaxation of $k$-means to achieve exponential rates (but with a suboptimal constant in the exponent).

We emphasize that, as in Proposition \ref{prop:cons}, there is no spectral gap (i.e., singular value gap) condition assumed in Theorem \ref{thm:main}. 
It is possible that  the population matrix $\E X$ has a rank that is smaller than $k$, such that the smallest singular values of the population matrix $\E X$ are $0$ or near $0$. For instance, this occurs when some of the centers are (nearly) collinear. This is contrary to the existing literature   \cite{abbe2017entrywise, lei2015consistency,Jin15, rohe2011spectral}, where the spectral gap is assumed to be sufficiently large to apply spectral perturbation theory.
 The spectral gap condition is not natural, as the minimax rate in  (\ref{eqn:minimax}) only depends on $\Delta$ and is invariant to any spectral structure. In Theorem \ref{thm:main} 
 we completely drop any spectral gap condition, and our results match with the intuition that the difficulty of cluster recovery is determined only by $\Delta$, the minimum distance among the centers.

 \subsection{ 
 $(1+\varepsilon)$-solutions to k-means}\label{subsec:one_plus_epsilon}
 
 Computing the $k$-means objective in Algorithm \ref{alg:main} has complexity $O(n^{k^2+1})$ \cite{InabaKatohImai94} and quickly becomes impractical, even for moderate values of $k$. A potential alternative is to use an $(1+\varepsilon)$-solution. An  $(1+\varepsilon)$-solution is a pair $(\tilde z, \{ \tilde c_j \}_{j=1}^k)$, such that its $k$-means objective value is within a factor of $(1+\varepsilon)$ of the global minimum of the $k$-means objective.
 For instance, \cite{KumarSabharwalSen04} proposed an $(1+\varepsilon)$-approximation algorithm with complexity $O(2
^{(k/\varepsilon)^{O(1)}}n)$, which is linear in $n$ when $k$ is constant and polynomial in $n$ as long as $k$ grows sublogarithmically in $n$. Proposition \ref{prop:cons} is still valid when an $(1+\varepsilon)$-solution is used.  However, $(1+\varepsilon)$-solutions do not necessarily enjoy a {\em local} optimality guarantee for the estimated labels, i.e., $\|\hat Y_i-\tilde c_{\tilde z_i} \| \leq \| \hat Y_i-\tilde c_j\|, \forall i\in[n], j \neq \tilde z_i$, which is required in the proof of Theorem \ref{thm:main}.
To overcome this problem, we propose to run an extra one step Lloyd's algorithm \cite{Lloyd82} as described in Algorithm \ref{alg:1+eps}.
Consequently, the statement of Theorem \ref{thm:main} still holds for Algorithm \ref{alg:1+eps}, which we present below in Theorem \ref{thm:1_epsilon_approx}.


\begin{algorithm}[h]
\SetAlgoLined
\KwIn{Data matrix $X \in\mathr^{p\times n}$, number of clusters $k$, approximation level $\varepsilon$}
\KwOut{Clustering assignment vector $\tilde z\in [k]^n$}
 \nl Implement Steps 1-2 of Algorithm \ref{alg:main} to obtain $\hat Y \in\mathr^{k\times n}$\;
 \nl  Compute a $(1+\varepsilon)$-solution (e.g., \cite{KumarSabharwalSen04}) for the $k$-means algorithm on the columns  of $\hat Y $ 
 and return $(\check z, \{ \check c_j\}_{j=1}^k)$, the cluster assignment vector and centers, such that
 \begin{align*}
\sum_{i\in[n]}  \norm{\hat Y_\cd{i} - \check c_{\check z_i}}^2  \leq (1+\varepsilon) \inf_{ \{ c_j\}_{j=1}^k \in \mathbb{R}^k} \sum_{i\in[n]} \min_{j \in [k]} \norm{\hat Y_\cd{i} - c_{j}}^2 
 \end{align*}
  \nl Update the centers
 \begin{align*}
    \tilde c_j = \frac{\sum_{i\in[n]} \hat Y_\cd{i} \indic{\check z_i = j}}{\sum_{i\in[n]} \indic{\check z_i = j}},
     ~~~~~~j=1, \dots, k.
 \end{align*}
 \nl Update the labels 
 \begin{align*}
     \tilde z_i=\argmin_{j \in [k]} \| \hat Y_\cd{i} - \tilde c_j \|, ~~~~~~i=1, \dots, n.
 \end{align*}
\caption{Spectral Clustering with $(1+\varepsilon)$-solution}\label{alg:1+eps}
\end{algorithm}

\begin{theorem}\label{thm:1_epsilon_approx}
Assume that \begin{align*}
\frac{\Delta}{k^{10.5}\beta^{-0.5}(1+\frac{p}{n})\left ({\frac{n-k}{n}} \right )^{-0.5} (1+\varepsilon)^{0.5}} \rightarrow \infty
\end{align*} holds. Then the output of Algorithm \ref{alg:1+eps}, $\tilde z$, satisfies
\begin{align} \label{thm:1_eps_exp}
 \ell(\tilde z,z^*)  \leq \ebr{ - \br{1- \br{\frac{\Delta}{k^{10.5}\beta^{-0.5}(1+\frac{p}{n})^{}\left ({\frac{n-k}{n}} \right )^{-0.5}(1+\varepsilon)^{0.5}}}^{-0.1}} \frac{\Delta^2}{8}}
\end{align}
with probability at least $1-\ebr{-\Delta}-3nk\ebr{-0.08(n-k)}$.
\end{theorem}
The proof of Theorem \ref{thm:1_epsilon_approx} is almost identical to that of  Theorem \ref{thm:main} and we sketch the necessary modifications in Appendix \ref{apx:1_eps}. 




 \section{Discussion}  \label{sec:disc}
\subsection{Unknown Covariance Matrix \& Sub-Gaussian errors}\label{subsec:epsilon_distribution}

The consistency guarantee established in Proposition  \ref{prop:cons} can be  extended to more general settings where the noise variables $\{\epsilon_i \}_{i=1}^n$ have covariance matrix $\Sigma$ or are sub-Gaussian.  We include this extension in Appendix \ref{apx:prop_cons_extension} as Proposition \ref{prop:cons_extension}. 

In contrast, it is not possible to extend Theorem \ref{thm:main} and Theorem \ref{thm:1_epsilon_approx} to either sub-Gaussian distributed errors or unknown covariance matrices with our current proof techniques. This is due to the fact that the proof is highly reliant on both the isoperimetric inequality (c.f., (\ref{eqn:isoperimetric})) and rotation invariance of the singular vectors of the noise matrix $(\epsilon_1, \dots, \epsilon_n)$ (as in Lemma \ref{lem:hat_V_normal}). An isoperimetric inequality would also be fulfilled by strongly log-concave distributed errors \cite{OttoVillani00}. On the other hand, rotation invariance of the singular vectors of $(\epsilon_1, \dots, \epsilon_n)$ is equivalent to  $\epsilon_i$ being spherically Gaussian distributed.

\subsection{Unknown $k$}  Algorithm \ref{alg:main} and Theorem \ref{thm:main}  require that the number of clusters, $k$, is known. In practice, $k$ might be unknown and might need to be estimated. For this purpose, several approaches have been developed, including cross-validation \cite{Wang10}, the gap-statistic \cite{TibshiraniWaltherHastie01}, eigenvalue based heuristics \cite{von2007tutorial} and resampling strategies \cite{MontiTamayoMesirovGolub03}. However, while these methods often work well empirically, their theoretical performances are not fully understood, especially in high-dimensional regimes with growing $p$ and $n$. One may estimate $k$ by the aforementioned methods and use the resulting estimate in Algorithm \ref{alg:main}, but further investigation is beyond the scope of this paper.


\subsection{Parameter Regime $n\Delta^2=O(p)$}\label{sec:discfail} Proposition \ref{prop:cons} and Theorems \ref{thm:main} and \ref{thm:1_epsilon_approx} are limited to the parameter regimes $p=o(n\Delta^2)$ and $p=o(n\Delta)$, respectively, beyond which the performance of Algorithm \ref{alg:main} remains unclear. In the high-dimensional setting, where $p$ satisfies $n\Delta^2=O(p)$ and $p=o(n\Delta^4)$,  \cite{AbbeFanWang20} considers a simplified model where $X_i=z^*_i \theta^*+\epsilon_i$ with $z^*\in \{-1,1\}^n$ and shows that spectral clustering performed with one singular vector achieves the optimal misclustering rate. Nevertheless, spectral clustering should be used with caution in the high-dimensional setting. In particular, when $n\Delta^2=o(p)$, Theorem 2.2 in \cite{Ding20} indicates that, in general, the leading empirical singular values are all equal to $(1+o(1))(\sqrt{n}+\sqrt{p})$. As a result, running $k$-means on $\hat Y$ in Algorithm \ref{alg:main} is the same as on $\hat V$, and its performance may also depend on the structure of the population singular values \cite{HanTongFan20}. Thus, when the signal is weak compared to the dimensionality of the data, i.e.,  $n\Delta^2 = o(p)$, one may consider using alternative clustering methods such as SDP relaxations of $k$-means \cite{PengWei07,Royer17,GiraudVerzelen19,fei2018hidden}. However, rate optimal estimation is not guaranteed.

\subsection{Adaptive Dimension Reduction} The population matrix $(\theta^*_{z^*_1}, \dots, \theta^*_{z^*_n})$ might have  smaller rank than $k$. For instance, when the centers are collinear, the rank of the population matrix equals 1. Hence, in such cases it is conceivable to use a smaller number of singular vectors in Algorithm \ref{alg:main}, as this further reduces the computational burden of computing the $k$-means objective. One way to achieve this, while still retaining the theoretical guarantees of Theorem \ref{thm:main}, is to use the leading $\hat r$ singular vectors for the projection Step 2 in Algorithm \ref{alg:main}, where  $\hat r$ is an empirical version of $r$ defined in \eqref{def:r}. This pre-selection step keeps all the informative singular vectors without involving the noisy part of the projected data corresponding to small population singular values and allows to shorten the proof of Theorem \ref{thm:main}. On the other hand, estimating $r$ requires the noise level to be known or to be estimated, which adds additional computational complexity and introduces an additional tuning parameter.

\section{Proof of Main Results} \label{sec:proof}
In Section \ref{subsec:population}, we first introduce the population counterparts of the quantities appearing in Algorithm \ref{alg:main}. After that, several key lemmas for the proof are presented in Section \ref{subsec:lemma}. Since the proof of  Theorem \ref{thm:main}  is long and involved, we provide a proof sketch  in Section \ref{subsec:sketch}, followed by its complete  and detailed proof in Section \ref{subsec:proof}. Auxiliary lemmas are included in the supplement.

\subsection{Population Quantities} \label{subsec:population}
We define $P = \E X$ and  $E=\br{\epsilon_1,\ldots, \epsilon_n}\in\mathr^{p\times n}$, 
such that we have the  matrix representation  $X=P+E$.
We define several quantities related to $P$, the population version of $X$. We denote the SVD of $P$  (note that $P$ is at most rank of $k \wedge p$)
\begin{align*}
P =\sum_{i=1}^{k} \sigma_i u_i v_i^T=U\Sigma V^T
\end{align*}
where $\sigma_1 \geq \sigma_2 \geq \ldots \geq \sigma_{k} \geq 0$, $\Sigma=\text{diag}(\sigma_1, \dots, \sigma_k),$ 
$U = \br{u_1,\ldots, u_k}\in\mathr^{p\times k}, ~V = \br{ v_1,\ldots,  v_k} \in\mathr^{n\times k}.$ 
 Moreover, we define 
\begin{align*}
 Y =  U^TP=\Sigma  V^T\in\mathr^{k\times n}.
\end{align*}
In Appendix \ref{subsec:support_population}, we provide several propositions (Propositions \ref{prop:V_form}, \ref{prop:V_pop} and \ref{prop:iprod_theta_u}) to characterize the structure of these population quantities. 

\subsection{Key Lemmas}\label{subsec:lemma}
In this section, we present several key lemmas used in the proof of Theorem \ref{thm:main}.

\begin{algorithm}[h]
\SetAlgoLined
\KwIn{Data matrix $X\in\mathr^{p\times n}$, number of clusters $k$}
\KwOut{Clustering assignment vector $\hat z'\in [k]^n$}
 \nl Implement Steps 1-2 of Algorithm \ref{alg:main} to obtain $\hat \Sigma\in\mathr^{k\times k} ,~\hat V\in\mathr^{n\times k}$ and $\hat U\in\mathr^{p\times k}$. In addition, define 
 \begin{align*}
 \hat P =\hat U \hat \Sigma \hat V^T\in\mathr^{p\times n}.
 \end{align*}
  \nl
Perform $k$-means on the columns of $\hat P$ 
and return the estimated clustering assignment vector $\hat z'$ and estimated centers $\{\hat \theta_j\}_{j=1}^k$, i.e., 
 \begin{align}
 \br{\hat z',\cbr{\hat \theta_j}_{j=1}^k} = \argmin_{ z\in [k]^n , \cbr{ \theta_j}_{j=1}^{k} \in \mathr^{k}} \sum_{i\in[n]}\norm{\hat P_\cd{i} - \theta_{z_i}}^2.\label{eqn:kmeans_P}
 \end{align}
\caption{Clustering with rank-$k$ approximation}\label{alg:rankk}
\end{algorithm}

In Lemma \ref{lem:equivalence}, we show that Algorithm \ref{alg:main}  has the same output as Algorithm \ref{alg:rankk}, where clustering is performed on the columns of $\hat U \hat Y$ instead of $\hat Y$. 
We defer its proof to the supplement.
\begin{lemma}\label{lem:equivalence}
Denote by $(\hat z,\{\hat c_j\}_{j=1}^k)$ and $ (\hat z',\{\hat \theta_j\}_{j=1}^k)$ the outputs of Algorithm \ref{alg:main} and Algorithm \ref{alg:rankk}, respectively. 
Then, after a label permutation,  $\hat z$ equals $\hat z'$, i.e., there exists a $\phi \in\Phi$ such that
\begin{align*}
\hat z'_i = \phi(\hat z_i),\forall i\in[n].
\end{align*}
In addition, we have that 
\begin{align*}
\hat \theta_j = \hat U \hat c_{\phi(j)},\forall j\in[k].
\end{align*}
\end{lemma}
In Lemma \ref{lem:theta_dist}, we show consistency of Algorithm  \ref{alg:rankk} on the following event
\begin{align}\label{eqn:F_def}
\mathcal{F} = \cbr{\opnorm{E}\leq \sqrt{2}(\sqrt{n}+\sqrt{p})}.
\end{align} which occurs with high probability (as proven in Lemma \ref{lem:E_opnorm}).  


\begin{lemma}\label{lem:theta_dist}
Assume that the  event $\mathcal{F}$  holds~
and that $\Delta/(\beta^{-0.5} k\br{1+p/n}^{0.5}) \geq C$
for some constant $C>0$. Then there exists another constant $C'$ such that the output of Algorithm \ref{alg:rankk} $ (\hat z',\{\hat \theta_j\}_{j=1}^k)$ satisfies
\begin{align}
&\ell(\hat z',z^*)\leq  \frac{C'k\br{1+\frac{p}{n}}}{\Delta^2},\label{eqn:Lemma_4.2_1}\\
\text{and }\quad &\min_{\phi\in \Phi}\max_{j\in[k]} \norm{\hat \theta_j  - \theta^*_{\phi(j)}} \leq C'\beta^{-\frac{1}{2}}k\sqrt{1+\frac{p}{n}}.\label{eqn:theta_dist} 
\end{align}
Consequently, if the ratio $\Delta/(\beta^{-0.5} k\br{1+p/n}^{0.5})$
is sufficiently large, we have that  $\min_{j\in[k]} \abs{\cbr{i\in[n]:\hat z_i = j}} \geq \frac{\beta n}{2k}$.
\end{lemma}

The proof of Lemma \ref{lem:theta_dist} is included in  the supplement.
The results of Lemma \ref{lem:E_opnorm}, Lemma  \ref{lem:theta_dist} and Lemma \ref{lem:equivalence} immediately imply Proposition \ref{prop:cons}.
%

Lemma \ref{lem:VV_pert} studies the difference between the empirical spectral projection matrix and its sample counterpart. It decomposes  $\hat V_{a:b} \hat V_{a:b}^T - V_{a:b}V_{a:b}^T$
into a linear part of the random noise matrix $E$ and a remaining part, which can be shown to be negligible.  The linear part has a simple form, and is the main component that leads to the exponent $\Delta^2/8$ in  (\ref{eqn:main}). The remaining  non-linear part, though without an explicit expression, is well-behaved and concentrates strongly around $0$. 
 Lemma \ref{lem:VV_pert} is a slight generalization of results due to  \citep{KoltchinskiiLouniciAAHP, KoltchinskiiXia15}, where $\sigma_a,\ldots, \sigma_b$ are assumed to be the same. 
 Here we relax this assumption, by allowing the corresponding singular values to vary. The proof of Lemma \ref{lem:VV_pert} is involved but mainly follows the line of arguments in \citep{KoltchinskiiLouniciAAHP, KoltchinskiiXia15}. We include the proof in the supplement for completeness.
\begin{lemma}\label{lem:VV_pert}
Consider any  rank-$k$ matrix $M\in\mathr^{p\times n}$ with SVD $M=\sum_{j=1}^k \sigma_j u_j v_j^T$ where $\sigma_1 \geq \sigma_2 \ldots \geq \sigma_k >0$. Define $\sigma_0 = +\infty $ and $ \sigma_{k+1} =0$.

Suppose that $E$ is a matrix with i.i.d. Gaussian entries,  $E_{i,j}
$. Define $\hat  M = M + E$ and suppose that $\hat  M$ has SVD $\sum_{j=1}^{p\wedge n} \hat  \sigma_j \hat u_j \hat v_j^T$ where $\hat \sigma_1 \geq \hat \sigma_2\geq \ldots \geq \hat \sigma_{p\wedge n}$. 
For any two indices $a,b$ such that $1\leq a\leq b\leq k$, define $V_{a:b}=\br{v_a,\ldots,v_b}$,  $\hat V_{a:b}=\br{\hat v_a,\ldots,\hat v_b}$ and  $V=\br{v_1,\ldots, v_k}$. Moreover, define the singular value gap $g_{a:b}= \min\cbr{\sigma_{a-1} -\sigma_a , \sigma_b -\sigma_{b+1}}$ and denote 
\begin{align*}
S_{a:b} = \br{I-VV^T}\br{\hat V_{a:b} \hat V_{a:b}^T - V_{a:b}V_{a:b}^T} V_{a:b} - \sum_{a\leq j\leq b}\frac{1}{\sigma_j}\br{I-VV^T}E^Tu_jv_j^TV_{a:b}.
\end{align*}
Suppose that $\E \opnorm{E} \leq \frac{g_{a:b}}{8}$. Then there exists some constant $C>0$ such that with probability at least $1-2e^{-t}$
\begin{align*}
{\abs{\iprod{S_{a:b} - \E S_{a:b} }{W}} \leq C  \left (1+\frac{\sigma_a-\sigma_b}{g_{a:b}} \right ) \frac{\sqrt{t}}{g_{a:b}} \left ( \frac{\sqrt{n+p}+\sqrt{t}}{g_{a:b}}\right ) \|W\|_* } 
\end{align*}
for any $W\in\mathr^{n\times \br{b-a}}$, any $t\geq \log 4$ and where $\| \cdot\|_*$ denotes the nuclear (Schatten-1) norm. 
\end{lemma}
The next lemma, Lemma \ref{lem:hat_V_normal}, characterizes the distribution of empirical singular vectors. Similar to Lemma \ref{lem:VV_pert}, Lemma \ref{lem:hat_V_normal} holds for matrices with any underlying structure, not necessarily in the clustering setting, as long as the noise is Gaussian distributed. The most important implication of Lemma \ref{lem:hat_V_normal} is that, for any empirical singular vector $\hat v_j$, its component that is orthogonal to the true signal $V$ (i.e., $(I-VV^T)\hat v_j$) is after normalization haar distributed on the sphere spanned by $(I-VV
^T)$.  This observation appears and is utilized in \citep{johnstone2018pca, paul2007asymptotics}. Lemma  \ref{lem:hat_V_normal}  is essentially the same as Theorem 6 of \citep{paul2007asymptotics}. For completeness, we give the proof in  the supplement.

\begin{lemma}\label{lem:hat_V_normal}
Consider a rank-$k$ matrix $M\in\mathr^{p\times n}$ with SVD $M=\sum_{j=1}^k \sigma_j u_j v_j^T$ where $\sigma_1 \geq \sigma_2 \ldots \geq \sigma_k >0$. 
Suppose that $E$ is a matrix with i.i.d. Gaussian entries, $E_{i,j}\iid\mathn(0,1)$. Define $\hat  M = M + E$ and suppose that $\hat  M$ has SVD $\sum_{j=1}^{p\wedge n} \hat  \sigma_j \hat u_j \hat v_j^T$ where $\hat \sigma_1 \geq \hat \sigma_2\geq \ldots \geq \hat \sigma_{p\wedge n}$. 
Define $V = \br{v_1,\ldots, v_k}$. Then for any $j\in[k]$, the following holds:
\begin{enumerate}
\item $\br{I - VV^T} \hat v_j / \norm{ \br{I - VV^T} \hat v_j}$ is uniformly distributed on the unit sphere spanned by $\br{I-VV^T}$, i.e., 
\begin{align*}
\frac{\br{I - VV^T} \hat v_j }{\norm{ \br{I - VV^T} \hat v_j}}  \overset{d}{=} \frac{\br{I - VV^T} w }{\norm{ \br{I - VV^T} w}},\text{ where }w\sim\mathn(0,I_n)
\end{align*}
and where $\overset{d}{=}$ denotes equality in distribution. 
In particular, we have that  $$\E \frac{\br{I - VV^T} \hat v_j }{ \norm{ \br{I - VV^T} \hat v_j}} =0.$$
\item $\br{I-VV^T}\hat v_j/\norm{\br{I-VV^T}\hat v_j}$ is independent of $VV^T\hat v_j$.
\item $\br{I-VV^T}\hat v_j/\norm{\br{I-VV^T}\hat v_j}$ is independent of $\norm{ \br{I - VV^T} \hat v_j}$.
\end{enumerate}
\end{lemma}

\subsection{Proof Sketch for  Theorem \ref{thm:main}} \label{subsec:sketch} 
In this section, we provide a  sketch for the proof of Theorem \ref{thm:main}. The complete and detailed proof is given in section \ref{subsec:proof}. Throughout the proof, we assume that the random event $\mathcal{F}$  (defined in  (\ref{eqn:F_def})) holds.

We use the equivalence between  Algorithm  \ref{alg:main} and Algorithm \ref{alg:rankk} (by Lemma \ref{lem:equivalence}), where clustering is performed on the columns of $\hat  P = \hat U \hat Y$. Hence, it is sufficient to study the behavior of $(\hat z, \{\hat \theta_j\}_{j\in[n]})$.
Particularly,  (\ref{eqn:kmeans_P}) implies a {\em local} optimality result of the estimated labels, i.e., 
\begin{align*}
\hat z_i =\argmin_{j\in[k]} \norm{\hat P_\cd{i} - \hat \theta_j}^2,\forall i\in[n].
\end{align*}
Then after label permutation, which without loss of generality we assume to be $\phi=\text{Id}$, $n\ell(\hat z,z^*)$ can be bounded by
\begin{align*}
n\ell(\hat z,z^*) &
= \sum_{i=1}^n \indic{\argmin_{a\in[k]} \norm{\hat P_\cd{i} - \hat \theta_a}^2 \neq z^*_i} \\
& \leq \sum_{i=1}^n \sum_{a\neq z^*_i}\indic{\norm{\hat P_\cd{i} - \hat \theta_a}^2 \leq \norm{\hat P_\cd{i} - \hat \theta_{z^*_i}}^2}.
\end{align*}
We divide the remaining proof into four steps, corresponding to Sections \ref{subsubsec:decomposition} to \ref{subsubsec:final} in the complete proof. 
\paragraph{Step 1 (Sketch of Section \ref{subsubsec:decomposition})} We decompose $\ell(\hat z,z^*)$ into two parts: the first part corresponds to the leading large singular values, and the other one is related to the remaining ones. To achieve this, we split  $\{\hat P_\cd{i} \}_{i\in[n]}$ and $\{ \hat \theta_j\}_{j\in[k]}$ into two parts.
We define $r\in[k]$ as follows (with $\sigma_{k+1}:=0$)
\begin{align} \label{def:r}
r := \max\cbr{ j \in [k]: \sigma_j  - \sigma_{j+1} \geq \rho (\sqrt{n}+\sqrt{p})},
\end{align} 
where $\rho\rightarrow\infty$ is some quantity whose value will be given in the complete proof. There are two benefits in choosing $r$ this way:  singular values with index larger than $r$ are relatively small; and the singular value gap $\sigma_r -\sigma_{r+1}$ is large enough to apply matrix spectral perturbation theory.
We split $\hat U$ into $(\hat U_{1:r}, \hat U_{\br{r+1}:k})$ and hence we obtain that $\hat P_\cd{i} = \hat P_\cd{i}^\tpo + \hat P_\cd{i}^\tpt$, where
\begin{align*}
\hat P_\cd{i}^\tpo = {\hat U_{1:r}\hat U_{1:r}^T}\hat P_\cd{i} 
,\quad \text{and }\hat P_\cd{i}^\tpt  = {\hat U_{\br{r+1}:k}\hat U_{\br{r+1}:k}^T}\hat P_\cd{i}.
\end{align*}
Likewise, we decompose $\hat \theta_j = \hat \theta_j^\tpo +  \hat \theta_j^\tpt$. Then we estimate 
\begin{align}
n\ell(\hat z,z^*) \notag  
& \leq  \sum_{i=1}^n \sum_{a\neq z^*_i} \indic{\norm{\hat P_\cd{i}^\tpo - \hat \theta_a^\tpo}^2  - \norm{\hat P_\cd{i}^\tpo - \hat \theta^\tpo_{z^*_i}}^2 \leq \gamma\Delta^2   }\\
&\quad + \sum_{i=1}^n \sum_{a\neq z^*_i} \indic{ \gamma\Delta^2 \leq  -  \norm{\hat P_\cd{i}^\tpt - \hat \theta_a^\tpt}^2 + \norm{\hat P_\cd{i}^\tpt - \hat \theta^\tpt_{z^*_i}}^2   }  \notag \\
& =: \sum_{i=1}  \sum_{a\neq z^*_i} A_{i,a} + \sum_{i=1}  \sum_{a\neq z^*_i} B_{i,a}. \label{eq:A+B}
\end{align}
for some  $\gamma =o(1)$ such that $\gamma\Delta /k\rightarrow\infty$. The value of $\gamma$ will be given in the complete proof. We now investigate the two double-sums above separately. 
\paragraph{Step 2 (Sketch of Section \ref{subsubsec:upper_A})}
Here we consider the terms $A_{i,a}$ in the first double-sum above. 
Lemma \ref{lem:theta_dist}  shows that $\{\hat \theta_j\}_{j\in[k]}$ are close to their true values $\{ \theta^*_j\}_{j\in[k]}$:
\begin{align*}
\max_{j\in[k]} \norm{\hat \theta_j  - \theta^*_{j}} =o(\Delta).
\end{align*}
Together with the fact that the centers $\{ \theta_j^*\}_{j=1}^k$ are separated by $\Delta$ and that $\max_{j\geq r+1}\hat \sigma_j$ is relatively small, we bound
\begin{align*} A_{i,a} = & \indic{\norm{\hat P_\cd{i}^\tpo - \hat \theta_a^\tpo}^2  - \norm{\hat P_\cd{i}^\tpo - \hat \theta^\tpo_{z^*_i}}^2 \leq \gamma\Delta^2   }\\  \leq &  
\indic{\br{1-o(1)}\Delta \leq 2\norm{\hat P_\cd{i}^\tpo - \hat U_{1:r}\hat U_{1:r}^T \theta^*_{z^*_i}}}. 
\end{align*}
Next, we observe that $\hat P_\cd{i}^\tpo - \hat U_{1:r}\hat U_{1:r}^T \theta^*_{z^*_i} = \hat U_{1:r}\hat U_{1:r}^T( \hat P -P)e_i$  
and show that $\|\hat U_{1:r}\hat U_{1:r}^T( \hat P -P)VV^Te_i\| = o(\Delta)$ by using that $|V_{i,j}|\leq \sqrt{k/(n\beta)}$. Hence, we obtain that  \begin{align*} A_{i,a} & \leq \indic{\br{1-o(1)}\Delta \leq 2\norm{ \hat U_{1:r}\hat U_{1:r}^T \hat P(I-VV^T)e_i}} \\ 
& = \indic{\br{1-o(1)}\Delta \leq 2\norm{ \hat \Sigma_{r\times r}\hat V_{1:r}^T (I-VV^T)e_i}}.
 \end{align*}

Since the singular values may vary in magnitude, a direct application of spectral perturbation theory on $\hat V_{1:r}$ is not sufficient. Instead, we split $[r]$ into disjoint sets $\cup_{1\leq m \leq s}J_m$, such that the condition number in each set equals approximately $1$, i.e., $\max_{j\in J_m}\sigma_j / \min_{j\in J_m}\sigma_j = 1+ o(1)$, and such that the the singular value gaps among $\{J_m\}_{m\in[s]}$ are sufficiently large. We carefully explain how to construct these sets in the complete proof.
We define $\hat \Sigma_{J_m\times J_m},\hat V_{J_m}, V_{J_m}, w_{J_m}$ as the corresponding parts of the related quantities.
We first replace $\hat \Sigma_{r\times r}$ above with $\Sigma_{\times r}$. Indeed, using the variational characterization of the Euclidean norm we have for some $w=(w_{J_1}, \dots, w_{J_s})$, $\|w\|=1$, that 
\begin{align*}
\norm{ \hat \Sigma_{r\times r}\hat V_{1:r}^T (I-VV^T)e_i} & = \sum_{m \in [s]} e_i^T(I-VV^T)\hat V_{J_m} \Sigma_{J_m \times J_m } w_{J_m} \\
& = \sum_{m \in [s]}e_i^T(I-VV^T)\hat V_{J_m} \hat V_{J_m}^T V_{J_m} \Sigma_{J_m \times J_m} w_{J_m}',
\end{align*}
for some $w'\in \mathbb{R}^r$. Since in each set $J_m$ the condition number is bounded by $1+o(1)$ and since $\|(\hat V_{1:r}^TV_{1:r})^{-1}\|=1+o(1)$, 
we can estimate $\|w'\| \leq 1+o(1)$. Thus, we obtain that 

\begin{align*}
&\norm{\hat \Sigma_{r\times r}\hat V_{1:r}^T \br{I-VV^T}e_i}  \\ 
& \leq (1+o(1))\sup_{w\in\mathr^r:\norm{w}=1}  \sum_{m\in[s]}  e_i^T \br{I-VV^T}\br{\hat V_{J_m} \hat V_{J_m}^T -  V_{J_m}  V_{J_m}^T}  V_{J_m} \Sigma_{J_m \times J_m}w_{J_m}.
\end{align*}
The rest of the proof in this section consists of using  spectral perturbation theory to show that $(I-VV^T)\hat V_{1:r}^T$ equals (up to a small order error term) a linear function of the noise matrix $E$. 
Applying Lemma \ref{lem:VV_pert} we show that the above sum is linear in $E$ (up to a $o(\Delta)$ error term ) and obtain 
\begin{align*}
&\norm{\Sigma_{r\times r}\hat V_{1:r}^T \br{I-VV^T}e_i} \\
 = & \sup_{w\in\mathr^r:\norm{w}=1}   \sum_{m\in[s]}  e_i^T  \br{ \sum_{l\in J_m} \frac{1}{\sigma_l} \br{I- VV^T} E^T u_lv_l^TV_{J_m} + S_m }\Sigma_{J_m \times J_m}w_{J_m} \\
= &   \norm{U_{1:r}^T E\br{I - VV^T}e_i } + o(\Delta).
\end{align*}
Hence, summarizing, on the event $\mathcal{F} \cap \mathcal{H}_G$ we bound
\begin{align*}
 \sum_{i=1}^n \sum_{a \neq z_i^*} A_{i,a} 
&\leq k \sum_{i=1}^n \indic{\br{1-o(1)}\Delta \leq 2   \norm{U_{1:r}^T E\br{I - VV^T}e_i }}.
\end{align*}
The tail probability and expectation of $\norm{U_{1:r}^T E\br{I - VV^T}e_i }^2$ are bounded by the tail probability and expectation of a  chi-square distributed random variable with $k$ degrees of freedom, $\chi^2_k$. Thus, there exist $\{\xi_i\}_{i\in[n]}\overset{i.i.d.}{\sim}\chi^2_k$, such that on the event $\mathcal{F} \cap \mathcal{F}'$ 
\begin{align*}
\sum_{i=1}^n \sum_{a \neq z_i^*} A_{i,a} & \leq k \sum_{i=1}^n \indic{\br{1-o(1)}\Delta \leq 2 \sqrt{\xi_i}} .
\end{align*}
The tail probability of the square root of a $\chi^2$ distribution can be bounded by using Borell's inequality and hence we obtain that 
\begin{align*}
\E \sum_{i=1}^n \sum_{a \neq z_i^* } A_{i,a}  \indic{\mathcal{F} \cap \mathcal{F}'} 
\leq nk\ebr{-\br{1-o(1)}{\Delta^2}{/8}}. 
\end{align*}

\paragraph{Step 3 (Sketch of Section \ref{subsubsec:upper_B})} 
We next provide an upper bound on the $B_{i,a}$-terms in \eqref{eq:A+B}, corresponding to small singular values. We have that  $$\langle\hat P_\cd{i}^\tpt , \hat \theta_a^\tpt - \hat \theta^\tpt_{z^*_i}\rangle = \sum_{l=r+1}^k\hat\sigma_l \hat V_{i,l} (\hat u_l^T \hat \theta_a - \hat u_l^T \hat \theta_{z^*_i}),$$ which, up to some constant scalar, can be upper bounded by $\sum_{l=r+1}^k \sqrt{n}|\hat V_{i,l}|$ by construction of $r$ and Weyl's inequality.  Hence, on the event $\mathcal{F}$ we obtain that
\begin{align*}  B_{i,a} : & = \indic{ \gamma\Delta^2 \leq  -  \norm{\hat P_\cd{i}^\tpt - \hat \theta_a^\tpt}^2 + \norm{\hat P_\cd{i}^\tpt - \hat \theta^\tpt_{z^*_i}}^2   } \\ & \leq   \sum_{l=r+1}^k  \indic{c\gamma\Delta^2/k \leq \sqrt{n}\abs{e_i^T \hat v_l}}. \end{align*}
We decompose $e_i^T\hat v_l=e_i^TVV^T\hat v_l+e_i^T(I-VV^T)\hat v_l$. Since, by Lemma \ref{prop:V_pop} $|V_{ij}| \leq \sqrt{k/(n\beta)}$ the first term in this decomposition is negligible, 
leaving $\br{I - VV^T}\hat v_l^T$ as the main term to be analyzed.

We apply Lemma \ref{lem:hat_V_normal} to show that, after normalization, $\br{I - VV^T}\hat v_l^T$ is Haar distributed on the unit sphere spanned by $I-VV^T$. Hence, on an event $\mathcal{T}$, $e_i^T\br{I - VV^T}\hat v_l^T$ has a Gaussian tail and variance at most $3/(n-k)$. This yields
\begin{align*}
\E B_{i,a} \indic{\mathcal{F} \cap \mathcal{T}} &  \leq \sum_{l=r+1}^k  \E \indic{c'\gamma\Delta^2/k \leq \sqrt{n}\abs{e_i^T \br{I-VV^T} \hat v_l}}\indic{\mathcal{T}} \\ & \leq k \ebr{-c''\br{{\gamma \Delta^2}{k^{-1}}}^2}. 
\end{align*}

\paragraph{Step 4 (Sketch of Section \ref{subsubsec:final})}
Summarizing the previous two sections, we obtain that
\begin{align*}
 \mathbb{E}n \ell(\hat z, z) \mathbb{I}(\mathcal{F} \cap \mathcal{H}_G \cap \mathcal{T}) \leq & \sum_{i=1}^n \sum_{a \neq z_i^*} \mathbb{E} ( A_{i,a}+B_{i,a}) \mathbb{I}(\mathcal{F} \cap \mathcal{H}_G \cap \mathcal{T}) \\ \leq &  nk\ebr{-\br{1-o(1)}\frac{\Delta^2}{8}}  + k^2 n\ebr{-c\br{{\gamma \Delta}{k^{-1}}}^2\Delta^2}\\ = &  n\ebr{-\br{1-o(1)}\frac{\Delta^2}{8}}. 
\end{align*}
By Markov's inequality,  with high probability, we achieve
\begin{align*}
\ell(\hat z,z^*) \mathbb{I}(\mathcal{F} \cap \mathcal{H}_G \cap \mathcal{T})  \leq   n\ebr{-\br{1-o(1)}{\Delta^2}{/8}}.
\end{align*}
Finally, a union bound with $\pbr{\mathcal{F} \cap \mathcal{H}_G \cap \mathcal{T}}$ leads to the desired rate for $\ell(\hat z,z^*)$.

\subsection{Proof of Theorem \ref{thm:main}}\label{subsec:proof}
In this section, we are going to give a complete and detailed proof of Theorem \ref{thm:main}. We divide this section into four parts, following the same structure as in the proof sketch (i.e,  Section \ref{subsec:sketch}). In Section \ref{subsubsec:decomposition}, we establish the decomposition $\ell(\hat z, z^*) \leq A + B$. Then in Section \ref{subsubsec:upper_A} and Section \ref{subsubsec:upper_B}, we provide upper bounds on $\E A$ and $\E B$, respectively. Finally in Section \ref{subsubsec:final}, we wrap everything up to achieve the desired rate. Again, throughout the whole proof, we assume the random event $\mathcal{F}$ (defined in  (\ref{eqn:F_def})) holds.

Applying Lemma \ref{lem:equivalence} we obtain that it suffices to bound $\ell(\hat z',z^*)$ where $\hat z'$ is the output of Algorithm \ref{alg:rankk}. 
Indeed, Lemma \ref{lem:equivalence} proves that there exists a label permutation $\phi_0 \in\Phi$ such that $\hat z_i = \phi_0(\hat z_i')$ for all $i\in[n]$. Without loss of generality, we assume that $\phi_0$ is the identity mapping. 
By definition of the $k$-means objective in  (\ref{eqn:kmeans_P}), we have that
\begin{align*}
 \br{\hat z,\cbr{\hat \theta_j}_{j=1}^k} = \argmin_{ z\in [k]^n , \cbr{ \theta_j}_{j=1}^{k} \in \mathr^{k}} \sum_{i\in[n]}\norm{\hat P_\cd{i} - \theta_{z_i}}^2.
 \end{align*}
In particular, $\hat z$ fulfills the {\em local} optimality condition
\begin{align*}
&\hat z_i =\argmin_{j\in[k]} \norm{\hat P_\cd{i} - \hat \theta_j}^2,\forall i\in[n], 
\end{align*}
Hence, assuming without loss of generality that $\phi=\text{Id}$, we obtain that 
\begin{align} \label{eq:loc_opt}
n\ell(\hat z, z^*) 
& = \sum_{i=1}^n \indic{\argmin_{a\in[k]} \norm{\hat P_\cd{i} - \hat \theta_a}^2 \neq z^*_i} \\
& \leq \sum_{i=1}^n \sum_{a\neq z^*_i}\indic{\norm{\hat P_\cd{i} - \hat \theta_a}^2 \leq \norm{\hat P_\cd{i} - \hat \theta_{z^*_i}}^2} \triangleq \sum_{i=1}^n \sum_{a\neq z^*_i} T_{i,a}.
\end{align}
\subsubsection{Decomposing $\ell(\hat z, z^*)$} \label{subsubsec:decomposition}
We decompose $\{\hat P_\cd{i}\}_{i\in[n]},\{\hat \theta_j\}_{j\in[k]}$ into two parts: the first part corresponds to singular values that are above the detection threshold and where $\hat  P_{\cd{i}}$ contains signal and the second part consists of the remainder noise term. 
We define $r\in[k]$ as (with $\sigma_{k+1}:=0$)
\begin{align}
r := \max\cbr{ j \in [k]: \sigma_j  - \sigma_{j+1} \geq \rho \sqrt{n+p}},\label{eqn:r_def}
\end{align}
for a sequence  $\rho \rightarrow \infty $ to be determined later. We note that if  $\Delta / (k^\frac{3}{2}\rho \beta^\frac{1}{2} \br{1+p/n}^\frac{1}{2})\rightarrow\infty$, the set $\cbr{ j \in [k]: \sigma_j  - \sigma_{j+1} \geq \rho \sqrt{n+p}}$ is not empty. Otherwise, this would imply $\sigma_1 \leq k\rho \sqrt{n+p}$ which would contradict Proposition \ref{prop:V_form}.

Thus, $r$ is the largest index in $[k]$ such that the corresponding singular value gap is greater than or equal to  $\rho\sqrt{n+p}$. An immediate implication is 
\begin{align}
\max_{r + 1 \leq j \leq k} \sigma_j \leq k\rho \sqrt{n+p}.\label{eqn:singular_value_small_upper}
\end{align}
We split $\hat U$ into $(\hat U_{1:r} ,\hat U_{(r+1):k})$ where $\hat U_{1:r}=\br{\hat u_1,\ldots, \hat u_{r}}$. 
Recall that $\hat P_\cd{i} = \hat U \hat Y_\cd{i}$  and $\hat \theta_j = \hat U \hat c_j.$ 
We decompose $\hat P_\cd{i} = \hat P_\cd{i}^\tpo + \hat P_\cd{i}^\tpt$, where
\begin{align*}
\hat P_\cd{i}^\tpo = {\hat U_{1:r}\hat U_{1:r}^T}\hat P_\cd{i} 
,\quad \text{and }\hat P_\cd{i}^\tpt  = {\hat U_{\br{r+1}:k}\hat U_{\br{r+1}:k}^T}\hat P_\cd{i}. 
\end{align*}
Similarly, for each $j\in[k]$, we decompose $\hat \theta_j = \hat \theta_j^\tpo + \hat \theta_j^\tpt$, where
\begin{align*}
\hat \theta_j^\tpo =   {\hat U_{1:r}\hat U_{1:r}^T}\hat \theta_j   
,\quad \text{and }\hat \theta_j^\tpt = {\hat U_{\br{r+1}:k}\hat U_{\br{r+1}:k}^T}\hat \theta_j.  
\end{align*}
With this notation and due to the orthogonality of $\cbr{\hat u_l}_{l\in[k]}$, we obtain that 
\begin{align*}
&  T_{i,a}  \leq  \indic{\norm{\hat P_\cd{i}^\tpo + \hat P_\cd{i}^\tpt - \hat \theta_a^\tpo -  \hat \theta_a^\tpt}^2 \leq \norm{\hat P_\cd{i}^\tpo + \hat P_\cd{i}^\tpt - \hat \theta_{z^*_i}^\tpo - \hat \theta_{z^*_i}^\tpt }^2}  \\
& = \indic{ 2\iprod{\hat P_\cd{i}^\tpo - \hat \theta^\tpo_{z^*_i}}{\hat \theta^\tpo_{z^*_i} - \hat \theta^\tpo_{a}} + \norm{\hat \theta^\tpo_{z^*_i} - \hat \theta^\tpo_{a}}^2 \leq  2\iprod{\hat P_\cd{i}^\tpt }{ \hat \theta_a^\tpt - \hat \theta^\tpt_{z^*_i}}  - \norm{\hat \theta_a^\tpt}^2 + \norm{\hat \theta^\tpt_{z^*_i}}^2}.
\end{align*}
We denote by $\rho'' = o(1)$ another sequence which we will specify later. 
We split the indicator function above according to our decomposition and obtain that 
\begin{align*}
  T_{i,a}  \leq  &   \indic{  \norm{\hat \theta^\tpo_{z^*_i} - \hat \theta^\tpo_{a}} -\frac{\rho''\Delta^2  + \norm{\hat \theta^\tpt_{z^*_i}}^2}{\norm{\hat \theta^\tpo_{z^*_i} - \hat \theta^\tpo_{a}}} \leq  2\norm{\hat P_\cd{i}^\tpo - \hat \theta^\tpo_{z^*_i}} } \\ &  + \indic{\rho''\Delta^2 \leq 2\iprod{\hat P_\cd{i}^\tpt }{ \hat \theta_a^\tpt - \hat \theta^\tpt_{z^*_i}}  }  =:A_{i,a}+B_{i,a}
\end{align*}
where we also used the Cauchy-Schwarz inequality. 
We now consider $A_{i,a}$ and $B_{i,a}$ separately. 
\subsubsection{Upper Bounds on $\E A_{i,a}$}\label{subsubsec:upper_A}
By Lemma \ref{lem:theta_dist}, we have on the event $\mathcal{F}$ that 
$\max_{j\in[k]}\|\hat \theta_j - \theta^*_{\phi'(j)}\| \leq  8\sqrt{2}\sqrt{\beta^{-1}k^2\br{1+p/n}}$
for some label permutation mapping $\phi'\in\Phi$. Without loss of generality, we assume again that $\phi'=\text{Id}$. 
Define $\hat Z\in\cbr{0,1}^{n\times k}$ to be the estimated label matrix, i.e.,  
$
\hat Z_{i,j}=\indic{\hat z_i = j}. 
$
With this notation and by definition of the $k$-means objective we obtain that 
\begin{align} \label{eq:thetaex}
\hat \theta_j = \frac{\sum_{\hat z_i = j} \hat P_\cd{i}}{\sum_{\hat z_i = j} 1} = \frac{ \hat P \hat Z_\cd{j} }{\abs{\cbr{i\in[n]:\hat z_i =j}}}=  \frac{\sum_{l\in[k]} \hat \sigma_l \hat u_l \hat v_l^T\hat Z_\cd{j}}{\abs{\cbr{i\in[n]:\hat z_i =j}}}. 
\end{align}
Hence, using the above, we obtain that 
\begin{align*}
\abs{\iprod{\hat u_l}{\hat \theta_j}} = \frac{\abs{\hat \sigma_l \hat v_l^T \hat Z_\cd{j}}}{\abs{\cbr{i\in[n]:\hat z_i =j}}}  \leq\frac{\hat \sigma_l \norm{\hat v_l} \norm{\hat Z_\cd{j}} }{\abs{\cbr{i\in[n]:\hat z_i =j}}}  = \frac{\hat \sigma_l}{\sqrt{\abs{\cbr{i\in[n]:\hat z_i =j}}}} .
\end{align*}
By  (\ref{eqn:singular_value_small_upper}) and Lemma \ref{lem:singular_diff}, we have on the event $\mathcal{F}$ that 
\begin{align}
\max_{r+1\leq j\leq k} \hat \sigma_j \leq \sqrt{2}(\sqrt{n}+\sqrt{p})+ \max_{r+1\leq j\leq k}  \sigma_j  \leq \br{k\rho + 4}\sqrt{n+p}.\label{eqn:hat_sigma_max}
\end{align}
By Lemma \ref{lem:theta_dist} we have that  $\abs{\cbr{i\in[n]:\hat z_i = j}} \geq \frac{\beta n}{2k}$ and  thus we obtain 
\begin{align}
\max_{j\in[k]}\max_{r+1\leq l\leq k} \abs{\iprod{\hat u_l}{\hat \theta_j} } \mathbb{I}(\mathcal{F})\leq   \br{k\rho + 4} \sqrt{\frac{2k}{\beta}\br{1+\frac{p}{n}}}.\label{eqn:iprod_hat_u_hat_theta}
\end{align}
Consequently, we bound, working on the event $\mathcal{F} $
\begin{align}
\max_{j\in[k]} \norm{\hat \theta_j^\tpt}^2&=  \max_{j\in[k]}   \sum_{r+1\leq l\leq k} \iprod{\hat u_l}{\hat \theta_j}^2\nonumber\\
&\leq  \frac{2k^2}{\beta}\br{1+\frac{p}{n}}\br{k\rho + 4}^2.\label{eqn:hat_theta_2}
\end{align}

Applying Lemma \ref{lem:theta_dist} we have on the event $\mathcal{F}$ for any $a \neq b $ that  $$\|{\hat \theta_{b} - \hat \theta_{a}}\| \geq {\norm{ \theta_{b}^* -  \theta_{a}^*} - \|{\hat \theta_{b} -  \theta^*_{b}}\|  - \|{ \theta^*_{a} - \hat \theta_{a}}}\| 
\geq  \Delta - 16\sqrt{2}\sqrt{\beta^{-1}k^2\br{1+p/n}}.$$
Hence, using also  (\ref{eqn:hat_theta_2}), we have on the event $\mathcal{F}$ that
\begin{align}
\min_{a,b\in[k]:a\neq b} \norm{\hat \theta_{b}^\tpo - \hat \theta_{a}^\tpo} &\geq  \min_{a,b\in[k]:a\neq b} \br{\norm{\hat \theta_{b} - \hat \theta_{a}}  - \norm{\hat \theta_a^\tpt}  - \norm{\hat \theta_b^\tpt}} \nonumber\\
&\geq  \Delta - \br{16\sqrt{2} + 2\sqrt{2}\br{k\rho + 4}}\sqrt{\beta^{-1}k^2\br{1+\frac{p}{n}}}.\label{eqn:hat_theta_dist_1}
\end{align}
Therefore, by the above, we obtain that 
\begin{align*}  A_{i,a}\mathbb{I}(\mathcal{F}) \leq &  \mathbb{I} \Bigg\{  \br{  \Delta - \br{16\sqrt{2} + 2\sqrt{2}\br{k\rho + 6}}\sqrt{\beta^{-1}k^2\br{1+\frac{p}{n}}}} \\ & -\frac{  \rho''\Delta^2   +  \frac{2k^2}{\beta}\br{1+\frac{p}{n}}\br{k\rho + 4}^2}{  \Delta - \br{16\sqrt{2} + 2\sqrt{2}\br{k\rho + 4}}\sqrt{\beta^{-1}k^2\br{1+\frac{p}{n}}}}
 \leq  2\norm{\hat P_\cd{i}^\tpo - \hat \theta^\tpo_{z^*_i}} \Bigg\}\mathbb{I}(\mathcal{F}). \end{align*} 
For simplicity, define
\begin{align*}
\eta := \sqrt{1 + p/n}.
\end{align*}
Since by construction $\rho\rightarrow\infty$ and by assumption $\Delta /(k^2\rho\beta^{-1/2}\eta) \rightarrow\infty$, there exists some constant $c_1>0$, such that the above  can be simplified into
\begin{align*}
 &  A_{i,a}\mathbb{I}(\mathcal{F})  \leq \indic{\br{1-c_1 \rho'' - \frac{c_1k^2\rho\beta^{-\frac{1}{2}} \eta }{\Delta}} \Delta \leq 2\norm{\hat P_\cd{i}^\tpo - \hat \theta^\tpo_{z^*_i}} }\mathbb{I}(\mathcal{F})
 \end{align*} 
Still working on the event $\mathcal{F}$, we further bound
\begin{align*}
& \norm{\hat P_\cd{i}^\tpo - \hat \theta^\tpo_{z^*_i}}   \leq  \norm{\hat P_\cd{i}^\tpo - \hat U_{1:r}\hat U_{1:r}^T \theta^*_{z^*_i}} +  \norm{ \hat \theta^\tpo_{z^*_i} -  \hat U_{1:r}\hat U_{1:r}^T \theta^*_{z^*_i}} \\
&\leq \norm{\hat P_\cd{i}^\tpo - \hat U_{1:r}\hat U_{1:r}^T \theta^*_{z^*_i}}  +  \norm{ \hat \theta_{z^*_i} - \theta^*_{z^*_i}}\leq  \norm{\hat P_\cd{i}^\tpo - \hat U_{1:r}\hat U_{1:r}^T \theta^*_{z^*_i}}  + 8\sqrt{2} \sqrt{\beta^{-1} k^2 \br{1+\frac{p}{n}}},
\end{align*}
where the last inequality is due to Lemma \ref{lem:theta_dist}. Since $\theta^*_{z^*_i} = P_\cd{i}$, we have that  $\hat P_\cd{i}^\tpo - \hat U_{1:r}\hat U_{1:r}^T \theta^*_{z^*_i} = ({\hat U_{1:r}\hat U_{1:r}^T \hat P - \hat U_{1:r}\hat U_{1:r}^T P})e_i$. Thus, we obtain that 
\begin{align*} 
\hat P_\cd{i}^\tpo - \hat U_{1:r}\hat U_{1:r}^T \theta^*_{z^*_i} 
& = \hat U_{1:r}\hat U_{1:r}^T \br{\hat P - P}VV^Te_i +  \hat U_{1:r}\hat U_{1:r}^T \hat P\br{I-VV^T}e_i.
\end{align*}
We first bound $\hat U_{1:r}\hat U_{1:r}^T (\hat P - P)VV^Te_i $.
Indeed, by Proposition \ref{prop:V_form} and Lemma \ref{lem:theta_dist} we have on the event $\mathcal{F}$ that
\begin{align*}
\| \hat U_{1:r}\hat U_{1:r}^T (\hat P - P)VV^Te_i \| \leq \|\hat P-P\|_F  \|V^Te_i\|  \leq 4 \sqrt{\beta^{-1} k^2}\br{1+\sqrt{\frac{p}{n}}}
\end{align*}
Thus, there exists some constant $c_2 > 0$ such that 
\begin{align*}
    A_{i,a}\mathbb{I}(\mathcal{F}) & \leq  \indic{\br{1-c_1 \rho'' - \frac{c_2k^2\rho\beta^{-\frac{1}{2}} \eta }{\Delta}} \Delta \leq 2 \| \hat U_{1:r}\hat U_{1:r}^T \hat P(I-VV^T)e_i \|  }\mathbb{I}(\mathcal{F})
    \\ 
    & =  \indic{\br{1-c_1 \rho'' - \frac{c_2k^2\rho\beta^{-\frac{1}{2}} \eta }{\Delta}} \Delta \leq 2\norm{ \hat \Sigma_{r\times r}\hat V_{1:r}^T \br{I-VV^T}e_i} }\mathbb{I}(\mathcal{F}),
\end{align*}
where we define $\hat \Sigma_{r\times r}=\text{diag}\{\hat \sigma_1,\ldots, \hat \sigma_{r}\}$ and $\hat V_{1:r} =\br{\hat v_1,\ldots, \hat v_r}$. We define the corresponding population counterparts analogue.


For any unit vector $w\in\mathr^r$, define $w' = \Sigma_{r\times r}^{-1} \hat \Sigma_{r\times r} w$. This yields the identity $\Sigma_{r\times r}w' =  \hat \Sigma_{r\times r} w$. By definition of $r$ and Lemma \ref{lem:singular_diff} we obtain that 
%
\begin{align*}
\max_{j\in[r]} {\frac{\hat \sigma_j}{\sigma_j}} \leq \max_{j\in[r]}  {\frac{\sigma_j + 4\sqrt{n+p}}{\sigma_j}} \leq  1+6\rho^{-1}
\end{align*}
and therefore $\norm{w'}\leq 1+6\rho^{-1}$.
Thus, using the variational characterization of the Euclidean norm, we bound 
\begin{align*}
A_{i}\indic{\mathcal{F}} & \leq  \indic{\br{1-  c_1\rho'' - \frac{c_2k^2\rho\beta^{-\frac{1}{2}} \eta }{\Delta}} \Delta\leq 2\sup_{w:\norm{w}\leq 1}  e_i^T \br{I-VV^T}  \hat V_{1:r} \hat \Sigma_{r\times r}w}\indic{\mathcal{F}} \\
& \leq  \sum_{i\in[n]} \indic{\frac{1-  c_1\rho'' - \frac{c_2k^2\rho\beta^{-\frac{1}{2}} \eta }{\Delta}}{1+6\rho^{-1}} \Delta\leq 2 \|  \Sigma_{r\times r} \hat V_{1:r}^T (I-VV^T)e_i \| }\indic{\mathcal{F}}
\end{align*}
We further investigate $ e_i^T \br{I-VV^T}\hat V_{1:r}  \Sigma_{r\times r}w$. First we partition the leading $[r]$ singular values. Define $s$ as 
\begin{align}\label{eqn:J_m_def}
s:=\left |\cbr{l\in[r]: \frac{\sigma_l - \sigma_{l+1}}{\sigma_{l+1}}\geq \frac{1}{\rho' k}} \right | ,
\end{align}
for some $\rho'\rightarrow\infty$ whose value will be specified later. We denote its entries by $j'_1 < j'_2 < \ldots  < j'_s$.  Due to  (\ref{eqn:singular_value_small_upper}) we have that $j'_s = r$. We define $j'_0 = 0$,
\begin{align*}
j_m = j'_{m-1} + 1 ,~ m\in[s]
\end{align*}
and split $[r]$ into disjoint sets $\cbr{J_m}_{m=1}^s$, where $J_m = \cbr{j_m,j_m + 1,\ldots, j'_m}$.
This partition has the following properties: 
\begin{itemize}
\item   Defining the singular value gaps 
$g_m  :=\min \cbr{ \sigma_{j_{m-1}'} -\sigma_{j_{m}}  ,  \sigma_{j_m'} -\sigma_{j_{m+1}}},~ m\in[s],$
with $j_{s+1} = r+1$ and $\sigma_0 = +\infty$, we obtain by  (\ref{eqn:J_m_def}) for any $m\in[s-1]$ that 
$$ \sigma_{j_m' } - \sigma_{j_{m+1}}  = \sigma_{j_m' } - \sigma_{j_m'+1}  \geq  \frac{\sigma_{j_m'+1}}{\rho'k}\nonumber \geq  \frac{\sigma_r}{\rho'k} \geq \frac{\rho \sqrt{n+p}}{\rho'k}.
$$
Hence, we obtain that 
\begin{align}\label{eqn:sigma_j_m_gap}
\min_{m\in[s]} g_m \geq  \frac{\rho \sqrt{n+p}}{\rho'k}.
\end{align}
\item The set defined in  (\ref{eqn:J_m_def}) has an alternative representation, i.e., $$ \cbr{l\in[r]: \frac{\sigma_l - \sigma_{l+1}}{\sigma_{l+1}}\geq \frac{1}{\rho' k}}  = \left  \{l\in[r]: \frac{\sigma_l}{ \sigma_{l+1}} > 1 + \frac{1}{\rho’k} \right \}.$$ 
Therefore, and since $\rho'\rightarrow\infty$, we obtain that 
\begin{align}\label{eqn:sigma_j_m_ratio}
\max_{m\in[s]} \frac{\sigma_{j_m}}{\sigma_{j'_m}} & \leq \br{1 + \frac{1}{\rho'k}}^{\abs{J_m}}\leq \br{1 + \frac{1}{\rho'k}}^k \leq 1+ \frac{2}{\rho'}.
\end{align}
\item Due to  (\ref{eqn:J_m_def}) we have that 
$$\max_{m\in[s]}  \frac{\sigma_{j'_m}}{\sigma_{j'_m} - \sigma_{j'_m+1}}   \leq   \max_{m\in[s]} \frac{1 +  \sigma_{j'_m+1}}{\sigma_{j'_m} - \sigma_{j'_m+1}} \leq 1+ \rho'k.$$
Hence, using also  (\ref{eqn:sigma_j_m_ratio}) and since $\rho'\rightarrow\infty$, we obtain that
\begin{align}\label{eqn:sigma_j_diff_g_ratio}
\max_{m\in[s]} \frac{\sigma_{j_m} - \sigma_{j'_m}}{g_m } &\leq \frac{2}{\rho'}  \max_{m\in[s]}  \frac{\sigma_{j'_m}}{\sigma_{j'_m} - \sigma_{j'_m+1}} \leq  \frac{2}{\rho'}  \br{1+ \rho'k} \leq 3k,
\end{align}
and 
\begin{align}\label{eqn:sigma_j_max_g_ratio}
\max_{m\in[s]} \frac{\sigma_{j_m} }{g_m } &\leq  \br{1 + \frac{2}{\rho'}  }\max_{m\in[s]}  \frac{\sigma_{j'_m}}{\sigma_{j'_m} - \sigma_{j'_m+1}}  \leq 1+ 2\rho'k.
\end{align}
\end{itemize}

Now, consider any fixed $w\in\mathr^r$.
For $m\in[s]$, we denote $\hat V_{J_m} = (\hat v_{j_m},\ldots, \hat v_{j'_m})$,  $ V_{J_m} = ( v_{j_m},\ldots,  v_{j'_m})$, $\Sigma_{J_m \times J_m} = \text{diag}\{\sigma_{j_m},\ldots, \sigma_{j'_m}\} $, and $w_{J_m} = (w_{j_m},\ldots,  w_{j'_m}) $.
Applying this notation, we have that 
\begin{align}\label{eqn:w_into_w_m}
 e_i^T \br{I-VV^T}\hat V_{1:r}  \Sigma_{r\times r}w & =  \sum_{m\in[s]}  e_i^T \br{I-VV^T}\hat V_{J_m}  \Sigma_{J_m\times J_m}w_{J_m}.
\end{align}
For any $m\in [s]$, by the Davis-Kahan-Wedin $sin(\Theta)$ Theorem (see Lemma \ref{lem:wedin}), there exists an orthonormal matrix $O_m \in\mathr^{\abs{J_m}\times \abs{J_m}}$ such that
\begin{align}
\opnorm{\hat V_{J_m}  - V_{J_m}O_m }&\leq \sqrt{2}\opnorm{\hat V_{J_m} \hat V_{J_m}^T - V_{J_m} V_{J_m}^T}\nonumber \\
&\leq  \frac{4\sqrt{2} \opnorm{E}}{g_m} 
\leq \frac{16\sqrt{2}\rho'k}{\rho},\label{eqn:V_J_m_Davis}
\end{align}
where we use  (\ref{eqn:sigma_j_m_gap}) in the last inequality. 
Moreover, we have that 
\begin{align}\label{eqn:hat_V_V_pert}
\opnorm{\hat V_{J_m}^T V_{J_m}O_m - I} \leq \opnorm{\hat V_{J_m} -V_{J_m}O_m}\leq\frac{16\sqrt{2}\rho'k}{\rho}.
\end{align}
and hence, choosing $\rho$ and $\rho'$ such that  $\rho/(\rho' k) > 16\sqrt{2}$, we obtain that  $V_{J_m}^T \hat V_{J_m}$ is invertible and \begin{align} \label{eq:invlower} \|( V_{J_m}^T \hat V_{J_m})^{-1}\|\leq \left ({ 1-\frac{16\sqrt{2}\rho'k}{\rho}} \right )^{-1}.\end{align} 
Now, for fixed $w_{J_m}$ we define 
\begin{align}
w'_{J_m} = \Sigma_{J_m \times J_m}^{-1} \br{\hat V_{J_m}^T V_{J_m}}^{-1} \Sigma_{J_m \times J_m} w_{J_m},\forall m\in[s]\label{eqn:w_prime_def}.
\end{align}
Plugging the above into  (\ref{eqn:w_into_w_m}), we obtain that 
\begin{align*}
e_i^T \br{I-VV^T}\hat V_{1:r}  \Sigma_{r\times r}w & =  \sum_{m\in[s]}  e_i^T \br{I-VV^T}\hat V_{J_m} \hat V_{J_m}^T V_{J_m} \Sigma_{J_m \times J_m}w'_{J_m}.
\end{align*}
By definition of $w_{J_m}'$ we have that 
\begin{align*}
\max_{m\in[s]}\frac{\norm{w_{J_m}'}}{\norm{w_{J_m}}}& \leq  \max_{m\in[s]} \opnorm{ \Sigma_{J_m \times J_m}^{-1}\br{\hat V_{J_m}^T V_{J_m}}^{-1} \Sigma_{J_m \times J_m} } \\
& \leq   \br{1+2\rho'^{-1}}\left ({ 1-\frac{16\sqrt{2}\rho'k}{\rho}} \right )^{-1},
\end{align*}
where we used in the last inequality that  $ \max_{m\in[s]} \opnorm{ \Sigma_{J_m \times J_m}^{-1}} \opnorm{\Sigma_{J_m \times J_m}}  \leq 1+2\rho'^{-1}$ by  (\ref{eqn:sigma_j_m_ratio}) and the upper bound \eqref{eq:invlower}. 
Hence, using also that $\br{I- VV^T}V_{J_m}=0$, we obtain that 
\begin{align*}
&\sup_{w:\norm{w}\leq 1} w^T \Sigma_{r\times r} {\hat V_{1:r}}^T\br{I - VV^T}e_i \\
& \leq    \frac{1+2\rho'^{-1}}{1- 16\sqrt{2}\rho'k\rho^{-1}}  \sup_{w:\norm{w}\leq 1} \sum_{m\in[s]}  e_i^T \br{I-VV^T}\br{\hat V_{J_m} \hat V_{J_m}^T -V_{J_m}V_{J_m}^T} V_{J_m} \Sigma_{J_m \times J_m}w_{J_m}.
\end{align*}
We further evaluate the term on the right hand side above. Applying Lemma \ref{lem:VV_pert}, 
we obtain that 
\begin{align*}
&e_i^T \br{I-VV^T}\br{\hat V_{J_m} \hat V_{J_m}^T -V_{J_m}V_{J_m}^T} V_{J_m} \Sigma_{J_m \times J_m}w_{J_m} \\
  = &  \sum_{l\in J_m} w_l e_i^T \br{I - VV^T} E^Tu_l +  e_i^T  \E S_m \Sigma_{J_m \times J_m}w_{J_m}  + e_i^T  \br{S_m - \E S_m} \Sigma_{J_m \times J_m}w_{J_m}.
\end{align*}
We next show that $\E S_m = 0$. 
Indeed, we have that 
\begin{align*}
\E S_m  
 & =  \br{I- VV^T}\E  \br{\sum_{j\in J_m} \hat v_j \hat v_j^T} V_{J_m} = \sum_{j\in J_m} \E \br{\br{I- VV^T}\hat v_j } \br{ V_{J_m}^T \hat v_j}^T\\
  & = \sum_{j\in J_m} \E \br{ \frac{\br{I- VV^T}\hat v_j}{\norm{\br{I- VV^T}\hat v_j}} } \br{ \norm{\br{I- VV^T}\hat v_j} V^T_{J_m} VV^T  \hat v_j}^T.
\end{align*}
Applying Lemma \ref{lem:hat_V_normal} we obtain that $ \br{I- VV^T}\hat v_j/{\norm{\br{I- VV^T}\hat v_j}}$ and $\norm{\br{I- VV^T}\hat v_j} V^T_{J_m} VV^T  \hat v_j$ are independent. 
Hence, using Lemma \ref{lem:hat_V_normal} again, 
we obtain that 
\begin{align*}
\E S_m& = \sum_{j\in J_m}  \E \br{\frac{\br{I- VV^T}\hat v_j}{\norm{\br{I- VV^T}\hat v_j}} }  \E \br{\norm{\br{I- VV^T}\hat v_j} V^T_{J_m} VV^T  \hat v_j}^T =0.
\end{align*}
Hence, we obtain that 
\begin{align*}
&\sup_{w\in\mathr^r:\norm{w}\leq 1} \sum_{m\in[s]} e_i^T \br{1-VV^T}\br{\hat V_{J_m} \hat V_{J_m}^T -V_{J_m}V_{J_m}^T} V_{J_m} \Sigma_{J_m \times J_m}w_{J_m}  \\
& \leq  \sup_{w\in\mathr^r:\norm{w}\leq 1}  e_i^T\br{I - VV^T} E^TU_{1:r} w 
 +\sup_{w\in\mathr^r:\norm{w}\leq 1}\sum_{m\in[s]}  e_i^T  \br{S_m - \E S_m}\Sigma_{J_m \times J_m}w_{J_m}.
\end{align*}
Summarizing, we obtain that
\begin{align*}
 A_{i,a} & \leq  \mathbb{I}\Bigg\{\frac{1- 16\sqrt{2}\rho'k\rho^{-1}}{\br{1+6\rho^{-1}}\br{1+2\rho'^{-1}}}\br{1-  c_1\rho'' - \frac{c_2k^2\rho\beta^{-\frac{1}{2}} \eta }{\Delta}} \Delta  \\
&\quad \leq 2\norm{U_{1:r}^T E\br{I - VV^T}e_i } 
 +\sup_{w\in\mathr^r:\norm{w}\leq 1}\sum_{m\in[s]}  e_i^T \br{S_m - \E S_m}\Sigma_{J_m \times J_m}w_{J_m} \Bigg\}.
\end{align*}
We next bound the higher order perturbation term on the right hand side. 
Applying Lemma \ref{lem:E_opnorm} and by construction of the partition we obtain that $g_m \geq 8\E\norm{E}$,  and hence we can apply Lemma \ref{lem:VV_pert}. Note that $\norm{\Sigma_{J_m \times J_m}w_{J_m}e_i^T}_* =\norm{\Sigma_{J_m \times J_m}w_{J_m}} \norm{e_i^T }  \leq \sigma_{j_m} \norm{w_{J_m}}.
$ Together with  (\ref{eqn:sigma_j_m_gap}), (\ref{eqn:sigma_j_diff_g_ratio}) and (\ref{eqn:sigma_j_max_g_ratio}), for some constant $c_0>0$,
we have with probability at least $1-2e^{-(\Delta^2 \wedge n)}$ that
\begin{align*}
&\abs{e_i^T  \br{S_m - \E S_m}\Sigma_{J_m \times J_m}w_{J_m}} \\ \leq &  c_0 \br{1 + \frac{\sigma_{j_m} - \sigma_{j_m'}}{g_m}} \frac{\Delta }{g} \br{\frac{\sqrt{n+p} }{g}}\sigma_{j_m} \norm{w_{J_m}}\\
  \leq &  16c_0\rho^{-1}k^3\rho'^2 \Delta \norm{w_{J_m}}.
\end{align*}
Taking a union bound over $J_m$ and since $\sum_m \|w_{J_m}\| \leq \sqrt{k} \|w\|=\sqrt{k}$ we obtain with probability at least $1-2k\ebr{-(\Delta^2 \wedge n)}$  that 
\begin{align*}
\sum_{m\in[s]}  e_i^T  \br{S_m - \E S_m} \Sigma_{J_m \times J_m}w_{J_m} \leq 16c_0\rho^{-1}k^\frac{7}{2}\rho'^2 \Delta. 
\end{align*}
By applying a standard $\varepsilon$-net argument with  a union bound, we  obtain  with probability at least $1-2ke^k\ebr{-(\Delta^2 \wedge n)}$ that
\begin{align} \label{eq:unism}
\sup_{w\in\mathr^r:\norm{w}\leq 1}\sum_{m\in[s]}  e_i^T  \br{S_m - \E S_m} \Sigma_{J_m \times J_m}w_{J_m} \leq 32c_0\rho^{-1}k^\frac{7}{2}\rho'^2 \Delta.
\end{align}
We denote by $\mathcal{H}_{i}$ the event where \eqref{eq:unism} above holds and note that $\mathbb{P}(\mathcal{H}_i)\geq 1-2ke^k\ebr{-(\Delta^2 \wedge n)}$. 
To avoid that $\Delta \wedge n$ (instead of $\Delta$) appears in the convergence rate we further introduce the global event $$\mathcal{H}_G:= \left \{ \{ \Delta > \sqrt{n} \} \bigcap_{i=1}^n \mathcal{H}_i  \right \} \cup \{ \Delta \leq  \sqrt{n} \}  $$
and note that 
$\mathbb{P}(\mathcal{H}_G) \geq 1-2nke^{-n+k}$. 
%
We are finally ready to bound $A_{i,a}$. Indeed, by the above we obtain that 
\begin{align*}
 \mathbb{E}A_{i,a}\mathbb{I}(\mathcal{F} \cap \mathcal{H}_G)  \leq &  \mathbb{E} \mathbb{I}\Bigg\{\frac{\br{1- 16\sqrt{2}\rho'k\rho^{-1}}\br{1-  c_1\rho'' - \frac{c_2k^2\rho\beta^{-\frac{1}{2}} \eta }{\Delta}}}{\br{1+6\rho^{-1}}\br{1+2\rho'^{-1}}} \Delta \\ & - 32c_0\rho^{-1}k^\frac{7}{2}\rho'^2 \Delta     \leq 2\norm{U_{1:r}^T E\br{I - VV^T}e_i}\Bigg\} \\& +\mathbb{E}\mathbb{I}(\mathcal{H}_i \cap  \{ \Delta < \sqrt{n} \} ). 
\end{align*}
 We observe that $U_{1:r}^T E\br{I - VV^T}e_i \sim \mathn\br{0,\norm{\br{I - VV^T}e_i}^2I_{r\times r}}$. Moreover, since $\norm{\br{I - VV^T}e_i}\leq 1$, we have that  $$\mathbb{P}\left (  \norm{U_{1:r}^T E\br{I - VV^T}e_i}^2 > t \right ) \leq \mathbb{P} (\xi_i > t), $$ 
where by $\xi_i$ we denote a chi-square distributed random variable with $k$ degrees of freedom. 
%
Hence, assuming additionally that $\rho'\rightarrow\infty$, $\rho/(k^{7/2}\rho'^2)\rightarrow\infty$ and $\Delta /(k^2\rho\beta^{-1/2} \eta )\rightarrow\infty$, there exists a constant $c_3>0$, such that 
\begin{align*}
& \mathbb{E} A_{i,a}\mathbb{I}(\mathcal{F} \cap \mathcal{H}_G) \\  \leq &  \indic{\br{1 - c_3 \rho'' - \frac{c_3k^2\rho\beta^{-\frac{1}{2}} \eta }{\Delta} - \frac{c_3k^\frac{7}{2}\rho'^2}{\rho}}\Delta \leq  2 \sqrt{\xi_i}}+2ke^{-\Delta^2+k} 
 \\ \leq & \ebr{-\frac{1}{8}\br{1 - c_3 \rho'' - \frac{c_3k^2\rho\beta^{-\frac{1}{2}} \eta }{\Delta} - \frac{c_3k^\frac{7}{2}\rho'^2}{\rho}-\frac{2\sqrt{k}}{\Delta}}^2\Delta^2}+2ke^{-\Delta^2+k},
\end{align*}
where we used Jensen's inequality and Borell's inequality (e.g. Theorem 2.2.7 in \cite{GineNickl16book})  to bound $\mathbb{P}(\sqrt{\xi_i} > t)\leq \exp(-(t-\sqrt{k})^2)$. 

\subsubsection{Upper Bounds on $\E B_{i,a}$}\label{subsubsec:upper_B}
We now bound
$$B_{i,a}:=\indic{\rho''\Delta^2 \leq 2\iprod{\hat P_\cd{i}^\tpt }{ \hat \theta_a^\tpt - \hat \theta^\tpt_{z^*_i}}  }.$$
We recall that $\hat P_\cd{i}^\tpt =({\hat U_{\br{r+1}:k}\hat U_{\br{r+1}:k}^T})\hat P_\cd{i} = \sum_{l=r+1}^k \hat u_l \hat Y_{l,i} =   \sum_{l=r+1}^k \hat u_l \hat\sigma_l\hat V_{i,l}$  and $ \hat \theta_a^\tpt - \hat \theta^\tpt_{z^*_i} = ({\hat U_{\br{r+1}:k}\hat U_{\br{r+1}:k}^T})({\hat \theta_a - \hat \theta_{z^*_i}})$. Hence, we obtain that 
\begin{align*}
\iprod{\hat P_\cd{i}^\tpt }{ \hat \theta_a^\tpt - \hat \theta^\tpt_{z^*_i}} 
=  \sum_{l=r+1}^k\hat\sigma_l \hat V_{i,l} \br{\hat u_l^T \hat \theta_a - \hat u_l^T \hat \theta_{z^*_i}}.
\end{align*}
Note that $|\hat u_l^T \hat \theta_a - \hat u_l^T \hat \theta_{z^*_i}| \leq 2\max_{j\in[k]}\max_{r+1\leq l\leq k} |{\langle\hat u_l},{\hat \theta_j} \rangle|$.
Using   (\ref{eqn:hat_sigma_max}) and (\ref{eqn:iprod_hat_u_hat_theta}) , we have that
\begin{align}\label{eqn:iprod_hat_P_hat_theta}
&\abs{\iprod{\hat P_\cd{i}^\tpt }{ \hat \theta_a^\tpt - \hat \theta^\tpt_{z^*_i}} } 
\leq 2\br{k\rho + 4}^2 \sqrt{\frac{2nk}{\beta}\br{1+\frac{p}{n}}^2}\sum_{l=r+1}^k\abs{\hat V_{i,l}}
\end{align}
and hence we bound
%
\begin{align*}
\sum_{a \neq  z_i^*} B_{i,a}\mathbb{I}(\mathcal{F} \leq & k \indic{  \rho''\Delta^2  \leq 4\br{k\rho + 4}^2 \sqrt{\frac{2nk}{\beta}\br{1+\frac{p}{n}}^2}\sum_{l=r+1}^k\abs{\hat V_{i,l}} } \\
\leq & k \sum_{l=r+1}^k \indic{  \rho''\Delta^2  \leq 4k\br{k\rho + 4}^2 \sqrt{\frac{2nk}{\beta}\br{1+\frac{p}{n}}^2}\abs{\hat V_{i,l}} }=:k \sum_{l=r+1}^k C_{i,l}.
\end{align*}

We bound each $C_{i,l}$ separately, by showing that $\hat V_{i,l}$ is, approximately, univariate Gaussian with variance $1/n$. 
We first apply Proposition \ref{prop:V_pop} to obtain that 
\begin{align*}
&\abs{\hat V_{i,l}} 
\leq \|V^Te_i\| + \abs{e_i^T {\br{I - VV^T}\hat v_l} }  \leq  \sqrt{\beta^{-1}k/n}+ \abs{e_i^T {\br{I - VV^T}\hat v_l} }. 
\end{align*}
Hence, assuming that $\Delta^2\rho''/(k^4 \rho^2 \beta^{-1} (1+p/n) )$ is large enough and afterwards applying Lemma \ref{lem:hat_V_normal}, we obtain that for some constant $c_4>0$
\begin{align*}
C_{i,l} & \leq  \indic{c_4  \frac{\rho''\Delta^2}{k^\frac{7}{2}\rho^2 \beta^{-\frac{1}{2}} (1+\frac{p}{n})}  \leq  \sqrt{n}\frac{\abs{e_i^T {\br{I - VV^T}\hat v_l}}}{\norm{ {\br{I - VV^T}\hat v_l}}}   } \\
& \overset{d}{=} \indic{c_4  \frac{\rho''\Delta^2}{k^\frac{7}{2}\rho^2 \beta^{-\frac{1}{2}} (1+\frac{p}{n})}  \leq  \sqrt{n}\frac{\abs{e_i^T {\br{I - VV^T}\zeta_{i,l} }}}{\norm{ {\br{I - VV^T}\zeta_{i,l} }}}    },
\end{align*}
where $\overset{d}{=}$ denotes equality in distribution and where $\zeta_{i,l} {\thicksim} \mathn(0,I_n)$. 
We next provide a lower bound for the denominator above. Indeed, since $ (I - VV^T)\zeta_{i,l} \thicksim \mathcal{N}(0, (I-VV^T))$, we see that $\|(I-VV^T)\zeta_{i,l}\|^2$ is chi-square distributed with at least $n-k$ degrees of freedom. 
Hence, using tail-bounds for the lower tail of chi-square distributed random variables (e.g. Lemma 1 in \cite{LaurentMassart00}) we obtain that the event $\mathcal{T}$ defined below occurs with high probability, i.e.,
\begin{align}
    \label{eq:lowertail}
    \mathbb{P} \left ( \mathcal{T} \right ) := \mathbb{P} \left ( \bigcap_{i,l} \left \{ \|(I-VV^T)\zeta_{i,l}\|^2 \geq \frac{(n-k)}{3} \right \}  \right ) \geq 1-nk\exp \left ( - \frac{(n-k)}{9}\right ).
\end{align}
Hence, working on the event $\mathcal{T}\cap \mathcal{F}$, we bound
\begin{align*}
\mathbb{E} C_{i,l} \mathbb{I}(\mathcal{T} \cap \mathcal{F}) & \leq \mathbb{E} \indic{c_4  \frac{\rho''\Delta^2}{k^\frac{7}{2}\rho^2 \beta^{-\frac{1}{2}} (1+\frac{p}{n})}  \sqrt{\frac{n-k}{3n}} \leq \abs{e_i^T {\br{I - VV^T}\zeta_{i,l} }}   } \\ 
 &\leq 2 \ebr{-\frac{1}{2} \left ( c_4 \frac{\rho''\Delta}{k^\frac{7}{2}\rho^2 \beta^{-\frac{1}{2}} (1+\frac{p}{n})}  \sqrt{\frac{n-k}{3n}} \right )^2 \Delta^2} ,
\end{align*}
where we used  that $e_i^T(I-VV^T \zeta_{i,l})$ is univariate Gaussian with variance bounded by $1$. 
\subsubsection{Obtaining the final Result}\label{subsubsec:final}
 Combining the above upper bounds together, we have that 
\begin{align*}
& \E \ell(\hat z,z^*) \indic{\mathcal{F} \cap \mathcal{H}_G \cap\mathcal{T}} \\  \leq & \frac{1}{n}\sum_{i=1}^n \sum_{a \neq z_i^*} \E A_{i,a} \indic{\mathcal{F} \cap \mathcal{H}_G}+ \frac{k}{n}\sum_{i=1}^n \sum_{l=r+1}^k \E C_{i,l} \indic{\mathcal{F} \cap \mathcal{T}}\\
& \leq k \ebr{-\frac{1}{8}\br{1 - c_3 \rho'' - \frac{c_3k^2\rho\beta^{-\frac{1}{2}} \eta }{\Delta} - \frac{c_3k^\frac{7}{2}\rho'^2}{\rho}-\frac{2\sqrt{k}}{\Delta}}^2\Delta^2}+2k^2e^{-\Delta^2+k}, \\
&\quad + 2 k^2  \ebr{-\frac{1}{2} \left ( c_4 \frac{\rho''\Delta}{k^\frac{7}{2}\rho^2 \beta^{-\frac{1}{2}}\eta^2 }  \sqrt{\frac{n-k}{3n}} \right )^2 \Delta^2} .
\end{align*} 
Since, by assumption, 
$\frac{\Delta}{k^{10.5}\beta^{-0.5} \eta^2 \left ({\frac{n-k}{n}} \right )^{0.5}} \rightarrow \infty $, recalling that $\eta=\sqrt{1+p/n}$ and denoting $\lambda=\left ({\frac{n-k}{n}} \right )^{-0.5}$
we can choose
\begin{align*}
\rho = \frac{k^\frac{7}{2}}{8c_3} \br{\frac{\Delta}{k^{10.5}\beta^{-0.5} \eta^2 \lambda  }}^{0.3}, \quad \rho' = \frac{1}{8c_3}\br{\frac{\Delta}{k^{10.5}\beta^{-0.5} \eta^2 \lambda }}^{0.1},\text{ and }\rho'' = \frac{1}{8c_3}\br{\frac{\Delta}{k^{10.5}\beta^{-0.5} \eta^2\lambda  }}^{-0.1},
\end{align*}
to obtain that 
\begin{align*}
\E \ell(\hat z,z^*) \indic{\mathcal{F} \cap \mathcal{H}_G\cap \mathcal{T}} \leq n\ebr{ - \br{1- \frac{1}{2} \br{\frac{\Delta}{k^{10.5}\beta^{-0.5} \eta^2 \lambda  }}^{-0.1}} \frac{\Delta^2}{8}} .
\end{align*}
Applying Markov's inequality, we obtain that 
\begin{align*}
 \ell(\hat z,z^*) \indic{\mathcal{F} \cap \mathcal{H}_G\cap \mathcal{T}}   \leq \ebr{ - \br{1-  \br{\frac{\Delta}{k^{10.5}\beta^{-0.5} \eta^2 \lambda  }}^{-0.1}} \frac{\Delta^2}{8}} ,
\end{align*}
with probability at least $1-\ebr{-\Delta}$.  Finally, the proof is completed by using a union bound accounting for the events $\mathcal{F},\mathcal{H}_G$ and $\mathcal{T}$. \\


\textbf{Acknowledgments.} We would like to thank  Zhou Fan from Yale University for pointing out the references \cite{paul2007asymptotics,johnstone2018pca}. We are further grateful to three anonymous referees and an associate editor for careful reading of the manuscript and their valuable remarks and suggestions.

\begin{supplement}
\sname{Supplement A}\label{suppA}
\stitle{Supplement to ``Optimality of Spectral Clustering in the Gaussian Mixture Model''}
\slink[url]{url to be specified}
\sdescription{In the supplement \cite{supplement}, we first present some propositions that characterize the population quantities in Appendix \ref{subsec:support_population}. Then  in Appendix \ref{subsec:auxiliary_lemma}, we give several auxiliary lemmas related to the noise matrix $E$.
In Appendix \ref{subsec:proof_important_lemma}, we include proofs of Lemma \ref{lem:equivalence}, Lemma \ref{lem:theta_dist} and Lemma \ref{lem:hat_V_normal}. We give an extension of Proposition \ref{prop:cons} in Appendix \ref{apx:prop_cons_extension} and prove Theorem \ref{thm:1_epsilon_approx} in Appendix \ref{apx:1_eps}.
The proof of Lemma \ref{lem:VV_pert} is given in Appendix \ref{apx:VV_pert}.
}
\end{supplement}

\bibliographystyle{plainnat}
\bibliography{spectral}

\begin{thebibliography}{77}
\providecommand{\natexlab}[1]{#1}
\providecommand{\url}[1]{\texttt{#1}}
\expandafter\ifx\csname urlstyle\endcsname\relax
  \providecommand{\doi}[1]{doi: #1}\else
  \providecommand{\doi}{doi: \begingroup \urlstyle{rm}\Url}\fi

\bibitem[Abbe et~al.(2020{\natexlab{a}})Abbe, Fan, and Wang]{AbbeFanWang20}
E.~Abbe, J.~Fan, and K.~Wang.
\newblock {An $\ell_p$-theory of PCA and spectral clustering}.
\newblock \emph{arxiv preprint}, 2020{\natexlab{a}}.

\bibitem[Abbe et~al.(2020{\natexlab{b}})Abbe, Fan, Wang, and
  Zhong]{abbe2017entrywise}
E.~Abbe, J.~Fan, K.~Wang, and Y.~Zhong.
\newblock Entrywise eigenvector analysis of random matrices with low expected
  rank.
\newblock \emph{Ann. Statist.}, 48\penalty0 (3):\penalty0 1452--1474,
  2020{\natexlab{b}}.

\bibitem[Alpert and Yao(1995)]{alpert1995spectral}
C.J. Alpert and S.~Yao.
\newblock Spectral partitioning: the more eigenvectors, the better.
\newblock In \emph{32nd Design Automation Conference}, pages 195--200. IEEE,
  1995.

\bibitem[Anandkumar et~al.(2014)Anandkumar, Ge, Hsu, and
  Kakade]{anandkumar2014tensor}
A.~Anandkumar, R.~Ge, D.~Hsu, and S.M. Kakade.
\newblock A tensor approach to learning mixed membership community models.
\newblock \emph{J. Mach. Learn. Res.}, 15\penalty0 (1):\penalty0 2239--2312,
  2014.

\bibitem[Bach and Jordan(2006)]{bach2006learning}
F.R. Bach and M.I. Jordan.
\newblock Learning spectral clustering, with application to speech separation.
\newblock \emph{J. Mach. Learn. Res.}, 7:\penalty0 1963--2001, 2006.

\bibitem[Balakrishnan et~al.(2011)Balakrishnan, Xu, Krishnamurthy, and
  Singh]{balakrishnan2011noise}
S.~Balakrishnan, M.~Xu, A.~Krishnamurthy, and A.~Singh.
\newblock Noise thresholds for spectral clustering.
\newblock In \emph{Advances in Neural Information Processing Systems}, pages
  954--962, 2011.

\bibitem[Bandeira and van Handel(2016)]{BandeiraVanHandel16}
A.S. Bandeira and R.~van Handel.
\newblock Sharp nonasymptotic bounds on the norm of random matrices with
  independent entries.
\newblock \emph{Ann. Probab.}, 44\penalty0 (4):\penalty0 2479--2506, 2016.

\bibitem[Belkin and Niyogi(2003)]{belkin2003laplacian}
M.~Belkin and P.~Niyogi.
\newblock Laplacian eigenmaps for dimensionality reduction and data
  representation.
\newblock \emph{Neural computation}, 15\penalty0 (6):\penalty0 1373--1396,
  2003.

\bibitem[Chaudhuri et~al.(2012)Chaudhuri, Chung, and
  Tsiatas]{chaudhuri2012spectral}
K.~Chaudhuri, F.~Chung, and A.~Tsiatas.
\newblock Spectral clustering of graphs with general degrees in the extended
  planted partition model.
\newblock In \emph{Conference on Learning Theory}, pages 35--1, 2012.

\bibitem[Chen and Yang(2020)]{ChenYang20}
X.~Chen and Y.~Yang.
\newblock Cutoff for exact recovery of gaussian mixture models.
\newblock \emph{arXiv preprint}, 2020.

\bibitem[Coja-Oghlan(2010)]{coja2010graph}
A.~Coja-Oghlan.
\newblock Graph partitioning via adaptive spectral techniques.
\newblock \emph{Combin. Probab. Comput.}, 19\penalty0 (2):\penalty0 227--284,
  2010.

\bibitem[Davison and Szarek(2001)]{DavidsonSzarek01}
K.R. Davison and S.J. Szarek.
\newblock {Local operator theory, random matrices and Banach spaces}.
\newblock In \emph{Handbook of the geometry of Banach spaces}, volume~1, pages
  317--366. North-Holland, Amsterdam, 2001.

\bibitem[Dhillon(2001)]{dhillon2001co}
I.S. Dhillon.
\newblock Co-clustering documents and words using bipartite spectral graph
  partitioning.
\newblock In \emph{Proceedings of the seventh ACM SIGKDD international
  conference on Knowledge discovery and data mining}, pages 269--274. ACM,
  2001.

\bibitem[Ding et~al.(2005)Ding, He, and Simon]{ding2005equivalence}
C.~Ding, X.~He, and H.D. Simon.
\newblock On the equivalence of nonnegative matrix factorization and spectral
  clustering.
\newblock In \emph{Proceedings of the 2005 SIAM international conference on
  data mining}, pages 606--610. SIAM, 2005.

\bibitem[Ding et~al.(2001)Ding, He, Zha, Gu, and Simon]{ding2001min}
C.H.Q. Ding, X.~He, H.~Zha, M.~Gu, and H.D. Simon.
\newblock A min-max cut algorithm for graph partitioning and data clustering.
\newblock In \emph{Proceedings 2001 IEEE International Conference on Data
  Mining}, pages 107--114. IEEE, 2001.

\bibitem[Ding(2020)]{Ding20}
X.~Ding.
\newblock High dimensional deformed rectangular matrices with applications in
  matrix denoising.
\newblock \emph{Bernoulli}, pages 387--417, 2020.

\bibitem[Donath and Hoffman(2003)]{donath2003lower}
W.E. Donath and A.J. Hoffman.
\newblock Lower bounds for the partitioning of graphs.
\newblock In \emph{Selected Papers Of Alan J Hoffman: With Commentary}, pages
  437--442. World Scientific, 2003.

\bibitem[Fei and Chen(2018)]{fei2018hidden}
Y.~Fei and Y.~Chen.
\newblock Hidden integrality of sdp relaxations for sub-gaussian mixture
  models.
\newblock In \emph{Conference On Learning Theory}, pages 1931--1965, 2018.

\bibitem[Fiedler(1973)]{fiedler1973algebraic}
M.~Fiedler.
\newblock Algebraic connectivity of graphs.
\newblock \emph{Czechoslovak mathematical journal}, 23\penalty0 (2):\penalty0
  298--305, 1973.

\bibitem[Fishkind et~al.(2013)Fishkind, Sussman, Tang, Vogelstein, and
  Priebe]{fishkind2013consistent}
D.E. Fishkind, D.L. Sussman, M.~Tang, J.T. Vogelstein, and C.E. Priebe.
\newblock Consistent adjacency-spectral partitioning for the stochastic block
  model when the model parameters are unknown.
\newblock \emph{SIAM J. Matrix Anal. Appl.}, 34\penalty0 (1):\penalty0 23--39,
  2013.

\bibitem[Furui(1989)]{furui1989unsupervised}
S.~Furui.
\newblock Unsupervised speaker adaptation based on hierarchical spectral
  clustering.
\newblock \emph{IEEE Transactions on Acoustics, Speech, and Signal Processing},
  37\penalty0 (12):\penalty0 1923--1930, 1989.

\bibitem[Gao et~al.(2018)Gao, Ma, Zhang, and Zhou]{gao2018community}
C.~Gao, Z.~Ma, A.Y. Zhang, and H.H. Zhou.
\newblock Community detection in degree-corrected block models.
\newblock \emph{Ann. Statist.}, 46\penalty0 (5):\penalty0 2153--2185, 2018.

\bibitem[Gin{\'e} and Koltchinskii(2006)]{gine2006empirical}
E.~Gin{\'e} and V.~Koltchinskii.
\newblock Empirical graph laplacian approximation of laplace--beltrami
  operators: Large sample results.
\newblock In \emph{High dimensional probability}, pages 238--259. Institute of
  Mathematical Statistics, 2006.

\bibitem[Gin{\'e} and Nickl(2016)]{GineNickl16book}
E.~Gin{\'e} and R.~Nickl.
\newblock \emph{Mathematical Foundations of infinite-Dimensional Statistical
  Methods}.
\newblock Cambridge University Press, 2016.

\bibitem[Giraud and Verzelen(2019)]{GiraudVerzelen19}
C.~Giraud and N.~Verzelen.
\newblock Partial recovery bounds for clustering with the relaxed k-means.
\newblock \emph{Mathematical Statistics and Learning}, 1\penalty0
  (3/4):\penalty0 317--374, 2019.

\bibitem[Guattery and Miller(1998)]{guattery1998quality}
S.~Guattery and G.L. Miller.
\newblock On the quality of spectral separators.
\newblock \emph{SIAM J. Matrix Anal. Appl.}, 19\penalty0 (3):\penalty0
  701--719, 1998.

\bibitem[Hall(1970)]{hall1970r}
K.M. Hall.
\newblock An r-dimensional quadratic placement algorithm.
\newblock \emph{Management science}, 17\penalty0 (3):\penalty0 219--229, 1970.

\bibitem[Han et~al.(2020)Han, Tong, and Fan]{HanTongFan20}
X.~Han, X.~Tong, and Y.~Fan.
\newblock Eigen selection in spectral clustering: a theory guided practice.
\newblock \emph{arxiv preprint}, 2020.

\bibitem[Hein(2006)]{hein2006uniform}
M.~Hein.
\newblock Uniform convergence of adaptive graph-based regularization.
\newblock In \emph{International Conference on Computational Learning Theory},
  pages 50--64. Springer, 2006.

\bibitem[Hein et~al.(2005)Hein, Audibert, and von Luxburg]{hein2005graphs}
M.~Hein, J.~Audibert, and U.~von Luxburg.
\newblock From graphs to manifolds--weak and strong pointwise consistency of
  graph laplacians.
\newblock In \emph{International Conference on Computational Learning Theory},
  pages 470--485. Springer, 2005.

\bibitem[Hendrickson and Leland(1995)]{hendrickson1995improved}
B.~Hendrickson and R.~Leland.
\newblock An improved spectral graph partitioning algorithm for mapping
  parallel computations.
\newblock \emph{SIAM J. Sci. Comput}, 16\penalty0 (2):\penalty0 452--469, 1995.

\bibitem[Horn and Johnson(2012)]{horn2012matrix}
R.A. Horn and C.R. Johnson.
\newblock \emph{Matrix analysis}.
\newblock Cambridge university press, 2012.

\bibitem[Inaba et~al.(1994)Inaba, Katoh, and Imai]{InabaKatohImai94}
M.~Inaba, N.~Katoh, and H.~Imai.
\newblock {Applications of weighted Voronoi diagrams and randomization to
  variance-based k-clustering}.
\newblock In \emph{Proceedings of 10th ACM Symposium on Computational
  Geometry}, pages 332--339, 1994.

\bibitem[Jin(2015)]{Jin15}
J.~Jin.
\newblock Fast community detection by score.
\newblock \emph{Ann. Statist.}, 43\penalty0 (1):\penalty0 57--89, 2015.

\bibitem[Johnstone and Paul(2018)]{johnstone2018pca}
I.M. Johnstone and D.~Paul.
\newblock Pca in high dimensions: An orientation.
\newblock \emph{Proceedings of the IEEE}, 106\penalty0 (8):\penalty0
  1277--1292, 2018.

\bibitem[Kannan and Vempala(2009)]{kannan2009spectral}
R.~Kannan and S.~Vempala.
\newblock Spectral algorithms.
\newblock \emph{Found. Trends Theor. Comput. Sci.}, 4\penalty0 (3--4):\penalty0
  157--288, 2009.

\bibitem[Kannan et~al.(2004)Kannan, Vempala, and Vetta]{kannan2004clusterings}
R.~Kannan, S.~Vempala, and A.~Vetta.
\newblock On clusterings: Good, bad and spectral.
\newblock \emph{J. ACM}, 51\penalty0 (3):\penalty0 497--515, 2004.

\bibitem[Kato(2013)]{kato2013perturbation}
T.~Kato.
\newblock \emph{Perturbation theory for linear operators}, volume 132.
\newblock Springer Science \& Business Media, 2013.

\bibitem[Koltchinskii and Lounici(2016)]{KoltchinskiiLouniciAAHP}
V.~Koltchinskii and K.~Lounici.
\newblock Asymptotics and concentration bounds for bilinear forms of spectral
  projectors of sample covariance.
\newblock \emph{Annales de l'Institut Henri Poincar{\'e}, Probabilit{\'e}s et
  Statistiques}, 52\penalty0 (4):\penalty0 1976--2013, 2016.

\bibitem[Koltchinskii and Xia(2016)]{KoltchinskiiXia15}
V.~Koltchinskii and D.~Xia.
\newblock Perturbation of linear forms of singular vectors under gaussian
  noise.
\newblock In \emph{High Dimensional Probability VII}, pages 397--423. Springer,
  2016.

\bibitem[Kumar and Kannan(2010)]{kumar2010clustering}
A.~Kumar and R.~Kannan.
\newblock Clustering with spectral norm and the k-means algorithm.
\newblock In \emph{2010 IEEE 51st Annual Symposium on Foundations of Computer
  Science}, pages 299--308. IEEE, 2010.

\bibitem[Kumar et~al.(2004)Kumar, Sabharwal, and Sen]{KumarSabharwalSen04}
A.~Kumar, Y.~Sabharwal, and S.~Sen.
\newblock {A Simple Linear Time (1 + $\epsilon$)-Approximation Algorithm for
  k-Means Clustering in Any Dimensions}.
\newblock In \emph{45th Annual IEEE Symposium on Foundations of Computer
  Science}, pages 454--462, 2004.

\bibitem[Laurent and Massart(2000)]{LaurentMassart00}
B.~Laurent and P.~Massart.
\newblock Adaptive estimation of a quadratic functional by model selection.
\newblock \emph{Ann. Statist.}, 28\penalty0 (5):\penalty0 1302--1338, 2000.

\bibitem[Lei and Rinaldo(2015)]{lei2015consistency}
J.~Lei and A.~Rinaldo.
\newblock Consistency of spectral clustering in stochastic block models.
\newblock \emph{Ann. Statist.}, 43\penalty0 (1):\penalty0 215--237, 2015.

\bibitem[Lloyd(1982)]{Lloyd82}
S.~Lloyd.
\newblock Least squares quantization in pcm.
\newblock \emph{IEEE Trans. Inf. Theor.}, 28\penalty0 (2):\penalty0 129--137,
  1982.

\bibitem[L\"offler et~al.(2020)L\"offler, Zhang, and Zhou]{supplement}
M.~L\"offler, A.~Y Zhang, and H.H. Zhou.
\newblock Supplement to ``optimality of spectral clustering in the gaussian
  mixture model''.
\newblock 2020.

\bibitem[Lu and Zhou(2016)]{lu2016statistical}
Y.~Lu and H.H. Zhou.
\newblock {Statistical and computational guarantees of Lloyd's algorithm and
  its variants}.
\newblock \emph{arXiv preprint}, 2016.

\bibitem[Mahajan et~al.(2009)Mahajan, Nimbhorkar, and
  Varadarajan]{MahajanNimbhorkarVaradarajan09}
M.~Mahajan, P.~Nimbhorkar, and K.~Varadarajan.
\newblock The planar k-means problem is np-hard.
\newblock \emph{International Workshop on Algorithms and Computation}, pages
  274--285, 2009.

\bibitem[McSherry(2001)]{mcsherry2001spectral}
F.~McSherry.
\newblock Spectral partitioning of random graphs.
\newblock In \emph{Proceedings 42nd IEEE Symposium on Foundations of Computer
  Science}, pages 529--537. IEEE, 2001.

\bibitem[Meila and Shi(2001)]{meila2001learning}
M.~Meila and J.~Shi.
\newblock Learning segmentation by random walks.
\newblock In \emph{Advances in neural information processing systems}, pages
  873--879, 2001.

\bibitem[Monti et~al.(2003)Monti, Tamayo, Mesirov, and
  Golub]{MontiTamayoMesirovGolub03}
S.~Monti, P.~Tamayo, J.~Mesirov, and T.~Golub.
\newblock {Consensus Clustering: A Resampling-Based Method for Class Discovery
  and Visualization of Gene Expression Microarray Data}.
\newblock \emph{Mach. Learn.}, 52:\penalty0 91--118, 2003.

\bibitem[Ndaoud(2019)]{Ndaoud19}
M.~Ndaoud.
\newblock Sharp optimal recovery in the two component gaussian mixture model.
\newblock \emph{arXiv preprint}, 2019.

\bibitem[Ng et~al.(2002)Ng, Jordan, and Weiss]{ng2002spectral}
A.Y. Ng, M.I. Jordan, and Y.~Weiss.
\newblock On spectral clustering: Analysis and an algorithm.
\newblock In \emph{Advances in neural information processing systems}, pages
  849--856, 2002.

\bibitem[O'Rourke et~al.(2018)O'Rourke, Vu, and Wang]{o2018random}
S.~O'Rourke, V.~Vu, and K.~Wang.
\newblock Random perturbation of low rank matrices: Improving classical bounds.
\newblock \emph{Linear Algebra Appl.}, 540:\penalty0 26--59, 2018.

\bibitem[Otto and Villani(2000)]{OttoVillani00}
F.~Otto and C.~Villani.
\newblock {Generalization of an inequality by Talagrand and links with the
  logarithmic Sobolev inequality}.
\newblock \emph{J. Funct. Anal.}, 173\penalty0 (2):\penalty0 361--400, 2000.

\bibitem[Pan et~al.(2010)Pan, Ni, Sun, Yang, and Chen]{pan2010cross}
S.J. Pan, X.~Ni, J.~Sun, Q.~Yang, and Z.~Chen.
\newblock Cross-domain sentiment classification via spectral feature alignment.
\newblock In \emph{Proceedings of the 19th international conference on World
  wide web}, pages 751--760. ACM, 2010.

\bibitem[Paul(2007)]{paul2007asymptotics}
D.~Paul.
\newblock Asymptotics of sample eigenstructure for a large dimensional spiked
  covariance model.
\newblock \emph{Statistica Sinica}, pages 1617--1642, 2007.

\bibitem[Peng and Wei(2007)]{PengWei07}
J.~Peng and Y.~Wei.
\newblock Approximating k-means-type clustering via semidefinite programming.
\newblock \emph{SIAM J. on Optimization}, 18\penalty0 (1):\penalty0 186--205,
  2007.

\bibitem[Qin and Rohe(2013)]{qin2013regularized}
T.~Qin and K.~Rohe.
\newblock Regularized spectral clustering under the degree-corrected stochastic
  blockmodel.
\newblock In \emph{Advances in Neural Information Processing Systems}, pages
  3120--3128, 2013.

\bibitem[Rohe et~al.(2011)Rohe, Chatterjee, and Yu]{rohe2011spectral}
K.~Rohe, S.~Chatterjee, and B.~Yu.
\newblock Spectral clustering and the high-dimensional stochastic blockmodel.
\newblock \emph{Ann. Statist.}, 39\penalty0 (4):\penalty0 1878--1915, 2011.

\bibitem[Royer(2017)]{Royer17}
M.~Royer.
\newblock Adaptive clustering through semidefinite programming.
\newblock \emph{Advances in Neural Information Processing Systems}, pages
  1795--1803, 2017.

\bibitem[Sarkar and Bickel(2015)]{sarkar2015role}
P.~Sarkar and P.J. Bickel.
\newblock Role of normalization in spectral clustering for stochastic
  blockmodels.
\newblock \emph{Ann. Statist.}, 43\penalty0 (3):\penalty0 962--990, 2015.

\bibitem[Shi and Malik(2000)]{shi2000normalized}
J.~Shi and J.~Malik.
\newblock Normalized cuts and image segmentation.
\newblock \emph{Departmental Papers (CIS)}, page 107, 2000.

\bibitem[Simon(1991)]{simon1991partitioning}
H.D. Simon.
\newblock Partitioning of unstructured problems for parallel processing.
\newblock \emph{Computing systems in engineering}, 2\penalty0 (2-3):\penalty0
  135--148, 1991.

\bibitem[Spielman and Teng(1996)]{spielman1996spectral}
D.A. Spielman and S.~Teng.
\newblock Spectral partitioning works: Planar graphs and finite element meshes.
\newblock In \emph{Proceedings of 37th Conference on Foundations of Computer
  Science}, pages 96--105. IEEE, 1996.

\bibitem[Srivastava et~al.(2020)Srivastava, Purnamrita, and
  Hanasusanto]{SrivastavaSarkarHanasusanto20}
P.R. Srivastava, S.~Purnamrita, and G.A. Hanasusanto.
\newblock {A Robust Spectral Clustering Algorithm for Sub-Gaussian Mixture
  Models with Outliers}.
\newblock \emph{arxiv preprint}, 2020.

\bibitem[Tibshirani et~al.(2001)Tibshirani, Walther, and
  Hastie]{TibshiraniWaltherHastie01}
R.~Tibshirani, G.~Walther, and T.~Hastie.
\newblock Estimating the number of clusters in a data set via the gap
  statistic.
\newblock \emph{J. R. Stat. Soc. B}, 63\penalty0 (2):\penalty0 411--423, 2001.

\bibitem[Van~Driessche and Roose(1995)]{van1995improved}
R.~Van~Driessche and D.~Roose.
\newblock An improved spectral bisection algorithm and its application to
  dynamic load balancing.
\newblock \emph{Parallel computing}, 21\penalty0 (1):\penalty0 29--48, 1995.

\bibitem[Vempala and Wang(2004)]{VempalaWang04}
S.~Vempala and G.~Wang.
\newblock A spectral algorithm for learning mixture models.
\newblock \emph{J. Comput. Syst. Sci.}, 68\penalty0 (4):\penalty0 841--860,
  2004.

\bibitem[Vershynin(2012)]{Vershynin12}
R.~Vershynin.
\newblock Introduction to the non-asymptotic analysis of random matrices.
\newblock In \emph{Compressed sensing}, pages 210--268. Cambridge Univ. Press,
  Cambridge, 2012.

\bibitem[von Luxburg(2007)]{von2007tutorial}
U.~von Luxburg.
\newblock A tutorial on spectral clustering.
\newblock \emph{Statist. Comput}, 17\penalty0 (4):\penalty0 395--416, 2007.

\bibitem[von Luxburg et~al.(2008)von Luxburg, Belkin, and
  Bousquet]{von2008consistency}
U.~von Luxburg, M.~Belkin, and O.~Bousquet.
\newblock Consistency of spectral clustering.
\newblock \emph{Ann. Statist.}, 36\penalty0 (2):\penalty0 555--586, 2008.

\bibitem[Vu(2018)]{vu2018simple}
V.~Vu.
\newblock A simple svd algorithm for finding hidden partitions.
\newblock \emph{Combin. Probab. Comput.}, 27\penalty0 (1):\penalty0 124--140,
  2018.

\bibitem[Wang(2010)]{Wang10}
J.~Wang.
\newblock Consistent selection of the number of clusters via cross-validation.
\newblock \emph{Biometrika}, 97\penalty0 (4):\penalty0 893--904, 2010.

\bibitem[Yu and Shi(2003)]{stella2003multiclass}
S.~Yu and J.~Shi.
\newblock Multiclass spectral clustering.
\newblock In \emph{Proceedings Ninth IEEE International Conference on Computer
  Vision}, pages 313--319, 2003.

\bibitem[Zhang and Zhou(2016)]{ZhangZhou16}
A.Y. Zhang and H.H. Zhou.
\newblock Minimax rates of community detection in stochastic block models.
\newblock \emph{Ann. Statist.}, 44\penalty0 (5):\penalty0 2252--2280, 2016.

\bibitem[Zhou and Amini(2019)]{zhou2019analysis}
Z.~Zhou and A.A. Amini.
\newblock Analysis of spectral clustering algorithms for community detection:
  the general bipartite setting.
\newblock \emph{J. Mach. Learn. Res.}, 20\penalty0 (47):\penalty0 1--47, 2019.

\end{thebibliography}

\newpage
\thispagestyle{empty}
\setcounter{page}{1}
\begin{center}
\MakeUppercase{\large Supplement to ``Optimality of Spectral Clustering in the Gaussian Mixture Model''}
\medskip

{BY Matthias L\"offler, Anderson Y. Zhang and Harrison H.~Zhou}
\medskip

{ETH Z\"urich, University of Pennsylvania and Yale University}
\end{center}

\appendix

\section{Characteristics of the Population Quantities} \label{subsec:support_population}
In this section, we present several propositions that characterize the population quantities defined in Section \ref{subsec:population}. We first define two matrices related to $z^*$. Let $D\in\mathr^{k\times k}$ be a diagonal matrix with
\begin{align*}
D_{j,j} = \abs{\cbr{i\in[n]:z^*_i = j}},~ j\in[k],
\end{align*}
and let $Z^*\in\cbr{0,1}^{n\times k}$ be a matrix such that
\begin{align}
Z^*_{i,j} = \indic{z^*_i = j}, ~ i\in[n],j\in[k].\label{eqn:Z_star_def}
\end{align}

\begin{proposition} \label{prop:V_form}
There exists an orthogonal matrix $W\in\mathr^{k\times k}$ such that
\begin{align*}
V = Z^* D^{-\frac{1}{2}} W.
\end{align*}
Consequently, $V_\cdi{i} = V_\cdi{j}$ for all $i,j\in[n]$ such that $z^*_i = z^*_j$ 
In addition, we have that 
\begin{align*}
\sigma_1 \geq \sqrt{\frac{\beta n}{k}} \frac{\Delta}{2}.
\end{align*}
\end{proposition}
\begin{proof}
First note that
\begin{align*}
P = \br{\theta^*_1,\ldots, \theta^*_k} Z^{*T} = \br{\theta^*_1,\ldots, \theta^*_k} D^\frac{1}{2} D^{-\frac{1}{2}}Z^{*T}= \br{\theta^*_1,\ldots, \theta^*_k} D^\frac{1}{2} \br{Z^*D^{-\frac{1}{2}}}^{T},
\end{align*}
and observe that $Z^*D^{-\frac{1}{2}}$ has orthonormal columns. Now, we decompose $\br{\theta^*_1,\ldots, \theta^*_k} D^\frac{1}{2} = U \Lambda W^T$ into its SVD. Here $W$ is some orthonomal matrix $W\in\mathr^{k\times k}$. Then we have that 
\begin{align*}
P = U \Lambda\br{Z^*D^{-\frac{1}{2}} W}^{T},
\end{align*}
with $Z^*D^{-\frac{1}{2}} W$ having orthonormal columns. Hence, we have that $\Sigma = \Lambda$ and $V = Z^*D^{-\frac{1}{2}} W$. The structure of $Z^*$ leads to the second statement presented in the proposition. Indeed, due to  (\ref{eqn:delta}), the largest singular value of $\br{\theta^*_1,\ldots, \theta^*_k}$ must be greater than $\Delta/2$. Since $\br{\theta^*_1,\ldots, \theta^*_k} D^\frac{1}{2} = U \Sigma W^T$, we obtain that 
\begin{align*}
\sigma_1 \geq \sqrt{\frac{\beta n}{k}} \frac{\Delta}{2}.
\end{align*}
\end{proof}

\begin{proposition}\label{prop:V_pop}
The matrix $V$ satisfies
\begin{align*}
\max_{i\in[n]} \|V^Te_i\| \leq \sqrt{\frac{ k}{\beta n}}.
\end{align*}
\end{proposition}
\begin{proof}
By Proposition \ref{prop:V_form} we have that
\begin{align*}
\|V^Te_i\| = \|W^T D^{-1/2} (Z^*)^Te_i\|=\| D^{-1/2} (Z^*)^Te_i\|,
\end{align*}
where we used that $W$ is orthogonal. 
Hence, we obtain that 
\begin{align*}
\max_{i\in[n]} \|V^Te_i\|  \leq  \frac{1}{\min_{j\in[k]} D^\frac{1}{2}_{j,j}} \| (Z^*)^Te_i\| 
= \sqrt{\frac{k}{\beta n}}.
\end{align*}
\end{proof}

\begin{proposition}\label{prop:iprod_theta_u}
We have that 
\begin{align*}
 \abs{\iprod{u_l}{\theta^*_j}} \leq \sigma_l\sqrt{\frac{ k}{\beta n}},~\forall {j,l\in[k]}.
\end{align*}
\end{proposition}
\begin{proof}
Since $P= U\Sigma V^T$ and $P_{\cd{i}} = \theta^*_{z^*_i},~ i\in[n]$, we have for any $u,l\in[k]$ that 
\begin{align*}
\iprod{u_l}{\theta^*_j} = \sigma_l V_{i,l},\text{ where }i\in[n]\text{ is any index such that }z^*_i =j.
\end{align*}
The proof is completed by applying Proposition \ref{prop:V_pop}.
\end{proof}

\section{Auxiliary Lemmas Related to the Noise Matrix $E$} \label{subsec:auxiliary_lemma}
In this section, we present three basic lemmas for the control of the noise term $E$ and empirical singular values and vectors, used in the proof of Theorem \ref{thm:main}. 
\begin{lemma}\label{lem:E_opnorm}
	For a random matrix $E \in\mathr^{p\times n}$ with $\cbr{E_{i,j}}\iid \mathn\br{0,1}$, define the event $\mathcal{F}=\{\|E\| \leq  \sqrt{2} (\sqrt{n}+\sqrt{p}) \}$. We have that 
	\begin{align*}
\mathbb{P} \left ( \|E\| \geq \sqrt{n}+\sqrt{p}+t\right ) \leq e^{-t^2/2}
	\end{align*}
	and particularly
	\begin{align*}
	\mathbb{P} \left ( \mathcal{F} \right ) \geq 1- e^{-0.08n}.
	\end{align*}
\end{lemma}
\begin{proof}
	By Theorem 2.13 in \cite{DavidsonSzarek01} we have that $\mathbb{E}\|E\| \leq \sqrt{n}+\sqrt{p}$. Moreover, as $\|E\|=\sup_{\|u\|=\|v\|=1} \langle u, E v \rangle$, we have by Borell's inequality (e.g. Theorem 2.2.7 in \cite{GineNickl16book}) that $\mathbb{P} \left ( \|E\| \geq \mathbb{E} \|E\| +t\right ) \leq e^{-t^2/2}$. 
	\end{proof}
Weyl’s inequality (e.g. Theorem 4.3.1 of \citep{horn2012matrix}), the fact that $X = P  + E$ and Lemma \ref{lem:E_opnorm} imply the following lemma.
\begin{lemma}\label{lem:singular_diff}
	Assume that the random event $\mathcal{F}$ holds. We have that
	\begin{align*}
	\hat \sigma_j \leq \sigma_j + \sqrt{2}(\sqrt{n}+\sqrt{p}),~~\forall j\in[k].
	\end{align*}
\end{lemma}



The last lemma included in this section is the Davis-Kahan-Wedin $sin(\Theta)$ Theorem, which characterizes the distance between empirical and population singular vector spaces. We refer readers to Theorem 21 of \citep{o2018random} for its proof.

\begin{lemma} [Davis-Kahan-Wedin $sin(\Theta)$ Theorem] \label{lem:wedin}
	Consider any rank-$s$ matrices $W,\hat  W$. Let $W = \sum_{i=1}^s \sigma_i u_i v_i^T$ be its SVD with $\sigma_1 \geq \ldots \geq \sigma_s$. Similarly, let  $\hat  W = \sum_{i=1}^s \hat  \sigma_i \hat  u_i \hat  v_i^T$ be its SVD with $\hat \sigma_1 \geq \ldots \geq \hat \sigma_s$. For any $1\leq j\leq l\leq s$, define $V= \br{v_j,\ldots, v_l}$ and $\hat  V = \br{\hat  v_j,\ldots, \hat  v_l}$. Then,
	we have that 
	\begin{align*}
	\inf_{O: \text{ orthogonal matrix}}\opnorm{\hat  V - VO} \leq \sqrt{2} \opnorm{\hat  V \hat  V^T -VV^T } \leq \frac{4\sqrt{2} \opnorm{\hat  W - W}}{\min\cbr{\sigma_{j-1} - \sigma_j, \sigma_l - \sigma_{l+1}}},
	\end{align*}
	where we denote $\sigma_0 = +\infty$ and $\sigma_{s+1}=0$.
\end{lemma}

\section{Proofs of Key Lemmas}\label{subsec:proof_important_lemma}
In this section, we provide proofs of the lemmas stated in Section \ref{sec:proof}, except for the proof of of Lemma \ref{lem:VV_pert}, which is deferred to Appendix \ref{apx:VV_pert}. Throughout this section, for any matrix $W$, we denote by $\text{span}(W)$ the space spanned by the columns of $W$.


%


\begin{proof}[Proof of Lemma \ref{lem:equivalence}]
Since  $\hat P_\cd{i} = \hat U \hat Y_\cd{i} = \br{\hat U\hat U^T}\hat U \hat Y_\cd{i}$ lies in the column space $\text{span}(\hat U)$ any $\cbr{\theta_j}_{j=1}^k$ that achieves the minimum of  (\ref{eqn:kmeans_P}) must also lie in $\text{span}(\hat U)$. In particular, we have that 
\begin{align*}
\min_{ z\in [k]^n , \cbr{ \theta_j}_{j=1}^{k} \in \mathr^{k}} \sum_{i\in[n]}\norm{\hat P_\cd{i} - \theta_{z_i}}^2 & = \min_{ z\in [k]^n , \cbr{ c_j}_{j=1}^{k} \in \mathr^{k}} \sum_{i\in[n]}\norm{\hat U \hat Y_\cd{i} - \hat U c_{z_i}}^2\\
& = \min_{ z\in [k]^n , \cbr{ c_j}_{j=1}^{k} \in \mathr^{k}} \sum_{i\in[n]}\norm{ \hat Y_\cd{i} - c_{z_i}}^2,
\end{align*}
where the last equation is due to the fact that $\hat U$ is an orthogonal matrix.  
\end{proof}

\begin{proof}[Proof of Lemma \ref{lem:theta_dist}]
Due to the fact that $\hat P$  is the best rank-$k$ approximation of $X$ in spectral norm and $P$ is also rank-$k$, we have that 
\begin{align*}
\norm{\hat P - X} \leq  \norm{P-X}=\|E\|.
\end{align*}
This, the fact that both $\hat P$ and $P$ are at most rank $k$ and the fact that we work on the event $\mathcal{F}$ imply that,
\begin{align}
\fnorm{\hat P - P} &\leq 2\sqrt{2k} \|P-X\|
 = 2\sqrt{2k} \notag 
\|E\| \\
& \leq 4\sqrt{k}(\sqrt{n}+\sqrt{p}), \label{eqn:2.1_E_op_norm}
\end{align}
where the last inequality is due to Lemma \ref{lem:E_opnorm}.
Now, denote by $\hat \Theta$ the center matrix after solving  (\ref{eqn:kmeans_P}). That is, the $i$th column of $\hat \Theta$ is $\hat \theta_{\hat z'_i}$. Since $\hat \Theta$ is the solution to the $k$-means objective, we have that
\begin{align*}
\fnorm{\hat \Theta - \hat P} \leq \fnorm{\hat  P -P}.
\end{align*}
Hence, by the triangle inequality, we obtain that 
\begin{align*}
\fnorm{\hat \Theta - P}\leq 2 \fnorm{\hat  P -P} \leq 8\sqrt{k}(\sqrt{n}+\sqrt{p}).
\end{align*}
Now, define the set $S$ as
\begin{align*}
S = \cbr{i\in[n]: \norm{\hat \theta_{\hat z'_i} - \theta^*_{z^*_i}} > \frac{\Delta }{2}}.
\end{align*}
Since $\cbr{\hat \theta_{\hat z'_i} - \theta^*_{z^*_i}}_{i\in[n]}$ are exactly the  columns of $\hat \Theta - P$, we have that
\begin{align*}
\abs{S} \leq \frac{\fnorm{\hat \Theta - P}^2}{\br{\Delta /2}^2} \leq \frac{256k\br{n+p}}{\Delta^2}. 
\end{align*}
Assuming that 
\begin{align*}
\frac{\beta \Delta^2}{ k^2\br{1+\frac{p}{n}}}\geq 512,
\end{align*}
we have that
\begin{align*}
\abs{S} \leq \frac{\beta n}{2k}.
\end{align*}
We now  show that all the data points in $S^C$ are correctly clustered. We define
\begin{align*}
C_j=\left\{i\in[n]:z^*_i=j,i\in S^C\right\},~ j\in[k].
\end{align*}
The following holds: 
\begin{itemize}
\item For each $j\in[k]$, $C_j$ cannot be empty, as $|C_j|\geq |\{i:z^*_i=j\}| - |S|>0$.
\item For each pair $j,l\in[k],j\neq l$, there cannot exist some $i\in C_j,i'\in C_l$ such that $\hat z'_i=\hat z'_{i'}$. Otherwise $\hat \theta_{\hat z'_i} = \hat \theta_{\hat z'_{i'}}$ which would imply
\begin{align*}
\norm{\theta^*_j - \theta^*_l}  & = \norm{\theta^*_{z^*_{i}} - \theta^*_{z^*_{i'}}} \\ & \leq \norm{\theta^*_{z^*_{i}} -\hat \theta_{\hat z'_i} } + \norm{\hat \theta_{\hat z'_i} -\hat \theta_{\hat z'_{i'}}} +  \norm{\hat \theta_{\hat z'_{i'}}- \theta^*_{z^*_{i'}}}  <  \Delta,
\end{align*}
contradicting  (\ref{eqn:delta}).
\end{itemize}
Since $\hat z'_i$ can only take values in $[k]$, we conclude that the sets $\{\hat z'_i:i\in C_j\}$ are disjoint for all  $j\in[k]$. That is, there exists a permutation $\phi\in\Phi$, such that
\begin{align*}
\hat z'_i = \phi(j),~ i\in C_j,~ j\in[k].
\end{align*}
This implies that $\sum_{i\in S^C}\mathbb{I}\{\hat z_i \neq \phi (z^*_i)\}=0$. Hence, we obtain that 
\begin{align*}
\ell(\hat z,z^*) \leq \abs{S} \leq  \frac{256k\br{n+p}}{\Delta^2}.
\end{align*}
When the ratio $\Delta^2/\br{k^2 \br{n+p}}$ is large enough, an immediate implication is that $\min_{j\in[k]} \abs{\cbr{i\in[n]:\hat z_i = j}} \geq \frac{\beta n}{k} - \abs{S} \geq \frac{\beta n}{2k}$.
Moreover, in this case we obtain that 
\begin{align*}
\max_{j} \norm{\hat \theta_j - \theta^*_{\phi(j)}}^2 \leq  \frac{\fnorm{\hat \Theta - P}^2}{\frac{\beta n}{k} -\abs{S}} \leq  \frac{128k^2\br{n+p}}{\beta n} 
\end{align*}
\end{proof}

\begin{proof}[Proof of Lemma \ref{lem:hat_V_normal}]
Recall that $M$ has SVD $M = U\Sigma V^T$ where $U = (u_1,\ldots, u_k)$, $V=(v_1, \dots, v_k)$ and $\Sigma =\text{diag}\{\sigma_1,\ldots, \sigma_k\}\in\mathr^{k\times k}$ with $\sigma_1 \geq \sigma_2 \geq \dots \geq \sigma_k \geq 0$. We denote
$\mathbb{S} = \cbr{x\in \tspan{I-VV^T}: \norm{x} = 1}$
to be the unit sphere in $\tspan{I-VV^T}$. We also denote $\mathcal{O}$ to be the set of all orthonormal matrices in $\mathr^{n\times n}$ and furthermore 
\begin{align*}
\mathcal{O}' = \cbr{O \in \mathcal{O} : OV=V}.
\end{align*}
Let $V_\perp$ be an orthogonal extension of $V$ such that  $(V,V_\perp)\in \mathcal{O}$.  Then for any $O\in \mathcal{O}' $, due to the  fact that $O(V,V_\perp)\in\mathcal{O}$ and  $O(V,V_\perp) = (V, O V_\perp)$, we have that $OV_\perp$ is another orthogonal extension of $V$. This implies that 
\begin{align}\label{eqn:ortho_3}
Ox \in\tspan{I-VV^T}~~\forall x\in\tspan{I-VV^T}.
\end{align}
Hence $\mathcal{O}'$ includes all rotation matrices in $\tspan{I-VV^T}$. In the following, we prove the three assertions of Lemma \ref{lem:hat_V_normal} one by one.

~\\
\emph{Assertion (1).} Recall that $\hat M=M+E = U\Sigma V^T +E$ and $\hat M = \sum_{j=1}^{p\wedge n}\hat\sigma_j \hat u_j \hat v_j^T$ and denote by $\stackrel{d}{=}$ equality in distribution. 
For any $O\in\mathcal{O}'$, since $EO^T \stackrel{d}{=} E$, we have that $\hat M O^T = \br{U\Sigma V^T +E}O^T = U\Sigma V^T + EO^T  \stackrel{d}{=} \hat M$. 
On the other hand, $\hat M O^T$ has SVD
\begin{align*}
\hat M O^T =  \sum_{j=1}^{p\wedge n}\hat\sigma_j \hat u_j \br{ O \hat v_j}^T.
\end{align*}
Hence, for any $j\in[k]$, we have that $\hat v_j \dist O \hat v_j$. 

For any $x\in\mathr^n$, we define the mapping $f:\mathr^n \rightarrow \mathbb{S}$ as $f(x) = (I-VV^T)x/ \|(I-VV^T)x\|$. Applying $f$ on both $\hat v_j$ and $O\hat v_j$, we obtain that 
\begin{align*}
 \frac{(I-VV^T)O\hat v_j}{\|(I-VV^T)O\hat v_j\|}\dist \frac{(I-VV^T)\hat v_j}{\|(I-VV^T)\hat v_j\|} .
\end{align*}
Since $\hat v_j=VV^T\hat v_j + (I - VV^T)\hat v_j$, we have $O\hat v_j = VV^T\hat v_j + O(I - VV^T)\hat v_j.$ By  (\ref{eqn:ortho_3}), we have that  $ O (I - VV^T)\hat v_j \in\tspan{I-VV^T}$. Hence, we obtain that 
\begin{align}\label{eqn:ortho_1}
& VV^TO\hat v_j  = VV^T \hat v_j  \\
& (I-VV^T)O\hat v_j = (I-VV^T)O(I - VV^T)\hat v_j = O(I - VV^T)\hat v_j.\label{eqn:ortho_2}
\end{align}
As a consequence of  (\ref{eqn:ortho_2}), we obtain that 
\begin{align}\label{eqn:ortho_4}
 O\frac{(I-VV^T)\hat v_j}{\|(I-VV^T)\hat v_j\|}= \frac{O(I-VV^T)\hat v_j}{\|O(I-VV^T)\hat v_j\|} \stackrel{d}{=}  \frac{(I-VV^T)\hat v_j}{\|(I-VV^T)\hat v_j\|}~~\forall\; O\in\mathcal{O}'.
\end{align}
In particular, $(I-VV^T)\hat v_j/\|(I-VV^T)\hat v_j\|$ is contained  in $ \mathbb{S}$ and is rotation-invariant. Hence, we obtain that $(I-VV^T)\hat v_j/\|(I-VV^T)\hat v_j\| $ is uniformly distributed on $\mathbb{S}$.

~\\
\emph{Assertion (2). }
For any $x\in\mathr^n$, we define another mapping $g:\mathr^n \rightarrow \mathr^n$ as $g(x) = ((VV^Tx)^T, ((I-VV^T)x)^T/ \|(I-VV^T)x\|)^T$. Recall that  $\hat v_j \overset{d}{=} O\hat v_j ~\forall O\in\mathcal{O}'$.  Applying $g$ on both $\hat v_j $ and $O\hat v_j$ and using  (\ref{eqn:ortho_1}), (\ref{eqn:ortho_2}) and (\ref{eqn:ortho_4}), we obtain that 
\begin{align} \label{proof lemma 3.4 III}
\begin{pmatrix} VV^T\hat v_j \\ \frac{(I-VV^T) \hat v_j}{ \|(I-VV^T) \hat v_j\|} \end{pmatrix}  \overset{d}{=}\begin{pmatrix} VV^T\hat v_j \\ O\frac{(I-VV^T) \hat v_j}{ \|(I-VV^T) \hat v_j\|} \end{pmatrix}.
\end{align}
Let $\mathcal{A}$ be a Borel subset of $\text{span}(VV^T)$ and $\mathcal{B}$ a Borel subset of $\mathbb{S}$. By \eqref{proof lemma 3.4 III} we have for any $O\in\mathcal{O}'$ that 
\begin{align*}
\mathbb{P} \left ( \frac{(I-VV^T) \hat v_j}{\|(I-VV^T) \hat v_j\|} \in \mathcal{B}  \Big | VV^T\hat v_j \in \mathcal{A} \right )=\mathbb{P}\left ( O\frac{(I-VV^T) \hat v_j}{\|(I-VV^T) \hat v_j\|} \in \mathcal{B}  \Big | VV^T\hat v_j \in \mathcal{A} \right ).
\end{align*}
Hence, we obtain that $\frac{(I-VV^T) \hat v_j}{\|(I-VV^T) \hat v_j\|} \big | VV^T\hat v_j$ is also uniformly distributed on $\mathbb{S}$,  invariant to the value of $VV^T\hat v_j$. This implies that  $\frac{(I-VV^T) \hat v_j}{\|(I-VV^T) \hat v_j\|}$ is independent of $VV^T\hat v_j$.

~\\
\emph{Assertion (3). }
Since $\|(I-VV^T)\hat v_j\|=\sqrt{1-\|VV^T\hat v_j\|^2}$ is  a function of only $VV^T\hat v_j$, this is an immediate consequence of the second assertion.

\end{proof}

\section{Extension of Proposition 2.1}\label{apx:prop_cons_extension}
In this appendix, we provide an extension of Proposition \ref{prop:cons}.

\begin{proposition}\label{prop:cons_extension}
Assume the observations $\cbr{X_i}_{i\in[n]}$ are generated as follows:
\begin{align*}
    X_i = \theta^*_{z^*_i} + \epsilon_i.
\end{align*}
Denote $E: = (\epsilon_1,\ldots,\epsilon_n)$. Assume that  $\Delta/(\beta^{-0.5}kn^{-0.5}\norm{E})\geq C$ for some large enough constant $C>0$. Then the output of Algorithm \ref{alg:main}, $\hat z$,  satisfies for another constant $C'>0$
\begin{align}
    \ell(\hat z,z^*) \leq \frac{C'k\norm{E}^2}{n\Delta^2}. \label{eqn:cons_extension}
\end{align}
In particular, if $\cbr{\epsilon_i}_{i=1}^n\iid \mathn(0,\Sigma)$ and assuming that  $\Delta/(\beta^{-0.5}k(\|{\Sigma}\|+(\text{trace}({\Sigma})+\|\Sigma\|\log(p+n))/n)^{0.5})\geq C$, we have for another constant $C''>0$ with probability at least $1-\ebr{-0.08n}$ that 
    \begin{align}
         \ell(\hat z,z^*) \leq \frac{C''k\br{\|{\Sigma}\|+(\text{trace}({\Sigma})+\|\Sigma\|\log(p+n))/n}^2}{\Delta^2}\label{eqn:cons_extension_2}. 
    \end{align}
Moreover, if  $\cbr{\epsilon_i}_{i=1}^n\iid \text{subG}(\sigma^2)$ (i.e., sub-Gaussian with variance proxy $\sigma^2$) and assuming  that $\Delta/(\beta^{-0.5}kn^{0.5}\sigma (1+p/n)^{-0.5})\geq C$, we have    with probability at least $1-\ebr{-0.08n}$ that 
    \begin{align}
        \ell(\hat z,z^*) \leq \frac{C'' k\sigma^2\br{1+\frac{p}{n}}}{\Delta^2}\label{eqn:cons_extension_3}.
    \end{align}
    
\end{proposition}
\begin{proof}
Following the proof of Proposition \ref{prop:cons} line by line (\ref{eqn:cons_extension}) immediately follows. 

To obtain (\ref{eqn:cons_extension_2}) and (\ref{eqn:cons_extension_3}), we provide upper bounds for $\norm{E}$. 

When the errors $\cbr{\epsilon_i}_{i=1}^n$ are independent Gaussians with covariance matrix $\Sigma$, we bound $\norm{E}$ by applying Corollary 3.11 in \cite{BandeiraVanHandel16}. More precisely, assume that $\Sigma$ has eigendecomposition $\Sigma=\Gamma \Lambda \Gamma^T$ with $\Lambda $ being a diagonal matrix and $\Gamma$ an orthogonal matrix. Denote by $\tilde E$ a $p\times n$ matrix with i.i.d. standard Gaussian entries. Then, by rotation invariance of isotropic Gaussian random variables, we have that  $E\overset{d}{=}\Gamma \Lambda \Gamma
^T \tilde E\overset{d}{=}\Gamma\Lambda \tilde E$. Hence,  $\|E\| \overset{d}{=} \| \Gamma \Lambda \tilde E\| \leq \|\Lambda \tilde E\|$. The entries of $\Lambda\tilde E$ are independent and hence we can now apply Corollary 3.11 in \cite{BandeiraVanHandel16} and (\ref{eqn:cons_extension_2}) follows. 

When the errors are sub-Gaussian distributed, we bound $\norm{E}$ by a net argument, see for instance Theorem 5.39 in \cite{Vershynin12}. 
\end{proof}

\section{Proof of Theorem 2.2} \label{apx:1_eps}
To prove Theorem \ref{thm:1_epsilon_approx}, we first note that \begin{align*} \sum_{i\in[n]}  \|\hat Y_\cd{i} - \tilde c_{\tilde z_i}\|^2 &  \leq \sum_{i\in[n]}  \|\hat Y_\cd{i} - \check c_{\check z_i}\|^2 \\ & \leq (1+\varepsilon) \inf_{ \{ c_j\}_{j=1}^k \in \mathbb{R}^k} \sum_{i\in[n]} \min_{j \in [k]} \|\hat Y_\cd{i} - c_{j}\|^2, \end{align*} as each iteration of Lloyd's algorithm is guaranteed to not increase the value of the objective function. By the same analysis as for  $\hat z$ in Proposition \ref{prop:cons} , $\tilde z$ and $\check z$  satisfy  (\ref{eqn:Lemma_4.2_1}) with an additional factor of $(1+\varepsilon)$ on the right hand side of the inequality and  the centres $\{\tilde \theta_j \}_{j=1}
^k=\{ \hat U\tilde c_j \}_{j=1}^k $ satisfy (\ref{eqn:theta_dist}) with an additional factor of $\sqrt{1+\varepsilon}$ on the right hand side of the inequality. 
Then the exponential bound \eqref{thm:1_eps_exp} follows similarly as the proof of Theorem \ref{thm:main} and we line out the necessary modifications below. 

In particular, the local optimality guarantee $\|\hat Y_i-\tilde c_{\tilde z_i} \| \leq \| \hat Y_i-\tilde c_j\|, \forall i\in[n], j \neq \tilde z_i$  ensures that the equality in  \eqref{eq:loc_opt} holds.  Moreover, by definition of the centres $\tilde \theta_j$ we have, similarly as in  \eqref{eq:thetaex}, that  \begin{align*} 
\tilde \theta_j = \frac{\sum_{\check z_i = j} \hat P_\cd{i}}{\sum_{\check z_i = j} 1}= \frac{\hat \sigma_l}{\sqrt{\abs{\cbr{i\in[n]:\check z_i =j}}}}
\end{align*} 
and that 
\begin{align*}
    |\langle \hat u_l, \tilde \theta_j \rangle | \leq  \frac{\hat \sigma_l}{\sqrt{\abs{\cbr{i\in[n]:\check z_i =j}}}}. 
\end{align*}
Finally, since $\check z$ fulfills \eqref{eqn:Lemma_4.2_1} with an additional factor of $(1+\varepsilon)$ on the right hand side, we further have that  $\abs{\cbr{i\in[n]:\check z_i = j}} \geq \frac{\beta n}{2k}$ and  thus we obtain, as in  \eqref{eqn:iprod_hat_u_hat_theta}, that  
\begin{align*}
\max_{j\in[k]}\max_{r+1\leq l\leq k} \abs{\iprod{\hat u_l}{\tilde \theta_j} } \mathbb{I}(\mathcal{F})\leq   \br{k\rho + 4} \sqrt{\frac{2k}{\beta}\br{1+\frac{p}{n}}}.
\end{align*}
With these modifications, the rest of the proof is the same as the proof of Theorem \ref{thm:main}. 
\section{Spectral Projection Matrix Perturbation Theory}\label{apx:VV_pert}

In this section, we give the proof of Lemma \ref{lem:VV_pert}. Before that, we first introduce two lemmas used in the proof of Lemma \ref{lem:VV_pert}.

The following lemma gives an upper bound on the operator norm of $\norm{S_{a:b}}$. The setting considered here is slightly more general than that in Lemma \ref{lem:VV_pert}, as $E$ is not necessarily a Gaussian noise matrix. The proof of  Lemma \ref{lem:VV_pert_1} mainly follows that of Lemma 2 in \citep{KoltchinskiiLouniciAAHP}. It is included in the later part of this section for completeness.

\begin{lemma}\label{lem:VV_pert_1}
Consider any  rank-$k$ matrix $M\in\mathr^{p\times n}$ with SVD $M=\sum_{j=1}^k \sigma_j u_j v_j^T$ where $\sigma_1 \geq \sigma_2 \ldots \geq \sigma_k >0$. Define $\sigma_0 = \sigma_{k+1} =0$.

Consider any matrix $E\in\mathr^{p\times n}$. Define $\hat  M = M + E$. Let the SVD of $\hat  M$ be $\sum_{j=1}^{p\wedge n} \hat  \sigma_j \hat u_j \hat v_j^T$ where $\hat \sigma_1 \geq \hat \sigma_2\geq \ldots \geq \hat \sigma_{p\wedge n}$. 

For any two indexes $a,b$ such that $1\leq a\leq b\leq k$, define $V_{a:b}=\br{v_a,\ldots,v_b}$,   $\hat V_{a:b}=\br{\hat v_a,\ldots,\hat v_b}$ and $V:=\br{v_1,\ldots, v_k}$. Define the singular value gap $g_{a:b}= \min\cbr{\sigma_{a-1} -\sigma_a , \sigma_b -\sigma_{b+1}}$. Define
\begin{align}\label{eqn:appendix_S_def}
S_{a:b} = \br{I-VV^T}\br{\hat V_{a:b} \hat V_{a:b}^T - V_{a:b}V_{a:b}^T} V_{a:b} - \sum_{a\leq j\leq b}\frac{1}{\sigma_j}\br{I-VV^T}E^Tu_jv_j^TV_{a:b}.
\end{align}
Then, we have that 
\begin{align*}
\norm{S_{a:b}} \leq  \left(\frac{32 (\sigma_a-\sigma_b)}{\pi g_{a:b}}+ 16 \right ) \frac{\|E\|^2}{g_{a:b}^2}.
\end{align*}
\end{lemma}
$S_{a:b}$ in Lemma \ref{lem:VV_pert} and Lemma \ref{lem:VV_pert_1} depends on $E$. It can be written as $S_{a:b}(E)$ with $S_{a:b}\br{\cdot}$ treated as a function of the noise matrix. Lemma \ref{lem:VV_pert_2} studies the Lipschitz continuity of $S_{a:b}(\cdot)$. It slightly generalizes Lemma 2.4 in \citep{KoltchinskiiXia15} and its proof follows along the same arguments. Its proof will be given in the later part of this section for completeness.

\begin{lemma}\label{lem:VV_pert_2}
Consider the same setting as in Lemma \ref{lem:VV_pert_1}. Define $S_{a:b}(E)$ as in  (\ref{eqn:appendix_S_def}).
Consider another matrix $E'\in\mathr^{p\times n}$ and define  $\hat M':=M + E'$. Define  $S_{a:b}(E')$ analogously. Assuming that $\max\cbr{\norm{E},\norm{E'}}\leq g_{a:b}/4$, we have that 
	\begin{align}
\|S_{a:b}(E)-S_{a:b}(E')\| \leq 1024\left (1+\frac{\sigma_a-\sigma_b}{g_{a:b}} \right )\frac{\max \cbr{\|E\|, \|E'\| }}{g_{a:b}^2} \|E-E'\|. 
	\end{align}
\end{lemma}

Applying Lemma \ref{lem:VV_pert_1} and Lemma \ref{lem:VV_pert_2}, we are able to prove Lemma \ref{lem:VV_pert}. It generalizes Theorem 1.1 in \citep{KoltchinskiiXia15}, and its proof follows the same argument.

\begin{proof}[Proof of Lemma \ref{lem:VV_pert}]
Define $\phi$  as follows
 \begin{equation*}\phi(s)= \begin{cases} 1, ~~~~~~~~~~~~~s \leq 1 \\ 3-2s,~~~~~~ 1< s < 3/2 \\
0, ~~~~~ ~~~~~~~s \geq 3/2
\end{cases}
\end{equation*}
and note that $\phi$ is Lipschitz with Lipschitz constant $2$. As we mention earlier in this section, we write $S_{a:b}(E)$ and treat $S_{a:b}(\cdot)$ as a matrix valued  function.

\paragraph{Step 1}
Define a function
\begin{align*}
h_\delta(E)=\langle S_{a:b}(E), W \rangle \phi \left  (\frac{6\|E\|}{\delta} \right ). 
\end{align*}
We are going to show that $h_\delta$ is also Lipschitz for any $\delta \leq g_{a:b}/4$. We use the notation $\norm{\cdot}_*$ for the nuclear norm of a  matrix. 
\begin{itemize}
\item First suppose that $\max \cbr{\|E\|, \|E'\|}  \leq  \delta$. 
Then, by Lemma \ref{lem:VV_pert_1}, Lemma \ref{lem:VV_pert_2} and the fact that  $\phi$ is Lipschitz, we obtain that 
\begin{align*}
& |h_\delta(E)-h_\delta(E')| \\
& \leq \left | \langle S_{a:b}(E) - S_{a:b}(E'), W \rangle \right |  \phi \left  (\frac{6\|E\|}{\delta} \right ) + \left | \iprod{S_{a:b}(E')}{W} \right | \left | {\phi \left  (\frac{6\|E\|}{\delta} \right ) - \phi \left  (\frac{6\|E'\|}{\delta} \right )} \right | \\
 &\leq \|S_{a:b}(E)-S_{a:b}(E')\| \|W\|_* \phi\left  (\frac{6\|E\|}{\delta} \right ) + \|S_{a:b}(E')\| \|W\|_* \left | \phi \left  (\frac{6\|E\|}{\delta} \right ) -\phi \left  (\frac{6\|E'\|}{\delta} \right )\right |\\ 
& \leq 1024  \left (1+\frac{\sigma_a-\sigma_b}{g_{a:b}} \right )\frac{\max \cbr{\|E\|, \|E'\| }}{g_{a:b}^2} \|E-E'\| \|W\|_* \\
&\quad  + 16\left  (1+\frac{\sigma_a-\sigma_b}{g_{a:b}} \right )\frac{\|E'\|^2 }{g_{a:b}^2}  \|W\|_*\frac{12 \abs{\|E\| -\|E'\|}}{\delta} \\ 
& \leq    C_1\left (1+\frac{\sigma_a-\sigma_b}{g_{a:b}} \right )\frac{\delta}{g_{a:b}^2} \|E-E'\| \|W\|_*,
\end{align*}
for some constant $C_1 >0$ that is independent of $E,E'$.
\item If $\min \cbr{\|E\|, \|E'\|} \geq  \delta$ then $h(E)=h(E')=0$ and the above inequality trivially holds.
\item Finally, if $\|E \| < \delta \leq \|E'\|$, by a similar argument as above, we obtain that 
\begin{align*}
|h_\delta(E)-h_\delta(E')| & = |h_\delta(E)| = \abs{ \langle S_{a:b}(E), W \rangle \left ( \phi \left  (\frac{6\|E\|}{\delta} \right ) -\phi \left  (\frac{6\|E'\|}{\delta} \right )\right ) } \\ 
& \leq  \|S_{a:b}(E)\| \|W\|_* \abs{\phi \left  (\frac{6\|E\|}{\delta} \right ) -\phi \left  (\frac{6\|E'\|}{\delta} \right )}\\
& \leq C_1\left (1+\frac{\sigma_a-\sigma_b}{g_{a:b}} \right )\frac{\delta}{g_{a:b}^2} \|E-E'\| \|W\|_*,
\end{align*}
and the same bound holds if we switch the places of $E$ and $E'$ in the last case. 
\end{itemize}
Combining the above cases together, we have shown that for any $\delta$ such that $\delta \leq g_{a:b}/4$, $h_\delta$ is a Lipschitz function with Lipschitz constant  bounded by
\begin{align*}
C_1\left (1+\frac{\sigma_a-\sigma_b}{g_{a:b}} \right )\frac{ \delta }{g_{a:b}^2}  \|W\|_*. 
\end{align*} 

\paragraph{Step 2} In the following, for any two sequences $\cbr{x_n},\cbr{y_n}$, we adopt the notation $x_n \lesssim y_n$ meaning there exists some constant $c>0$ independent of $n$, such that $x_n \leq cy_n$.
	
By lemma \ref{lem:E_opnorm}, we have that for all $t>0$, 
	\begin{align*}
	\pbr{\abs{\norm{E} - \E \norm{E}} \geq \sqrt{2 t} } \leq \ebr{-t}.
	\end{align*}
	Set $\delta=\delta(t)=\mathbb{E}\|E\|+\sqrt{2 t}$. We consider the following two scenarios depending on the values of $t$. 
	\begin{itemize}
	\item We first consider  the case where $\sqrt{2 t}\leq g_{a:b}/24$, which implies $\delta(t) \leq g_{a:b}/6$. By the definition of $h_\delta (\cdot)$, we have that $h_\delta(E) = \iprod{S_{a:b}(E)}{W}$. Denoting by $m$ the median of $\langle S_{a:b}(E), W \rangle$  we have that
\begin{align*}
\mathbb{P} \left ( h_\delta(E)\geq m\right ) &  \geq \mathbb{P} \left ( h_\delta(E)\geq m, \|E\| \leq \delta(t) \right ) =  \mathbb{P} \left (  \iprod{S_{a:b}(E)}{W}\geq m, \|E\| \leq \delta(t) \right ) \\ 
& \geq \mathbb{P} ( \langle S_{a:b}, W \rangle \geq m)-\mathbb{P}(\|E\| > \delta(t)) \geq \frac{1}{2}-\frac{1}{2}e^{-t} \geq \frac{1}{4},
\end{align*}
and likewise $\mathbb{P}\left ( h_\delta(E) \leq m  \right ) \geq 1/4$. Hence, since $h_\delta$ is Lipschitz, we can apply Lemma 2.6 in \cite{KoltchinskiiXia15}, which is a corollary to the the Gaussian isoperimetric inequality, to show that with probability at least $1-e^{-t}$ that 
\begin{align}
|h_\delta(E)-m| & \lesssim  \sqrt{t} \left (1+\frac{\sigma_a-\sigma_b}{g_{a:b}} \right )\frac{\delta(t)}{g_{a:b}^2}  \|W\|_* .\label{eqn:isoperimetric}
\end{align}
By Lemma \ref{lem:E_opnorm}, we have that $\E \norm{E}\lesssim \sqrt{n+p}$. Thus, we obtain that 
\begin{align*}
|h_\delta(E)-m|  \lesssim \left (1+\frac{\sigma_a-\sigma_b}{g_{a:b}} \right ) \frac{\sqrt{t}}{g_{a:b}} \left ( \frac{\sqrt{n+p}+\sqrt{t}}{g_{a:b}}\right ) \|W\|_*.
\end{align*}
Moreover, the event where $\|E\| \leq \delta(t)$ occurs with probability at least $1-e^{-t}$ and on this event $h_\delta$ coincides with $\langle S_{a:b}(E), W \rangle$. Hence, with probability at least $1-2{e}^{-t}$
\begin{align}
|\langle S_{a:b}(E), W \rangle - m | \lesssim \left (1+\frac{\sigma_a-\sigma_b}{g_{a:b}} \right ) \frac{\sqrt{t}}{g_{a:b}} \left ( \frac{\sqrt{n+p}+\sqrt{t}}{g_{a:b}}\right ) \|W\|_*. \label{Isoperimetric}
\end{align}

\item We  need to prove a similar inequality in  the case  $\sqrt{2t}> g_{a:b}/24$. In this case we have that $\E\|E\| \lesssim \sqrt{t}$ as by assumption $\E\|E\| \leq g_{a:b}/8$. Hence, applying lemma \ref{lem:VV_pert_1}, we have that
\begin{align}\label{eqn:S_W_large_t_1}
|\langle S_{a:b}(E), W\rangle| \leq \norm{S_{a:b}(E)}\norm{W}_* \lesssim \left (1+\frac{\sigma_a-\sigma_b}{g_{a:b}} \right ) \frac{t}{g_{a:b}^2} \|W\|_*.
\end{align}
Hence, since $t \geq \log(4)$ and $e^{-t}\leq 1/4$, we conclude that we can bound
\begin{align}\label{eqn:S_W_large_t_2}
|m| \lesssim \left (1+\frac{\sigma_a-\sigma_b}{g_{a:b}} \right ) \frac{t}{g_{a:b}^2} \|W\|_*.
\end{align}
 (\ref{eqn:S_W_large_t_1}) and  (\ref{eqn:S_W_large_t_2}) together immediately imply that the inequality in \eqref{Isoperimetric} also holds for $\sqrt{2t}> g_{a:b}/24$. 
\end{itemize}

So far we have proved that  (\ref{Isoperimetric}) holds for all $t > \log 4$.
Integrating out the tails in the inequality in \eqref{Isoperimetric} we obtain that
\begin{align*}
| \mathbb{E} \langle S_{a:b}(E), W \rangle - m| \leq \mathbb{E} | \langle S_{a:b}(E), W \rangle - m| \lesssim \left (1+\frac{\sigma_a-\sigma_b}{g_{a:b}} \right ) \frac{\sqrt{n+p}}{g_{a:b}^2} \|W\|_*,
\end{align*}
and hence we can substitute the median by the mean in the concentration inequality \eqref{Isoperimetric}. 

\end{proof}

The last two things left are the proofs of Lemma \ref{lem:VV_pert_1} and Lemma \ref{lem:VV_pert_2}.
\begin{proof}[Proof of Lemma \ref{lem:VV_pert_1}]
	As in the proof of Lemma \ref{lem:hat_V_normal}, we use self-adjoint dilation. As before, we define for any matrix $W$
\begin{align*}
D\br{W} =  \begin{pmatrix}
	0 & W \\ W^T & 0
	\end{pmatrix}.
\end{align*}
Since $D(M)$ is symmetric and because of its relation to $M$ it has eigendecomposition 
	\begin{align}\label{eqn:D_M_def}
	D(M)=\sum_{1 \leq |i| \leq k} \sigma_i P_i, 
	\end{align} 
	where for $i\in[k]$,
	 \begin{align}\label{eqn:Pi_def}
	  \sigma_{-i}=-\sigma_i, ~P_i=\frac{1}{2}
	\begin{pmatrix}
	u_iu_i^T & u_i v_i^T \\ v_iu_i^T & v_iv_i^T
	\end{pmatrix}, ~P_{-i}=\frac{1}{2}
	\begin{pmatrix}
	u_iu_i^T & -u_i v_i^T \\ -v_iu_i^T & v_iv_i^T
	\end{pmatrix}
	\end{align}
	Similarly, we have that 
		\begin{equation*}
	D(\hat M)=\sum_{1 \leq |i| \leq p\wedge n} \hat \sigma_i \hat P_i.,
	\end{equation*}
	where for each $i\in[k]$, $\hat \sigma_{-i}, \hat P_i$ and $\hat P_{-i}$ are defined analogously. Denote
	\begin{align}\label{eqn:P_hat_P_def}
	& P= \sum_{|i| \in \{a, \dots, b \}} P_i, \quad\text{ and }\hat P=\sum_{|i| \in \{a, \dots, b \}} \hat P_i. 
	\end{align} 
Using this notation, we have that 
\begin{align}\label{eqn:hatVV_VV_simplify}
 \br{I-VV^T}\br{\hat V_{a:b} \hat V_{a:b}^T - V_{a:b}V_{a:b}^T} V_{a:b} & =   \begin{pmatrix}
 O_{n\times p } & \br{I-VV^T}
 \end{pmatrix} \br{\hat P - P} \begin{pmatrix}
O_{n\times p } \\ V_{a:b}
 \end{pmatrix},
\end{align}	
	 where $ O_{n\times p } $  denotes a $n \times p$-matrix consisting of $0$'s. 
We divide the following part of the proof into three steps. 
	\paragraph{Step 1} In this step, we decompose $ \br{I-VV^T}\br{\hat V_{a:b} \hat V_{a:b}^T - V_{a:b}V_{a:b}^T} V_{a:b}$. 
Denote by $ [\sigma_a, \sigma_b]$ the corresponding  interval on the real axis of the complex plane $\mathbb{C}$. Define    $\gamma^+$  to be a contour $\mathbb{C}$  around the intervals $ [\sigma_a, \sigma_b]$ with distance equal to $g_{a:b}/2$, i.e., 
	\begin{align}\label{eqn:gamma_plus}
	 \gamma^+=\cbr{ \eta \in \mathbb{C}: \text{dist}(\eta, [\sigma_a, \sigma_b])=\frac{g_{a:b}}{2} },
	\end{align}
	where for any point $\eta \in \mathbb{C}$ and interval $B\in\mathbb{C}$, $\text{dist}\br{\eta, B} = \min_{\eta'\in\mathbb{B}} \norm{\eta - \eta'}$. Likewise we define $\gamma^-$ as
	\begin{align}\label{eqn:gamma_minus}
	 \gamma^-=\cbr{ \eta \in \mathbb{C}: \text{dist}(\eta, [\sigma_{-b}, \sigma_{-a}])=\frac{g_{a:b}}{2} }.
	\end{align}
This way, among the singular values of $D(M)$, only those with index in $\cbr{a,\ldots, b}$ and $\cbr{-b,\ldots, -a}$ are included in $\gamma^+$ and $\gamma^-$ respectively, and the remaining ones lie outside of the contours.
By the Riesz representation Theorem for spectral projectors (c.f. page 39 of \cite{kato2013perturbation}), we have  that 
	\begin{align} \label{riesz}
	\hat P = - \frac{1}{2 \pi i} \oint_{\gamma^+}(D(\hat M)-\eta I)^{-1} d\eta - \frac{1}{2 \pi i} \oint_{\gamma^-}(D(\hat M)-\eta I)^{-1} d\eta .
	\end{align}
	For any matrix $W$ and any $\eta\in\mathbb{C}$, define the resolvent operator 
	\begin{align*}
	R_W(\eta) =  (D(W)-\eta I)^{-1}.
	\end{align*}
Then  (\ref{riesz}) can be written as
\begin{align*}
	\hat P = - \frac{1}{2 \pi i} \oint_{\gamma^+}R_{\hat M}(\eta)  d\eta - \frac{1}{2 \pi i} \oint_{\gamma^-}R_{\hat M}(\eta)  d\eta. 
\end{align*}	
Note that $D(\hat M) = D(M) + D(E)$ and that $R_{M}(\eta) =  (D(M)-\eta I)^{-1}$. We expand $R_{\hat M}(\eta) $ into its Neumann series:
	\begin{align}
R_{\hat M}(\eta)& =(D(M)-\eta I+D(E))^{-1}=  (  (D(M)-\eta I) (I+R_M(\eta)D(E)))^{-1}   \notag \\
	&  =  (I+R_M(\eta)D(E))^{-1} R_M(\eta) =  \sum_{j=0}^\infty (-1)^{j} [R_M(\eta)D(E)]^jR_M(\eta)   \notag \\
	& = R_M(\eta)-R_M(\eta)D(E)R_M(\eta)+\sum_{j=2}^\infty (-1)^{j} [R_M(\eta)D(E)]^jR_M(\eta). \label{resolvent expansion}
	\end{align}
	Applying the Riesz representation Theorem on $P$, we have that 
	\begin{align*}
	P &=  - \frac{1}{2 \pi i} \oint_{\gamma^+}(D( M)-\eta I)^{-1} d\eta - \frac{1}{2 \pi i} \oint_{\gamma^-}(D( M)-\eta I)^{-1} d\eta \\
	& =  - \frac{1}{2 \pi i} \oint_{\gamma^+}R_{ M}(\eta)  d\eta - \frac{1}{2 \pi i} \oint_{\gamma^-}R_{M}(\eta)  d\eta.
	\end{align*}
	As a result, we have the decomposition 
	\begin{align}\label{eqn:hat_P_P_diff}
	\hat P - P = L(E)+S(E)
	\end{align}
	where $L(E)$ and $S(E)$ are operators on $E$, defined as
		\begin{align}\label{eqn:L_E_def}
	L(E)=\frac{1}{2 \pi i} \oint_{\gamma^+} R_M(\eta)D(E) R_M(\eta)d \eta + \frac{1}{2 \pi i} \oint_{\gamma^-} R_M(\eta)D(E) R_M(\eta)d \eta. 
	\end{align}
	and 
	\begin{align} \label{nonlinear term}
	S(E)&=-\frac{1}{2 \pi i} \oint_{\gamma^+} \sum_{j=2}^\infty (-1)^{j} [R_M(\eta)D(E)]^jR_M(\eta)d \eta \nonumber \\
	&\quad  -\frac{1}{2 \pi i} \oint_{\gamma^-} \sum_{j=2}^\infty (-1)^{j} [R_M(\eta)D(E)]^jR_M(\eta)d \eta.
	\end{align}
	By  (\ref{eqn:hatVV_VV_simplify}), we have that 
	\begin{align*}
	 \br{I-VV^T}\br{\hat V_{a:b} \hat V_{a:b}^T - V_{a:b}V_{a:b}^T} V_{a:b} &=  \begin{pmatrix}
 O_{n\times p } & \br{I-VV^T}
 \end{pmatrix} L(E) \begin{pmatrix}
 O_{n\times p } \\ V_{a:b}
 \end{pmatrix} \\
 &\quad + \begin{pmatrix}
 O_{n\times p } & \br{I-VV^T}
 \end{pmatrix} S(E) \begin{pmatrix}
 O_{n\times p } \\ V_{a:b}
 \end{pmatrix}.
	\end{align*}

	\paragraph{Step 2}
	In the following, we show that the first term on the right hand side of the above formula equals$ \sum_{j}\frac{1}{\sigma_j}\br{I-VV^T}E^Tu_je_j^T$, which implies that the second term equals $S_{a:b}$.
	
	Define
	\begin{align}\label{eqn:L_a_b_def}
	L_{a:b} =  \begin{pmatrix}
 O_{n\times p } & \br{I-VV^T}
 \end{pmatrix} L(E) \begin{pmatrix}
 O_{n\times p } \\ V_{a:b}
 \end{pmatrix}.
	\end{align}
We first simplify $L(E)$. Recalling  (\ref{eqn:Pi_def}), for any $i$ such that $\abs{i}\leq k$, we have that  $P_i= \theta_i\theta_i^T$, where $\theta_i =\frac{1}{\sqrt{2}} (u_i^T,v_i^T)^T,\theta_{-i}=\frac{1}{\sqrt{2}}(u_i^T,-v_i^T)^T$. We expand this into an orthonormal basis of $\mathbb{R}^{p+n}$,  $\cbr{\theta_i,\theta_{-i}}_{i\in[k]}\cup \cbr{\theta_j}_{k+1 \leq j \leq p+n-k}$. This implies the following: 
	\begin{itemize}
	\item For $k+1 \leq j \leq p+n-k$, we define $P_j = \theta_j\theta_j^T$ and decomlpose the identity matrix as 
	\begin{align*}
	I = \sum_{i\in \cbr{1,\ldots,p+n-k}\cup\cbr{-k,\ldots, -1}} P_i 
\end{align*}	 
  In the rest of the proof, by default we treat $\cbr{1,\ldots,p+n-k}\cup\cbr{-k,\ldots, -1}$ to be the whole set for the index $i$. We drop it when there is no ambiguity. For instance, the above equation can be simply written as $I = \sum_i P_i$.
  
  \item We define
 \begin{align}\label{eqn:sigma_j_k_0}
 \sigma_j = 0,\forall k+1 \leq j \leq p+n-k.
\end{align}  
Then  (\ref{eqn:D_M_def}) can be expressed as 
  \begin{align*}
  D(M) = \sum_i \sigma_i P_i.
  \end{align*}
  
  \item  For $k+1 \leq j \leq p+n-k$, $\theta_j$ is orthogonal to $\theta_i - \theta_{-i},\forall i\in[k]$. This implies that the second part of $\theta_j$ (i.e., from the $(p+1)$th coordinate to the $(p+n)$th coordinate) is $0$, or orthogonal to $\text{span}(v_1,\ldots, v_k)$. Thus,
  \begin{align}\label{eqn:P_i_0}
  \begin{pmatrix}
 O_{n\times p } & \br{I-VV^T}
 \end{pmatrix} P_i  = O,\forall i\text{ s.t. }\abs{i}\leq k,
  \end{align} 
  \begin{align}\label{eqn:P_i_0_1}
  P_i \begin{pmatrix}
 O_{n\times p } \\ V_{a:b}
 \end{pmatrix} = O,\forall i\text{ s.t. }\abs{i}\notin \cbr{a,\ldots, b}.
  \end{align}
  and
  \begin{align}\label{eqn:P_i_0_2}
   \begin{pmatrix}
 O_{n\times p } & \br{I-VV^T}
 \end{pmatrix} \sum_{i >k}  P_i \begin{pmatrix}
 O_{p\times n } \\ O_{I_{n\times n}} 
 \end{pmatrix} = I- VV^T.
  \end{align}
	\end{itemize}
	By  (\ref{eqn:D_M_def}), we have that 
	\begin{align}\label{eqn:R_m_eta_def}
	R_{M}(\eta) & =  (D(M)-\eta I)^{-1}  =  \br{\sum_{i}\sigma_i P_i - \eta I}^{-1}  =  \br{\sum_{i}\br{\sigma_i-\eta} P_i }^{-1} \nonumber\\
	&= \sum_{i} \frac{1}{\sigma_i - \eta} P_i 
	 =  \sum_{i\in \cbr{a,\ldots, b}} \frac{1}{\sigma_i - \eta} P_i +  \sum_{i\notin \cbr{a,\ldots, b}} \frac{1}{\sigma_i - \eta}P_i,
	\end{align}
	defined as $R_1^+ (\eta)$ and $R_2^+ (\eta)$ respectively. With this, for the first term of $L(E)$ in  (\ref{eqn:L_E_def}), we have that 
	\begin{align*}
	\frac{1}{2 \pi i} \oint_{\gamma^+} R_M(\eta)D(E) R_M(\eta)d \eta &  = \frac{1}{2 \pi i} \oint_{\gamma^+} \left (R_1^+ (\eta)+ R_2^+ (\eta) \right )D(E) \left (R_1^+ (\eta)+R_2^+ (\eta) \right )d \eta.
	\end{align*}
	Observe that by the Cauchy-Goursat Theorem,
	\begin{align*}
	& \oint_{\gamma^+} R_1^+ (\eta) D(E) R_1^+ (\eta) d \eta \\  =&  \sum_{i \in \{a, \dots, b \}}P_iD(E)P_i \oint_{\gamma^+} \frac{1}{(\sigma_i-\eta)^2}d \eta + \sum_{i \neq j, ~i,j \in \{a, \dots, b \}} P_iD(E)P_j \oint_{\gamma^+} \frac{1}{(\sigma_i-\eta)(\sigma_j-\eta)}d \eta \\=& 0,
	\end{align*}
	since there is no singularity inside $\gamma^+$. The identical result holds for  $\oint_{\gamma^+} R_2^+ (\eta) D(E) R_2^+ (\eta) d \eta$. Using the Cauchy integral formula, we obtain that 
	\begin{align*}
	& \frac{1}{2 \pi i} \oint_{\gamma^+} R_1^+ (\eta)D(E) R_2^+ (\eta) d \eta  = \sum_{i \in \{a, \dots, b\} } \sum_{j \notin \{a, \dots, b\} }\frac{1}{2\pi i} \oint_{\gamma^+} \frac{d \eta}{(\sigma_i-\eta)(\sigma_j-\eta)} P_iEP_j \\
	= & \sum_{i\in \{ a, \dots, b\} } \sum_{j \notin \{a, \dots, b\}} \frac{P_i E P_j}{\sigma_i-\sigma_j}.
	\end{align*}
	A similar result holds for $ \frac{1}{2 \pi i} \oint_{\gamma^+} R_2^+ (\eta)D(E) R_1^+ (\eta) d \eta$. Hence, we obain that 
	\begin{align*}
		\frac{1}{2 \pi i} \oint_{\gamma^+} R_M(\eta)D(E) R_M(\eta)d \eta = \sum_{i \in \{  a, \dots, b\} } \sum_{j \notin \{a, \dots, b\}} \frac{P_i D(E) P_j+P_jD(E)P_i}{\sigma_i-\sigma_j}.
	\end{align*}
	In the same manner, splitting
	 \begin{equation}
	R_M(\eta)=R_1^-(\eta)+R_2^-(\eta)\triangleq \sum_{i \in \{-b, \dots, -a\}} \frac{P_{i}}{\sigma_i-\eta}+\sum_{i \notin \{-b, \dots, -a\}} \frac{P_{i}}{\sigma_i-\eta},
	\end{equation} we also obtain that
	\begin{align}\frac{1}{2 \pi i} \oint_{\gamma^-} R_M(\eta)D(E) R_M(\eta)d \eta =  \sum_{i \in \{-b, \dots, -a\} } \sum_{j \notin \{-b, \dots, -a\}} \frac{P_{i} D(E) P_{j}+P_{j}D(E)P_{i}}{\sigma_i-\sigma_j}. \end{align}
	Hence, we have that 
	\begin{align}
	L(E)=\br{ \sum_{i \in \{  a, \dots, b\} } \sum_{j \notin \{a, \dots, b\}} +  \sum_{i \in \{-b, \dots, -a\} } \sum_{j \notin \{-b, \dots, -a\}}  } \frac{P_i D(E) P_j+P_jD(E)P_i}{\sigma_i-\sigma_j}.
	\end{align}	
	Note that for any $i$ such $\abs{i}\in \{ a, \dots, b \}$ and any $\abs{j} \notin \{a, \dots, b\}$,  (\ref{eqn:P_i_0}) and (\ref{eqn:P_i_0_1}) imply
	\begin{align*}
	 \begin{pmatrix}
 O_{n\times p } & \br{I-VV^T}
 \end{pmatrix} P_i D(E) P_j  \begin{pmatrix}
 O_{n\times p } \\ V_{a:b}
 \end{pmatrix} =0.
	\end{align*}
	Together with  (\ref{eqn:L_a_b_def}), this implies that 
	\begin{align*}
	L_{a:b} = &\br{ \sum_{i \in \{  a, \dots, b\} } \sum_{j \notin \{a, \dots, b\}} +  \sum_{i \in \{-b, \dots, -a\} } \sum_{j \notin \{-b, \dots, -a\}}  }  \begin{pmatrix}
 O_{n\times p } & \br{I-VV^T}
 \end{pmatrix} \frac{P_jD(E)P_i}{\sigma_i - \sigma_j} \begin{pmatrix}
 O_{n\times p } \\ V_{a:b}
 \end{pmatrix} \\
  = &  \br{\sum_{i \in \{  a, \dots, b\} } +  \sum_{i \in \{  -b, \dots, -a\} } } \sum_{j >k}\begin{pmatrix}
 O_{n\times p } & \br{I-VV^T}
 \end{pmatrix} \frac{P_jD(E)P_i}{\sigma_i } \begin{pmatrix}
 O_{n\times p } \\ V_{a:b}
 \end{pmatrix}.
\end{align*}		 Recall that for all $i \leq k$, $\sigma_{-i} = -\sigma_i$. This yields 
\begin{align*}
L_{a:b}  & =  \sum_{i \in \{  a, \dots, b\} } \frac{1}{\sigma_i}\begin{pmatrix}
 O_{n\times p } & \br{I-VV^T}
 \end{pmatrix} \br{\sum_{j >k} P_j}D(E)\br{P_i - P_{-i}} \begin{pmatrix}
 O_{n\times p } \\ V_{a:b}
 \end{pmatrix}\\
 & =   \sum_{i \in \{  a, \dots, b\} } \frac{1}{\sigma_i}\begin{pmatrix}
 O_{n\times p } & \br{I-VV^T}
 \end{pmatrix}\br{\sum_{j >k} P_j} \begin{pmatrix}
 O & E \\
 E^T & O
 \end{pmatrix} \begin{pmatrix}
 O & u_i v_i^T \\ v_i^T u_i  & O 
 \end{pmatrix}\begin{pmatrix}
 O_{n\times p } \\ V_{a:b}
 \end{pmatrix}\\
 &= \sum_{i \in \{  a, \dots, b\} } \frac{1}{\sigma_i}\begin{pmatrix}
 O_{n\times p } & \br{I-VV^T}
 \end{pmatrix}\br{\sum_{j >k} P_j}  \begin{pmatrix}
 O_{p\times n} \\ I_{n\times n}
 \end{pmatrix} E^T u_iv_i^TV_{a:b}\\
 & =  \sum_{i \in \{  a, \dots, b\} } \frac{1}{\sigma_i} \br{I- VV^T} E^T u_iv_i^TV_{a:b},
 \end{align*}
 where  the last equation is due to  (\ref{eqn:P_i_0_2}). This implies 
 \begin{align}\label{eqn:S_ab_explicit}
 S_{a:b} &=  \br{I-VV^T}\br{\hat V_{a:b} \hat V_{a:b}^T - V_{a:b}V_{a:b}^T} V_{a:b}  - L_{a:b} \nonumber\\
 & =  \begin{pmatrix}
 O_{n\times p } & \br{I-VV^T}
 \end{pmatrix} S(E) \begin{pmatrix}
 O_{n\times p } \\ V_{a:b}
 \end{pmatrix}.
 \end{align}
	
	\paragraph{Step 3} In the final step, we upper bound $\norm{ S_{a:b} }$ by using the  formula above. By  (\ref{eqn:R_m_eta_def}), for any $\eta\in \gamma^+$ or $\eta\in \gamma^-$, we have that  
\begin{align}\label{eqn:R_m_upper}
\norm{R_M(\eta)} \leq  \frac{2}{g_{a:b}}.
\end{align}	
Moreover, we have that 
	\begin{align*}
	|\gamma^+|=|\gamma^-|\leq 2 (\sigma_a-\sigma_b)+ \pi g_{a:b}.
	\end{align*}
Recall the definition of $S(E)$ in  (\ref{nonlinear term}). Note that  $\norm{D(E)}  = \norm{E}$. 
\begin{itemize}
\item Under the assumption that $\|E\| \leq g_{a:b}/4$, we have that 
		\begin{align}\label{eqn:S_ab_upper_bound}
	\| S_{a:b}\| & \leq \| S(E)\|  \leq \frac{|\gamma^+| + |\gamma^-|}{2 \pi } \sum_{j=2}^\infty \|R_M(\eta)\|^{j+1} \|D(E)\|^j \notag  \\ & \leq \frac{2 (\sigma_a-\sigma_b)+ \pi g_{a:b}}{\pi} \|E\|^2 \left (\frac{2}{g_{a:b}} \right )^3 \sum_{j=0}^\infty \|E\|^{j} \left (\frac{2}{g_{a:b}} \right )^{j} \notag  \\
	& \leq \left (\frac{16 (\sigma_a-\sigma_b)}{\pi g_{a:b}}+ 8 \right ) \frac{\|E\|^2}{g_{a:b}^2}  \sum_{j=0}^\infty \|E\|^{j} \left (\frac{2}{g_{a:b}} \right )^{j} \notag \\
	& \leq \left (\frac{32 (\sigma_a-\sigma_b)}{\pi g_{a:b}}+ 16 \right ) \frac{\|E\|^2}{g_{a:b}^2}.
	\end{align}
	
	\item If $\|E\|> g_{a:b}/4$, by  (\ref{eqn:S_ab_explicit}) we have that 
	\begin{align*}
	\norm{S_{a:b}} \leq \norm{\hat V_{a:b} \hat V_{a:b}^T - V_{a:b}V_{a:b}^T} +\norm{L_{a:b}}.
	\end{align*}
	The first term is bounded by $2$. By the definition  of $\norm{L_{a:b}}$, the second term can be bounded as follows 
	\begin{align*}
	\norm{L_{a:b}} & =\norm{\br{I-VV^T}E^T \br{\sum_{i\in\cbr{a,\ldots, b}} \frac{1}{\sigma_i} u_iv_i^T}V_{a:b}} \leq  \frac{\norm{E}}{\min_{i\in\cbr{a,\ldots, b}} \sigma_i} \leq  \frac{\norm{E}}{g_{a:b}}.
	\end{align*}
	Hence, we finally obtain that 
	\begin{align*}
	\norm{S_{a:b}} \leq 2  + \frac{\norm{E}}{g_{a:b}} \leq 16  \frac{\|E\|^2}{g_{a:b}^2}.
	\end{align*}
\end{itemize}

	\end{proof}

Finally, we  prove \ref{lem:VV_pert_2}.

\begin{proof}[Proof of Lemma \ref{lem:VV_pert_2}]
We follow the same decomposition and notation as in the proof of Lemma  \ref{lem:VV_pert_1}. Recall the definition of $\hat P$ and $P$ in  (\ref{eqn:P_hat_P_def}). In particular, due to  (\ref{eqn:hat_P_P_diff}), we have that 
\begin{align*}
\hat P - P = L(E) + S(E),
\end{align*}
where $L(E)$ and $S(E)$ are defined in  (\ref{eqn:L_E_def}) and  (\ref{nonlinear term}), respectively. Define $\hat P',L(E'),S(E')$ in the same manner for $M'$. Then we have that 
\begin{align*}
S(E') - S(E) = \hat P' - \hat P - \br{L(E') - L(E)}.
\end{align*}
As a consequence, due to  (\ref{eqn:S_ab_explicit}), we obtain that 
\begin{align*}
S_{a:b}(E') - S_{a:b}(E) &=   \begin{pmatrix}
 O_{n\times p } & \br{I-VV^T}
 \end{pmatrix} \br{S(E') - S(E)} \begin{pmatrix}
 O_{n\times p } \\ V_{a:b}
 \end{pmatrix} \\
 &=\begin{pmatrix}
 O_{n\times p } & \br{I-VV^T}
 \end{pmatrix} \br{\hat P' - \hat P} \begin{pmatrix}
 O_{n\times p } \\ V_{a:b}
 \end{pmatrix}  \\
 &\quad - \begin{pmatrix}
 O_{n\times p } & \br{I-VV^T}
 \end{pmatrix} \br{L(E') - L(E)} \begin{pmatrix}
 O_{n\times p } \\ V_{a:b}
 \end{pmatrix} .
\end{align*}
In the proof of Lemma \ref{lem:VV_pert_1}, we analyze the difference between $\hat P$ and $P$. By the exactly the same argument, we analyze the difference between $\hat P'$ and $\hat P$. As in  (\ref{eqn:hat_P_P_diff}), we have that 
\begin{align*}
\hat P' - \hat P = \hat L(E' - E) + \hat S(E' - E),
\end{align*}
where
\begin{align}\label{eqn:hat_L_def}
	\hat L(E' - E) & =\frac{1}{2 \pi i} \oint_{\gamma^+} R_{\hat M}(\eta)D(E' - E) R_{\hat M}(\eta)d \eta \nonumber \\
	&\quad + \frac{1}{2 \pi i} \oint_{\gamma^-} R_{\hat M}(\eta)D(E' -E) R_{\hat M}(\eta)d \eta. 
	\end{align}
	and 
	\begin{align*} 
	\hat S(E'-E)&=-\frac{1}{2 \pi i} \oint_{\gamma^+} \sum_{j=2}^\infty (-1)^{j} [R_{\hat M}(\eta)D(E' -E)]^jR_{\hat M}(\eta)d \eta \nonumber \\
	&\quad  -\frac{1}{2 \pi i} \oint_{\gamma^-} \sum_{j=2}^\infty (-1)^{j} [R_{\hat M}(\eta)D(E'-E)]^jR_{\hat M}(\eta)d \eta,
	\end{align*}
	with $\gamma^+,\gamma^-$ defined in  (\ref{eqn:gamma_plus}) and  (\ref{eqn:gamma_minus}). 
Hence, we have that 
	\begin{align*}
	S_{a:b}(E') - S_{a:b}(E) & = \begin{pmatrix}
 O_{n\times p } & \br{I-VV^T}
 \end{pmatrix} \hat S(E'-E) \begin{pmatrix}
 O_{n\times p } \\ V_{a:b}
 \end{pmatrix}  \\
 & \quad + \begin{pmatrix}
 O_{n\times p } & \br{I-VV^T}
 \end{pmatrix} \br{\hat L(E' - E) - \br{L(E') - L(E)}} \begin{pmatrix}
 O_{n\times p } \\ V_{a:b}
 \end{pmatrix} ,
	\end{align*}
	which implies
	\begin{align}\label{eqn:S_Ep_E_diff}
	\norm{S_{a:b}(E') - S_{a:b}(E)} \leq \norm{\hat S(E'-E)} + \norm{\hat L(E' - E) - \br{L(E') - L(E)}}.
	\end{align}
	We are going to establish upper bounds on the two terms individually.
	
	\paragraph{Step 1}
	We first bound the second term above. Due to  (\ref{eqn:L_E_def}),   (\ref{eqn:hat_L_def}) and the fact that $D(E'-E)=D(E')-D(E)$, we have that 
	\begin{align*}
	&\hat L(E' - E) - \br{L(E') - L(E)} \\
	& = \frac{1}{2 \pi i} \oint_{\gamma^+}  \br{R_{\hat M}(\eta)D(E' - E) R_{\hat M}(\eta)d \eta -R_{ M}(\eta)D(E' - E) R_{ M}(\eta)}d \eta \\
	&\quad + \frac{1}{2 \pi i} \oint_{\gamma^-}  \br{R_{\hat M}(\eta)D(E' - E) R_{\hat M}(\eta)d \eta -R_{ M}(\eta)D(E' - E) R_{ M}(\eta)} d \eta.
	\end{align*}
By Weyl’s inequality (Theorem 4.3.1 of \citep{horn2012matrix}), we have $\abs{\hat \sigma_i - \sigma_i}\leq \norm{E},\forall i\in[p\wedge n]$. Assuming that $\norm{E}\leq g_{a:b}/4$, the minimum distance between $\gamma^+,\gamma^-$ to the points $\cbr{(\hat \sigma_i  , 0)}$  is at least $g_{a:b}/2- \norm{E} \geq g_{a:b}/4$, for all $i\in[p\wedge n]$. Similarly as   (\ref{eqn:R_m_upper}), we obtain that 
	\begin{align*}
	\norm{R_{\hat M}(\eta)} \leq \frac{4}{g_{a:b} },\forall \eta\in \gamma^+,\gamma^-.
	\end{align*}
Hence, 	together with the fact that $\norm{D(E'-E)} = \norm{E' -E}$, we have that 
	\begin{align*}
	&\norm{  \oint_{\gamma^+}  \br{R_{\hat M}(\eta)D(E' - E) R_{\hat M}(\eta)d \eta -R_{ M}(\eta)D(E' - E) R_{ M}(\eta)}d \eta}\\
	&\leq \norm{ \oint_{\gamma^+} R_{\hat M}(\eta)D(E'-E)( R_{ \hat M}(\eta)- R_{M}(\eta) ) d \eta }  + \norm{  \oint_{\gamma^+} (R_{\hat M}(\eta)-R_{ M}(\eta))D(E'-E) R_{  M}(\eta)d \eta}\\
& \leq   \frac{8|\gamma^+|}{g_{a_b}}  \|E'-E\| \sup_{\eta \in \gamma^+} \|R_{\hat M}(\eta)-R_{M}(\eta )\|.
	\end{align*}
	Moreover, by the expansion of the resolvent into a Neumann series in \eqref{resolvent expansion}, we have that 
	\begin{align*}
	\|R_{\hat M}(\eta)-R_{M}(\eta )\| &  \leq \sum_{j=1}^\infty (\|R_{ M}(\eta) \| \|E\|)^j \|R_{ M}(\eta)\| \leq   \norm{R_{ M}(\eta)}^2 \norm{E}      \sum_{j=0}^\infty (\|R_{ M}(\eta) \| \|E\|)^j \\
	&  \leq \frac{8\norm{E}}{g_{a:b}^2},
	\end{align*}
	where the last inequality is due to  (\ref{eqn:R_m_upper}). Hence, as $|\gamma^+|\leq \pi g_{a:b}+2(\sigma_a-\sigma_b)$, we have that 
	\begin{align*}
	&\norm{  \oint_{\gamma^+}  \br{R_{\hat M}(\eta)D(E' - E) R_{\hat M}(\eta)d \eta -R_{ M}(\eta)D(E' - E) R_{ M}(\eta)}d \eta} \\
	&\leq  64\br{\pi + \frac{2(\sigma_a-\sigma_b)}{g_{a:b}}} \frac{\norm{E}\norm{E'-E}}{g_{a:b}^2}.
	\end{align*}
	The same result holds for the other integral over $\gamma^{-}$. Hence, we obtain that 
	\begin{align*}
	\norm{\hat L(E' - E) - \br{L(E') - L(E)}} \leq 64\br{1+ \frac{2(\sigma_a-\sigma_b)}{\pi g_{a:b}}} \frac{\norm{E}\norm{E'-E}}{g_{a:b}^2}.
	\end{align*}
	
	\paragraph{Step 2} For the term related to $\hat S$, we bound it analogously as in the proof of Lemma \ref{lem:VV_pert_1}. Following  (\ref{eqn:S_ab_upper_bound}),  we have that 
	\begin{align*}
	 \norm{\hat S(E'-E)} \leq 64\left (\frac{32 (\sigma_a-\sigma_b)}{\pi g_{a:b}}+ 16 \right ) \frac{\|E' - E\|^2}{g_{a:b}^2}.
	\end{align*}
	Combining the above result with  (\ref{eqn:S_Ep_E_diff}), we obtain that 
	\begin{align*}
	\norm{S_{a:b}(E') - S_{a:b}(E)} \leq 1024\left (1+\frac{\sigma_a-\sigma_b}{g_{a:b}} \right )\frac{\max \cbr{\|E\|, \|E'\| }}{g_{a:b}^2} \|E-E'\|.
	\end{align*}

\end{proof}


\end{document}